\documentclass{amsart}
\usepackage{amsfonts}
\usepackage{amscd}
\usepackage{amsbsy}
\usepackage{epsfig}
\usepackage{epic,eepic}
\usepackage{graphicx}
\usepackage[all]{xy}
\usepackage{xypic}
\usepackage{psfrag}
\usepackage{amsmath}
\usepackage{amsthm}
\usepackage{setspace}
\usepackage{hyperref}

\theoremstyle{plain}
\newtheorem{theorem}{Theorem}[section]
\newtheorem{lemma}[theorem]{Lemma}
\newtheorem{proposition}[theorem]{Proposition}
\newtheorem{corollary}[theorem]{Corollary}

\theoremstyle{definition}
\newtheorem{remark}[theorem]{Remark}

\newcommand{\MM}{\mathcal M}
\newcommand{\OM}{\Omega{\mathcal M}}
\newcommand{\BM}{\overline{\mathcal M}}
\newcommand{\BMg}{\overline{\mathcal M}_{g}}

\newcommand{\proj}{{\mathbb P}}
\newcommand{\moduli}[1][g]{{\mathcal M}_{#1}}

\newcommand{\omoduli}[1][g]{{\Omega\mathcal M}_{#1}}
\newcommand{\wtomoduli}[1][g]{{\Omega\widetilde{\mathcal M}_{#1}}}

\newcommand{\pomoduli}[1][g]{{\proj\Omega\mathcal M}_{#1}}

\newcommand{\obarmoduli}[1][g]{{\Omega\overline{\mathcal M}}_{#1}} 
\newcommand{\barmoduli}[1][g]{{\overline{\mathcal M}}_{#1}}

\newcommand{\oteich}[1][g]{{\Omega\mathcal T}_{#1}}

\newcommand{\PP}{\mathbb P}
\newcommand{\TQ}{\widetilde{\mathcal Q}}
\newcommand{\BQQ}{\overline{\mathcal Q}}

\newcommand{\ol}{\overline}

\newcommand{\CC}{\mathcal C}
\newcommand{\bC}{\mathbb C}

\newcommand{\cE}{\mathcal E}
\newcommand{\QQ}{\mathcal Q}
\newcommand{\cQ}{\mathcal Q}
\newcommand{\cX}{\mathcal X}
\newcommand{\cF}{\mathcal F}

\newcommand{\chern}{\operatorname{ch}}
\newcommand{\OO}{\mathcal O}
\newcommand{\cO}{\mathcal O}

\newcommand{\RR}{\mathbb R}
\newcommand{\ZZ}{\mathbb Z}
\newcommand{\LL}{\mathcal L}
\newcommand{\II}{\mathcal I}

\newcommand{\fun}{{\rm fun}}

\newcommand{\divisor}{\operatorname{div}}
\newcommand{\eff}{\operatorname{eff}}
\newcommand{\XX}{\mathcal X}
\newcommand{\hyp}{{\operatorname{hyp}}}
\newcommand{\even}{{\operatorname{even}}}
\newcommand{\odd}{{\operatorname{odd}}}

\newcommand{\nh}{{\operatorname{non-hyp}}}

\newcommand{\nonhyp}{{\operatorname{non-hyp}}}
\newcommand{\PGL}{\operatorname{PGL}}
\newcommand{\SL}{\operatorname{SL}}
\newcommand{\irr}{{\operatorname{irr}}}
\newcommand{\reg}{{\operatorname{reg}}}
\DeclareMathOperator{\Pic}{Pic}

\newcommand{\Teichmuller}{Teich\-m\"uller\ }

\def\be{\begin{equation}}   \def\ee{\end{equation}}     \def\bes{\begin{equation*}}    \def\ees{\end{equation*}}
\def\ba{\be\begin{aligned}} \def\ea{\end{aligned}\ee}   \def\bas{\bes\begin{aligned}}  \def\eas{\end{aligned}\ees}

\begin{document}
\bibliographystyle{halpha}
\title[Quadratic differentials in low genus]{Quadratic differentials in low genus:
Exceptional and non-varying}

\author{Dawei Chen}

\address{Department of Mathematics, Boston College, Chestnut Hill, MA 02467, USA}

\email{dawei.chen@bc.edu}

\author{Martin M\"{o}ller}

\address{Institut f\"{u}r Mathematik, Goethe-Universit\"{a}t Frankfurt, Robert-Mayer-Str. 6-8, 60325 Frankfurt am Main, Germany}

\email{moeller@math.uni-frankfurt.de}

\thanks{During the preparation of this work the first named author is partially supported by the 
NSF grant DMS-1101153 (transferred as DMS-1200329). The second named author is partially supported
by the ERC-StG 257137. }

\begin{abstract}
We give an algebraic way of distinguishing the components of the exceptional strata
of quadratic differentials in genus three and four. The complete list of these
strata is $(9, -1)$, $(6,3,-1)$, $(3,3,3, -1)$ in genus three and $(12)$, $(9,3)$,
$(6,6)$, $(6,3,3)$ and $(3,3,3,3)$ in genus four.
\par
This result is part of a more general investigation of disjointness 
of \Teichmuller curves with divisors of Brill-Noether type on the moduli space of curves.
As a result we show that for many strata of quadratic differentials in low genus the 
sum of Lyapunov exponents for the \Teichmuller geodesic flow is the same for all 
Teichm\"uller curves in that stratum. 
\end{abstract}

\maketitle

\setcounter{tocdepth}{1}
\tableofcontents

\section{Introduction}

The moduli space $\omoduli$ of abelian differentials $\omega$ on a Riemann surface $X$
decomposes into strata $\omoduli(m_1,\ldots,m_k)$ according to the 
number and multiplicity of the zeros of $\omega$.
Since the Teichm\"uller geodesic flow preserves these strata,
many problems in Teichm\"uller theory can be dealt with stratum by stratum.
Similarly, the moduli space of quadratic differentials parameterizing pairs $(X,q)$
of a Riemann surface $X$ and a quadratic differential $q$ with at most simple
poles is stratified in the same way into $\QQ(d_1,\ldots,d_n)$.
\par
Not much is known on the topology of the strata. Kontsevich and Zorich determined
in \cite{kz03} the connected components of $\omoduli(m_1,\ldots,m_k)$. Some have
hyperelliptic components, some are distinguished by the parity of a spin structure
and the others are connected. The connected components for strata of quadratic
differentials were determined by Lanneau in \cite{lanneauENS}. Some have hyperelliptic
components and besides a short list of \emph{exceptional} cases, all the other strata are connected.
\par
To find an algebraic invariant distinguishing the exceptional cases remained
an open problem. Our first main result provides a solution to this problem. 
\par
\begin{theorem}
Let $\QQ(k_1,k_2,k_3,-1)$ be one of the  strata $(9,-1)$, $(6,3,-1)$ 
or $(3,3,3,-1)$ in genus three.
Then the stratum has precisely two connected components, 
distinguished by 
$$\dim H^0(X, \divisor(q)_0/3) = 1 \quad \text{resp.} \quad 
\dim H^0(X, \divisor(q)_0/3) = 2.$$
\end{theorem}
\par
We also construct the connected components using techniques from algebraic
geometry. This provides a proof of the connectedness (and irreducibility) of the two components
that does not rely on any geometry of flat surfaces.
\par
For $g=4$ we discovered that the list of exceptional strata was
incomplete in \cite{lanneauENS}.
\par
\begin{theorem}
Let $\QQ(k_1,\ldots,k_n)$ be one of the genus four strata
$(12)$, $(9,3)$, $(6,6)$, $(6,3,3)$ or $(3,3,3,3)$.
Then the stratum has precisely two non-hyperelliptic connected components, 
distinguished by 
$$\dim H^0(X, \divisor(q)/3) = 1 \quad \text{resp.} \quad \dim H^0(X, \divisor(q)/3) = 2.$$
\end{theorem}
\par
These results were obtained in parallel with our investigation of sums of Lyapunov
exponents for Teichm\"uller curves. In this sense, the present paper
is a continuation to quadratic differentials of our paper \cite{chenmoeller}.
There, a connected component of a stratum was called {\em non-varying}, 
if for all \Teichmuller curves in this stratum the sum of Lyapunov
exponents is the same, and {\em varying} otherwise. We proved
that many strata of abelian differentials in low genus are non-varying.
\par
Here we prove that many strata of quadratic differentials in low genus 
are non-varying. The basic idea is the same. We seek a divisor
$D$ in the moduli space of pointed curves $\barmoduli[g,n]$ such that $D$ is disjoint with all \Teichmuller curves $C$ in a given stratum.
In the case of abelian differentials, \Teichmuller curves do not intersect
higher boundary divisors $\delta_i$ of the moduli space for $i>0$. 
So, roughly, Noether's formula together with $C \cdot D = 0$
are sufficient to solve for the sum of Lyapunov exponents.
\par
For quadratic differentials, however, \Teichmuller curves may intersect 
higher boundary divisors. Thus, from a relation
$$ C \cdot \Big(a\lambda + b \delta_0 + \sum c_i \delta_i\Big) = 0$$
we cannot directly deduce the ratio $(C \cdot \delta)/(C \cdot \lambda)$,  
where $\delta = \sum \delta_i$ is the total boundary, and hence it is not sufficient to conclude 
the sum of Lyapunov exponents. It requires a considerable amount of work using both algebraic geometry and flat geometry to simplify the situation in the above equation. 
\par
Our second main result is the following list of non-varying strata of quadratic differentials.  
\par
\begin{theorem}
Consider the strata of quadratic differentials in low genus. 
\par
(1) In genus one, the strata $\QQ(n,-1^n)$ and 
$\QQ(n-1,1,-1^{n})$ are non-varying for $n\geq 2$ (Theorem~\ref{thm:nvg1}). 
\par
(2) In genus two, there are $12$ non-varying strata 
among all strata of dimension up to seven (Theorem~\ref{thm:nvg2}). 
\par
(3) In genus three, there are $19$ non-varying strata 
among all non-exceptional strata of dimension up to eight (Theorem~\ref{thm:nvg3}) and $6$ non-varying strata among all exceptional strata 
(Theorem~\ref{thm:execg3NV}). 
\par
(4) In genus four, there are $8$ non-varying strata among all non-exceptional strata of dimension up to nine 
(Theorem~\ref{thm:nvg4}) and $7$ non-varying strata among all exceptional strata 
(Theorem~\ref{thm:execg4NV}).  
\end{theorem}

Furthermore in genus five, we show that even the stratum with a unique zero is varying (Appendix~\ref{sec:varying}).  
Therefore, it seems quite plausible that our list of non-varying
strata (including all the known hyperelliptic strata) is complete.
\par
The paper is organized as follows. In Section~\ref{sec:backmoduli} we provide the background on strata of abelian
and quadratic differentials. A result of independent interest 
shows that near certain boundary strata of the moduli space the period and 
plumbing parameters are coordinates of strata of quadratic differentials.
\par
In Section~\ref{sec:divclass} we recall the Picard groups of moduli 
spaces and various divisor classes. Section~\ref{sec:propTeich} discusses properties of 
Teichm\"uller curves generated by quadratic differentials near the boundary of the moduli space.
\par
In order to prove disjointness of Teichm\"uller curves with various divisors
in genus three and four along the hyperelliptic locus and the Gieseker-Petri
locus, the use of the canonical model is not sufficient. Instead, we need to use
the bicanonical model. The necessary background is provided in Section~\ref{sec:limits}. Finally, 
Sections~\ref{sec:g3exceptional} to~\ref{sec:g4nonexc} contain the discussion 
of irreducible components and non-varying strata summarized in our main results.
\par
{\bf Acknowledgments:} We thank Vincent Delecroix and Anton Zorich for
supplying us with an ample amount of computer data on which our
investigation of non-varying strata and a proof for varying strata in 
the appendix were based. We also want to thank Izzet Coskun and Joe Harris 
for helpful suggestions on the geometry of low genus curves. Moreover, 
we thank Fei Yu for the conversation that led to the results in  
Appendix~\ref{sec:filtration}. 
\par
Part of the 
work was done during the Oberwolfach Meeting  'Billiards, Flat Surfaces 
and Dynamics on Moduli Spaces', May 2011 and the Park City Summer Program 
'Moduli Spaces of Riemann Surfaces', July 2011. We would like to thank the 
organizers for invitation and hospitality.

\section{Background on moduli spaces} \label{sec:backmoduli}

\subsection{Strata and hyperelliptic loci} \label{sec:hyploci}
For $d_i \geq -1$ and $\sum_{i=1}^n d_i = 4g-4$, let $\cQ(d_1,\ldots,d_n)$ 
denote the {\em moduli space of quadratic differentials}. It parameterizes
pairs $(X,q)$ of a Riemann surface $X$ and a quadratic differential $q$ on $X$ 
that have $n$ distinct zeros or poles of order $d_1,\ldots, d_n$. 
The condition $d_i\geq -1$ ensures that the quadratic 
differentials in $\cQ(d_1,\ldots,d_n)$ have at most simple poles and
that their total flat volume is thus finite. The pairs $(X,q)$
are called {\em half-translation surfaces}. We denote by 
$\PP\cQ(d_1,\ldots,d_n) = \cQ(d_1,\ldots,d_n)/\bC^*$
the associated projectivized space.
\par
Let $\omoduli$ denote the \emph{Hodge bundle} of holomorphic one-forms
over the moduli space $\moduli$ of genus $g$ curves and let $\pomoduli$ denote the associated projective bundle. The spaces
$\omoduli$ and $\pomoduli$ are stratified
according to the zeros of one-forms. For $m_i\geq 1$ and $\sum_{i=1}^k m_i = 2g-2$, 
let $\omoduli(m_1,\ldots,m_k)$ denote the stratum parameterizing one-forms that
have $k$ distinct zeros of order $m_1,\ldots, m_k$. 
\par
Denote by $\barmoduli$ the Deligne-Mumford compactification
of $\moduli$. The Hodge bundle extends
to the boundary of $\barmoduli$, parameterizing {\em stable one-forms}
or equivalently sections of the dualizing sheaf.
We denote the total space of this extension by $\obarmoduli$.
\par
Points in $\omoduli$, called {\em flat surfaces}, are usually
written as $(X,\omega)$ for a one-form $\omega$ on $X$.
For a stable curve $X$, denote the dualizing sheaf 
by $\omega_X$. We will stick to the notation that points in
$\obarmoduli$ are given by  a pair $(X,\omega)$ with $\omega \in H^0(X,\omega_X)$.
\par
If the quadratic differential is not a global square of a one-form, 
there is a {\em canonical double covering} $\pi: Y \to X$
such that $\pi^* q = \omega^2$. This covering is ramified precisely
at the zeros of odd order of $q$ and at the poles. It gives a
map
$$\phi: \cQ(d_1,\ldots,d_s) \to \omoduli(m_1,\ldots, m_k), $$
where the signature $(m_1,\ldots, m_k)$ is determined by the ramification
type (see \cite{kz03} for more details).
\par
If the domain and the range of the map $\phi$ have the same dimension
for some signature, we call the image a {\em component of hyperelliptic flat surfaces}
of the corresponding stratum of abelian differentials. This can only happen if the domain of $\phi$ 
parameterizes genus zero curves.
More generally, if the domain of $\phi$ parameterizes
genus zero curves, we call the image a {\em locus of hyperelliptic flat surfaces}
in the corresponding stratum. These loci are often called hyperelliptic loci, 
e.g.\ in \cite{kz03} and \cite{ekz}. We prefer to reserve {\em hyperelliptic locus} for the
subset of $\moduli$ (or its closure in $\barmoduli$) parameterizing hyperelliptic curves 
and thus specify with 'flat surfaces' if we speak of subsets of $\omoduli$.
\par
Instead of taking the canonical double covering one can start with $(X,q)$
in a stratum of quadratic differentials, prescribe the topology of a double
covering with branch points contained in the set of zeros and
poles of $q$ and consider the locus of branched coverings $(Y,q_Y)$ obtained in that way.
\par
The main result of \cite{lahypcomp} states that only if $g(X)=0$ and only for the following 
three types of non-canonical double coverings the dimensions of the strata
containing $(X,q)$ resp.\ $(Y,q_Y)$ coincide.
\begin{itemize}
\item[(1)] $\QQ(2(g-k)-3, 2k+1, -1^{2g+2}) \to \QQ(2(g-k)-3, 2(g-k)-3, 2k+1, 2k+1)^{\hyp}$. 
\item[(2)] $\QQ(2(g-k)-3, 2k, -1^{2g+1}) \to \QQ(2(g-k)-3, 2(g-k)-3, 4k+2)^{\hyp}$. 
\item[(3)] $\QQ(2(g-k)-4, 2k, -1^{2g}) \to \QQ(4(g-k)-6, 4k+2)^{\hyp}$. 
\end{itemize}
\par
Consequently the images of these maps are connected components of
the corresponding strata of quadratic differentials. They will be called {\em components of 
hyperelliptic half-translation surfaces}.
\par

\subsection{Sum of Lyapunov exponents and Siegel-Veech constant} 

{\em Lyapunov exponents} measure the Hodge norm growth of cohomology classes
under parallel transport along the \Teichmuller geodesic flow. The individual
exponents are hard to calculate, but their sum is a rational number 
that can be evaluated, one \Teichmuller curve at a time. The same holds
for the partial sum over all Lyapunov exponents that belong to a local
subsystem, in case the local system with fiber $H^1(X,\RR)$ over the
\Teichmuller curve splits into several subsystems. See \cite{moelPCMI} for a survey
on these results and related references.
\par
For a Teichm\"uller curve $C$ generated by $(X,q)$ in $\QQ(d_1,\ldots,d_s)$, let $(Y,\eta)$ be
the canonical double covering. The curve $Y$ comes with an involution $\tau$.
Its cohomology splits into the $\tau$-invariant and $\tau$-anti-invariant part. 
Adapting the notation of \cite{ekz} we 
let $g = g(X)$ and $g_{\eff} = g(Y) -g$. Let $\lambda^{+}_i$ be the
Lyapunov exponents of the $\tau$-invariant part of $H^1(Y,\RR)$ and
let $\lambda^{-}_i$ be the Lyapunov exponents of the $\tau$-anti-invariant part.
The $\tau$-invariant part descends to $X$ and hence the $\lambda^{+}_i$
are the Lyapunov exponents of $(X,q)$ we are primarily interested in.
Define
\ba
L^{+} &= \lambda^{+}_1 + \cdots + \lambda^{+}_g, \\
L^{-} &= \lambda^{-}_1 + \cdots + \lambda^{-}_{g_{\eff}}. 
\ea
The role of $L^{+}$ is analogous to the ordinary sum of Lyapunov exponents 
in the case of abelian differentials. 
\par
The main result of \cite{ekz} expresses the sum of Lyapunov exponents as 
\begin{equation} \label{eq:EKZ}
\begin{aligned}
L^{+} = c + \kappa, \quad \text{where} \quad 
\kappa = \frac{1}{24}\Big(\sum_{j=1}^n \frac{d_j(d_j+4)}{d_j+2}\Big)
\end{aligned}
\end{equation}
and where $c$ is the (area) Siegel-Veech constant of $(X,q)$. We will
not give the definition of Siegel-Veech constants here but rather
note that a similar
formula holds for $Y$ and the Siegel-Veech constants of $X$ and $Y$ are
closely related. As a result, \cite{ekz} obtain the key formula
\be \label{eq:L+L-}
L^{-} - L^{+} = \frac{1}{4} \sum_{odd \ d_j} \frac{1}{d_j+2}. 
\ee 
\par
Applying this formula to the various double coverings associated to
hyperelliptic half-translation surfaces shows that the following components
are non-varying.
\par
\begin{corollary}[\cite{ekz}] \label{cor:hypNonvar} Let $C$ be a 
\Teichmuller curve in one of the components of hyperelliptic half-translation
surfaces. Then: 
\par
For type (1), we have 
$$ L^{+}= \frac{g+1}{2} - \frac{g+1}{2(2g-2k-1)(2k+3)}. $$
\par
For type (2), we have 
$$ L^{+} = \frac{2g+1}{4} - \frac{1}{4(2g-2k-1)}. $$
\par
For type (3), we have 
$$ L^{+} = \frac{g}{2}. $$
\end{corollary}
\par

\subsection{Compactification of $\cQ(d_1,\ldots,d_n)$} \label{sec:compQ}

We now describe the moduli spaces of quadratic differentials
algebraically over a compactification of the moduli
space of curves. This construction has the main feature
that the boundary objects are the stable curves that
appear as limit objects of \Teichmuller curves as we 
will see in Section~\ref{sec:Teichbd}. Since we allow simple poles, 
i.e.\ $d_i=-1$ for some $i$, the spaces $\cQ(d_1,\ldots,d_n)$
are not strata of a single vector bundle, but of several,
according to the number of poles.
\par
Given a signature $(d_1,\ldots,d_n)$, let $k$ denote 
the number of poles, i.e.\ the number of indices $i$ 
with $d_i=-1$. We can assume $d_j \geq 0$ for $1\leq j \leq n-k$ and 
$d_j = -1$ for $n-k < j \leq n$. From now on we work over the moduli space
$\barmoduli[g,k]$, the Deligne-Mumford compactification
of the moduli space of genus $g$ curves with $k$ marked points.
Over $\barmoduli[g,k]$ there is a vector bundle 
$\BQQ_k \to \barmoduli[g,k]$, whose fiber over a stable
pointed curve $(X,p_1,\ldots,p_k)$ parameterizes the sections
$$ q \in H^0(X,\omega_X^{\otimes 2}(p_1 +\cdots + p_k)).$$
\par
Let $\BQQ(d_1,\ldots,d_n)$ be the closure of the subspace of $\BQQ_k$
where the associated divisor of $q$ has zeros (different from the $p_i$) of order $d_1,\ldots,d_{n-k}$.
Thus, a point $(X,q)$ in the interior of $\BQQ(d_1,\ldots,d_n)$ corresponds to a quadratic differential of type 
$(d_1,\ldots,d_n)$ with simple poles at $p_1,\ldots,p_k$, which are smooth points of $X$. 
\par

\subsection{Period coordinates and plumbing coordinates} \label{sec:periodcoord}

Both  $\omoduli(m_1,\ldots,m_k)$ and $\QQ(d_1,\ldots,d_n)$ are known to 
be smooth and they possess a convenient coordinate system given by period
coordinates (\cite{masur82}, \cite{veech90}). To obtain local coordinates on
a neighborhood $U$ of $(X,\omega)$ in the first case, we fix a basis 
of $H_1(Y, Z(\omega),\ZZ)$, where $Z(\omega)$ denotes the locus of zeros of $\omega$. 
For $(X', \omega') \in U$, integration of $\omega'$ along this basis provides a coordinate system.
In the second case, near $(X,q)$ we start with the oriented double 
cover $\pi: Y \to X$ so that $\pi^*q = \omega^2$
for some one-form $\omega$ on $Y$. Let $\sigma$ be the involution
of $Y$ with quotient $X$ and fix a basis of $H_1(Y,Z(\omega),\ZZ)^-$, 
the $\sigma$-anti-invariant part of the relative homology of $Y$ with respect to $Z(\omega)$. 
For $(X',q')$ in a neighborhood of $(X,q)$, 
the integration of a square root of $q'$ along this fixed basis
provides the desired coordinate system.
\par 
For stable curves we need one more type of coordinates to
deform them into smooth curves, 
coming from the construction of plumbing in a cylinder 
(see e.g.\ \cite{wolpert87} or \cite{bers73} and the proof below). 
These functions will not be a coordinate system for all stable
curves, since $\omega$ might be identically zero on one component of a reducible stable curve or 
a flat surface may have a separating node where $\omega$ is holomorphic and thus
the location of the node cannot be detected by periods.
However, for certain classes of translation and half-translation
structures on stable curves the combination of the above
functions, that we call {\em period plumbing coordinates}
still forms a coordinate system.
\par
We call a stable curve together with a stable one-form $(X,\omega)$ of 
{\em polar type} if there does not exist an irreducible component
of $X$ on which $\omega$ vanishes entirely and if $\omega$ has
a pole at each of the nodes of $X$. Similarly, we call a
pair $(X,q)$ of {\em polar type} if there does not exist an irreducible component
of $X$ on which $q$ vanishes entirely and if $q$ has
a double pole at each of the nodes of $X$.
Let $\wtomoduli[g](m_1,\ldots,m_k)$ be the partial compactification
of $\omoduli[g](m_1,\ldots,m_k)$ by adding stable flat surfaces 
of polar type and let $\TQ(d_1,\ldots,d_n)$ be the partial compactification
of $\QQ(d_1,\ldots,d_n)$ by adding stable half-translation surfaces of polar type.
\par
Note that in a stratum of stable one-forms, the stable curves
of polar type cannot possess a separating node, by the residue
theorem. On the other hand, for quadratic differentials, a stable
half-translation surface of polar type may have separating nodes.
In the case of one-forms the argument of the following proposition
is due to \cite{bainbridge07}.
\par
\begin{proposition} \label{prop:periodandplumb}
The partial compactifications of the  
strata $\wtomoduli(m_1,\ldots,m_k)$  and $\TQ(d_1,\ldots,d_n)$ are smooth. 
Local coordinates are given by period plumbing coordinates.
\end{proposition}
\par
\begin{proof}
We start with the case of stable one-forms. If $(X,\omega)$
is of polar type, then the normalization of $X$ is a possibly
disconnected smooth curve and the pullback of $\omega$
is a one-form, non-zero on each of the components, with at most 
simple poles, say $r$ of them. Let $\Sigma_{\underline{g},r}$
be the topological type of this punctured disconnected surface, 
where $\underline{g}$ is the tuple of genera of the irreducible
components of $\Sigma$. Period coordinates are local coordinates on the boundary
stratum of $(X,\omega)$, or equivalently, coordinates
on the Teichm\"uller space $\oteich[\underline{g},r](m_1,\ldots,m_k)$
by an easy generalization of the argument of \cite{masur82} or \cite{veech90}.
\par
We denote the loops around the $r$ punctures of $\Sigma_{\underline{g},r}$
by $\alpha_1,\ldots,\alpha_r$. For $(X,\omega)$ as above, let the Dehn space 
$\Omega \mathcal{D}_{\underline{g},r}(m_1,\ldots,m_k)$ for 
$\Sigma_{\underline{g},r}$ be union of the quotient of the $\oteich[g](m_1,\ldots,m_k)$
by the group $\ZZ^r$ generated by the Dehn twists around $\alpha_1,\ldots,\alpha_r$
and the boundary \Teichmuller space  $\oteich[\underline{g},r](m_1,\ldots,m_k)$.
This space is given the topology and complex structures such that
the quotient map to $\obarmoduli$ is a holomorphic covering map onto its image.
\par
We first define an {\em umplumbing} map
 $$\Psi = \psi \times (z_1,\ldots,z_r): 
\Omega \mathcal{D}'_{\underline{g},r}(m_1,\ldots,m_k) \to 
\oteich[\underline{g},r](m_1,\ldots,m_k)
\times \CC^r$$
as follows, where the prime denotes the restriction to a sufficiently
small neighborhood of the locus of stable curves of polar type in  
$\oteich[\underline{g},r](m_1,\ldots,m_k)$. Since each of the $\alpha_j$ 
corresponds to a loop
around a pole, in such a neighborhood each of the curves $\alpha_j$ for $j=1,\ldots,r$
is homotopic to the core curve of a maximal flat cylinder $C_j$. For each
of them we fix a curve in $H_1(X,Z(\omega),\ZZ)$ crossing $C_j$ once
but not crossing the other cylinders. We define $\psi$ as unplumbing on 
$\Omega \mathcal{D}_{\underline{g},r}(m_1,\ldots,m_k)$, i.e.\ replacing
each of the $C_j$ by a pair of half-infinite cylinders with residue
equal to $\int_{\alpha_j} \omega$. On the boundary $\psi$ is the identity
and we define $z_j = \exp(2\pi i (\int_{\beta_j} \omega/\int_{\alpha_j} \omega))$,
which is obviously well-defined up to the Dehn twists.
\par
We claim that $\Psi$ is biholomorphic onto its image. The converse
is given by {\em plumbing}, i.e.\ for any surface of polar type in 
$\oteich[\underline{g},r](m_1,\ldots,m_k)$ we replace the pair of 
half-infinite cylinders with residue equal to $r_j$ by a cylinder
with core curve $\alpha_j$ and $\int_{\alpha_j} \omega = r_j$ such
that $\int_{\beta_j} \omega$ satisfies $z_j = \exp(2\pi i (\int_{\beta_j} \omega/r_j))$.
The function $\Psi$ is obviously holomorphic outside the boundary and
continuous on all of $\Omega \mathcal{D}'_{\underline{g},r}(m_1,\ldots,m_k)$, hence holomorphic
there. Moreover, plumbing is obviously inverse to unplumbing, thus
proving the claim.
\par
Together with generalized period coordinates, this claim on $\Psi$
shows that period and plumbing functions are indeed coordinates
on $\wtomoduli(m_1,\ldots,m_k)$.
\par
The proof for the case of half-translation surfaces is the same. Again, the
anti-invariant periods on the canonical double cover
give, by the arguments of \cite{masur82} or \cite{veech90},
coordinates along the boundary of $\TQ(d_1,\ldots,d_n)$, since $q$ is
non-zero on each irreducible component. Note that the holonomy around
each curve $\alpha_j$ around a puncture of $\Sigma_{\underline{g},r}$
is among these anti-invariant period functions since the double covering map 
is unramified near the punctures, because by hypothesis $q$ has precisely a 
double pole there. We now can define the unplumbing map $\Psi$ and
its inverse given by plumbing as above.
\end{proof}
\par

\section{Divisor classes} \label{sec:divclass}

In this section we recall the Picard group of the moduli
space of curves with marked points and collect the expression of several geometrically
defined divisors on the moduli space of curves in low genus with
few marked points in terms of the standard generators of the Picard 
group. The results are basically contained in the literature (\cite{Logan}, \cite{F1}),  
but in several cases not all boundary terms were calculated in full detail.
We will thus perform the calculation for the cases we need. 
\par
Use $\Pic(\cdot)$ to denote the rational Picard group $\Pic_{\fun}(\cdot)_{\mathbb Q}$ of a 
moduli stack (see \cite{harrismorrison} for more details).
Since the quantities we are interested in, the sum of Lyapunov exponents and
slope, are invariant under finite base change, this is the group we want to use, not the Picard group of the coarse moduli scheme.
\par
Recall the standard notation for elements in the Picard group. Let $\lambda$
denote the first Chern class of the Hodge bundle. Let $\delta_i$, $i=1,\ldots, \lfloor g/2 \rfloor$
be the boundary divisor of $\barmoduli$ whose generic element is a smooth curve of genus $i$
joined at a node to a smooth curve of genus $g-i$. The generic element of the
boundary divisor $\delta_0$ is an irreducible nodal curve of geometric genus $g-1$. In the literature sometimes 
$\delta_0$ is denoted by $\delta_{\irr}$. We write $\delta$ for the total boundary class. 
\par
For moduli spaces of curves with marked points we denote by $\omega$ the first Chern class of the
relative dualizing sheaf of $\barmoduli[{g,1}] \to \barmoduli$ and $\omega_{i}$ its pullback to $\barmoduli[{g,n}]$ via
the map forgetting all but the $i$-th marked point. For a subset $S \subset
\{1,\ldots,n\}$ let $\delta_{i;S}$ denote the boundary divisor whose generic 
element is a smooth curve of genus $i$ joined at a node to a smooth curve of genus $g-i$ such that 
the component of genus $i$ contains exactly the marked points labeled by $S$.
\par
\begin{theorem} [\cite{AC}]
The rational Picard group of $\barmoduli$ for $g \geq 3$ is freely
generated by $\lambda$ and $\delta_i$, $i=0,\ldots, \lfloor g/2 \rfloor$. 
\par
More generally, the rational Picard group of $\barmoduli[{g,n}]$ 
for $g \geq 3$ is freely generated by $\lambda$,
$\omega_{i}$, $i=1,\ldots,n$, by $\delta_0$ and by 
$\delta_{i;S}$,  $i=0,\ldots, \lfloor g/2 \rfloor$, 
where $|S| > 1$ if $i=0$ and $1 \in S$ if $i= g/2$.
\end{theorem}
\par
Alternatively, we define $\psi_i \in \Pic(\barmoduli[{g,n}])$ to be the class
with value $-\pi_*(\sigma_i^2)$ on the universal family $\pi: \cX \to C$ with section $\sigma_i$ corresponding to the $i$th marked point. We have the relation
$$ \omega_{i} = \psi_i - \sum_{i\in S} \delta_{0; S}.$$ 
Consequently, a basis of $\Pic(\barmoduli[{g,n}])$ can be formed
by $\lambda$, the $\psi_i$ and the boundary classes as well.
\par
Let $W$ be the divisor of Weierstrass points in $\BM_{g,1}$. It has 
divisor class (see e.g. \cite{Cu})
\be \label{eq:PicWP}
W =- \lambda + \frac{g(g+1)}{2}\omega_1 - \sum_{i=1}^{g-1}\frac{(g-i)(g-i+1)}{2}\delta_i. 
\ee
\par
Let $BN^1_{g, (s_1,\ldots,s_r)}$ be the pointed Brill-Noether divisor 
parameterizing pointed curves $(X, z_1, \ldots, z_r)$ in $\barmoduli[g,r]$ where
$\sum_{i=1}^r s_i = g$ such that  
$h^0(X, \sum_{i=1}^r s_iz_i) = 2$. In particular, $BN^1_{g,(g)}$ is just the divisor $W$ of Weierstrass points. 
\par
The divisor class of $BN^1_{g, (1,\ldots,1)}$ was fully worked out in \cite[Section 5]{Logan}. The divisor class 
of $BN^1_{g, (s_1,\ldots,s_r)}$ was also implicitly calculated there. Below we give explicitly the divisor classes for the cases we need.
\par

\subsection{\bf Genus $3$}
Let $H$ be the divisor of hyperelliptic curves in $\BM_3$. 
It has divisor class 
\be \label{eq:PicH}
 H = 9\lambda - \delta_0 - 3\delta_1. 
\ee

We also have pointed Brill-Noether divisor classes as follows. 
\ba \label{eq:PicBNonly1s}
BN^1_{3,(1,1,1)} = &- \lambda + \omega_1 + \omega_2 + \omega_3 - \sum_{i,j} 
\delta_{0; \{i,j\}} \\
& - 3\delta_{0; \{1,2,3\}} -  \sum_{i,j}\delta_{1,\{i,j\}} - \delta_{1;\emptyset} -3\delta_{1; \{1,2,3\}}. 
\ea
\be \label{eq:PicBN3;21}
BN^1_{3,(2,1)} = - \lambda + 3\omega_1 + \omega_2 - 2\delta_{0;\{1,2\}} - \delta_{1;\emptyset} 
- \delta_{1;\{1\}} - 3\delta_{1;\{1,2\}}. 
\ee 
\par
As noted above, the class of $BN^1_{3,(1,1,1)}$ was calculated in \cite[Section 5]{Logan}. The class of 
$BN^1_{3,(2,1)}$ essentially follows from $BN^1_{3,(1,1,1)}$. We skip this calculation and instead, 
we will prove a completely analogous but harder case in genus four.  
\par
\subsection{\bf Genus $4$} In genus four we need the following pointed 
Brill-Noether divisors.
\par
\begin{lemma}\label{le:PicBN4}
The pointed Brill-Noether divisors in genus four
have divisor classes as follows. 
\ba  \label{eq:PicBN4a}
W = BN^1_{4, (4)} =  &- \lambda + 10\omega - 6\delta_1 - 3\delta_2 - \delta_3. \\
BN^1_{4,(1,1,1,1)}  = &- \lambda + \omega_1 + \omega_2 + \omega_3 + \omega_4 - 
\sum_{|S|\geq 2} \frac{|S|(|S|-1)}{2}\delta_{0;S} \\
&- \sum_{|S|\neq 1} \frac{(||S|-1|)(||S|-1|+1)}{2} \delta_{1;S} \\
&- \sum_{|S|\neq 2} \frac{(||S|-2|)(||S|-2|+1)}{2} \delta_{2;S}. \\
BN^1_{4,(2,1,1)} = &-\lambda + 3 \omega_1 + \omega_2 + \omega_3 - 2\delta_{0;\{1,2\}} 
- 2\delta_{0;\{1,3\}} \\
&- \delta_{0;\{2,3\}} - 5 \delta_{0;\{1,2,3\}}  
 - \delta_{1;\emptyset} - \delta_{1;\{1\}} - \delta_{1;\{2,3\}} - 3 \delta_{1;\{1,3\}} \\
&- 3\delta_{1;\{1,2\}} - 6\delta_{1;\{1,2,3\}} - 6\delta_{2;\emptyset} - 2\delta_{2;\{2\}} - 
2\delta_{2;\{3\}}. \\
\ea
\ba  \label{eq:PicBN4b}
BN^1_{4,(3,1)} = &- \lambda + 6\omega_1 + \omega_2 - 3\delta_{0;\{1,2\}}  \\
&- \delta_{1;\emptyset} - 3 \delta_{1;\{1\}} - 6\delta_{1;\{1,2\}} - 
6\delta_{2;\emptyset} - 2\delta_{2;\{1\}}.\phantom{mmml}\\
BN^1_{4,(2,2)} = &-\lambda + 3\omega_1 + 3\omega_2 - 4\delta_{0;\{1,2\}} \\
&- \delta_{1;\emptyset} - \delta_{1;\{1\}} - \delta_{1;\{2\}} - 6\delta_{1;\{1,2\}} - 
6\delta_{2;\emptyset}.\\
\ea
\end{lemma}
\par
\begin{proof}
Let $\pi_i : \BM_{g,n}\to \BM_{g,n-1}$ be the map forgetting the $i$th marked point. Then
$$\pi_{n*}(BN^{1}_{g, (a_1,\ldots, a_n)} \cdot \delta_{0;\{n-1,n\}}) = BN^1_{g, (a_1,\ldots, a_{n-2}, a_{n-1}+ a_n)}.$$ 
We use it to determine the class of $BN^1_{4,(2,1,1)}$, namely, 
$$ \pi_{4*}(BN^1_{4,(1,1,1,1)} \cdot \delta_{0; \{1,4\}}) = BN^1_{4,(2,1,1)}. $$
Based on \cite[Table~1, p.~112]{Logan}, we have 
$$\begin{array}{ll}
\pi_{4*}(\omega_1\cdot \delta_{0; \{1,j\}}) = \omega_j, & \text{for} \quad j=1,2,3,\\
\pi_{4*}(\omega_4\cdot \delta_{0; \{1,4\}}) = \omega_1, & \pi_{4*}(\lambda\cdot \delta_{0; \{1,4\}}) = \lambda, \\
\pi_{4*}(\delta_{0;\{1,4\}}\cdot \delta_{0; \{1,4\}}) = -\psi_1  & 
\!\!\!\!\!\!\!\!\!\!\!\!\!\!\!\!\!\!\!\!\!\!\!\!\!\!\!\!\!\!\!\!\!\!\!\!\!
= - \omega_1 - \sum_{1\in S}\delta_{0; S} , \\
\pi_{4*}(\delta_{i;S}\cdot \delta_{0; \{1,4\}}) = \delta_{i;S}, \ \{1,4\}\cap S = \emptyset, &
\pi_{4*}(\delta_{i;S}\cdot \delta_{0; \{1,4\}}) = 0, \ 1\in S, \ 4\not\in S, \\
\pi_{4*}(\delta_{i;S\cup \{4\}}\cdot \delta_{0; \{1,4\}}) = \delta_{i;S}, \ 1\in S, &
\pi_{4*}(\delta_{i;S\cup \{4\}}\cdot \delta_{0; \{1,4\}}) = 0, \ 1\not\in S. \
\end{array}$$
As a consequence, 
$$ BN^1_{4,(2,1,1)} = -\lambda + 2\omega_1 + \omega_2 + \omega_3 + \omega_1 + \sum_{1\in S}\delta_{0; S} - \delta_{0;\{2,3\}} - 3\delta_{0;\{1,2\}} - 3\delta_{0;\{1,3\}} $$
$$ - 6\delta_{0;\{1,2,3\}} - \delta_{1;\emptyset} - \delta_{1;\{1\}} - \delta_{1;\{2,3\}} - 3 \delta_{1;\{1,3\}} - 3\delta_{1;\{1,2\}} - 6\delta_{1;\{1,2,3\}}$$
$$ - 3\delta_{2;\emptyset} - \delta_{2;\{2\}} - \delta_{2;\{3\}} - \delta_{2;\{1,2\}} - \delta_{2;\{1,3\}} - 3\delta_{2;\{1,2,3\}}. $$
Using $ \pi_{3*}(BN^1_{4,(2,1,1)} \cdot \delta_{0; \{1,3\}}) = BN^1_{4,(3,1)}$, we deduce that 
$$ BN^1_{4,(3,1)} = -\lambda + 3\omega_1 + \omega_2 + \omega_1 + 2 \omega_1 + 2\delta_{0; \{1,2\}} - 5\delta_{0;\{1,2\}} $$
$$ - \delta_{1;\emptyset} - 3 \delta_{1;\{1\}} - 6\delta_{1;\{1,2\}} - 6\delta_{2;\emptyset} - 2\delta_{2;\{2\}}. $$ 
Similarly, using $ \pi_{3*}(BN^1_{4,(2,1,1)} \cdot \delta_{0; \{2,3\}}) = BN^1_{4,(2,2)}$, we conclude that 
$$ BN^1_{4,(2,2)} = -\lambda + 3\omega_1 + \omega_2 + \omega_2 + \omega_2 + \delta_{0;\{1,2\}} - 5\delta_{0;\{1,2\}} $$
$$ - \delta_{1;\emptyset} - \delta_{1;\{1\}} - \delta_{1;\{2\}} - 6\delta_{1;\{1,2\}} - 6\delta_{2;\emptyset}. $$
In each case, after simplifying we get the result stated above.
\end{proof}
\par
\subsection{\bf Genus $1$}
Recall the divisor theory of $\BM_{1,n}$. Its rational Picard group is freely generated by $\lambda$ and $\delta_{0;S}$ for 
$2\leq |S| \leq n$. Moreover, we have 
$$\delta_{0} = 12\lambda, \quad \quad \quad 
 \psi_i = \lambda + \sum_{i\in S} \delta_{0;S}. $$

\subsection{\bf Genus $2$}

On the moduli space $\barmoduli[2]$, the rational Picard group is generated by
$\lambda$, $\delta_0$ and $\delta_1$ with the relation (see \cite[Theorem~2.2]{ArCoCalculating}) 
\be \label{eq:relg2}
\lambda = \frac{\delta_0}{10} + \frac{\delta_1}{5}.
\ee
Given the structure of the Picard group it is natural to define the
{\em slope} of a \Teichmuller curve $C$ as 
$$s(C) = \frac{C \cdot \delta}{C \cdot \lambda}.$$
In general, slope can be defined for any one-parameter family of stable genus $g$ curves, which measures how the complex structures  
vary with respect to the number of singularities in the family.  
\par
For abelian and quadratic differentials, the slope of a \Teichmuller curve carries
as much information as the Siegel-Veech constant or the sum of Lyapunov
exponents thanks to the following consequence of Noether's formula
(see \cite[Theorem~1.8]{chenrigid})
\begin{equation}
\label{eq:Lcs}
s(C) = \frac{12c(C)}{L(C)} = 12 - \frac{12 \kappa_\mu}{L(C)}.  
\end{equation}
\par
Consequently the following result is quite useful 
to prove non-varying strata in genus two. 
\par
\begin{lemma}
\label{slope10}
If a one-dimensional family of stable curves of genus two does not intersect $\delta_1$, then its slope is $10$. 
\end{lemma}
\begin{proof}
This follows from the definition of slope and the relation~\eqref{eq:relg2}. 
\end{proof}
\par
The rational Picard group of $\barmoduli[2,n]$ is generated by the divisor classes 
$\lambda$, $\omega_{i}$, $i=1,\ldots,n$, by $\delta_0$,  by 
$\delta_{0;S}$ with $|S| > 1$ and by  $\delta_{0,S}$ with $1 \in S$. By 
\cite[Theorem~2.2]{ArCoCalculating} the only relation among them is
$$  5 \lambda = 5 \psi + \delta_{0} - \sum_{|S|>1} \delta_{0,S} +7 \sum_{1\in S} \delta_{1,S}. $$
We will use two divisors in genus two. The divisor of
Weierstrass points in $\barmoduli[2,1]$ has class 
\be \label{eq:Wg2}
W = -\lambda + 3\omega_1 - \delta_1
\ee
and the pointed Brill-Noether divisor in $\barmoduli[2,2]$ has class
\be \label{eq:BNg2}
BN^1_{2, (1,1)} = -\lambda + \omega_1 + \omega_2 
- \delta_{0;\{1,2\}} - \delta_{1;\emptyset}.
\ee
\section{Properties of  Teichm\"uller curves} \label{sec:propTeich}

\subsection{Computing intersection numbers}

Let $C$ be a Teichm\"uller curve generated by a half-translation
surface in $\QQ(d_1,\ldots, d_n)$. We always work with an appropriate unramified
cover of $C$ whose uniformizing group is torsion free and we denote
this cover still by $C$. In particular we will assume that
over $C$ there is a universal family $\pi: \XX \to C$. 
This also implies that $\chi = 2g(C) - 2 + |\Delta|$, where 
$\Delta$ is the set of cusps in $C$ and $\chi$ is the orbifold Euler characteristic of $C$. 
Denote by $S_j$ the section of $\XX$ corresponding to the zero or pole of order $d_j$. Use $\omega_{\pi}$ to denote the relative dualizing sheaf of $\pi$. 
\par
\begin{proposition}[\cite{kontsevich}; \cite{bouwmoel}]
For a Teichm\"uller curve $C$ generated by a
half-translation surface in $\QQ(d_1,\ldots, d_n)$, we have 
$$ L^+(C) = \frac{2 \deg \lambda }{\chi}. $$
\end{proposition}
\par
\begin{proposition}
\label{intersection}
Intersection numbers of the Teichm\"uller curves with various divisor classes and the sum of Lyapunov exponents are related as follows: 
\ba  S_j^2 = - \frac{\chi}{d_j+2},  \quad\quad & S_j\cdot \omega_{\pi} = \frac{\chi}{d_j+2}, \\
C\cdot \delta = 6\chi\cdot c, \quad\quad & C\cdot \lambda = \frac{\chi}{2}\cdot (c + \kappa), 
\ea
where $c$ is the Siegel-Veech constant related to $L^+$ by
\eqref{eq:EKZ}.  
\end{proposition}
\par
\begin{proof}
Let $\cF$ be the universal line bundle on $C$ parameterizing the quadratic differentials 
that generate $C$. Denote by $S$ 
the union of the sections $S_j$ for $j = 1,\cdots, n$. 
By the exact sequence 
$$ 0\to \pi^{*}\cF \to \omega^{\otimes 2}_{\pi} \to \OO_S\Big(\sum_{j=1}^n d_jS_j\Big) \to 0 $$
and the fact that $\deg \cF = \chi$ (see \cite{moeller06}), one calculates that 
$$S^2_j = - S_j \cdot \omega_{\pi} = - \frac{\chi}{d_j + 2}, $$
which shows the first two formulas. Moreover, it implies that 
$$ c_1^2(\omega_{\pi}) = \frac{\chi}{4}\cdot\Big(\sum_{j=1}^n \frac{d_j(d_j+4)}{d_j+2}\Big) = 6\chi\cdot \kappa. $$
By Noether's formula, we know that
$$ 12 \lambda = \delta + c_1^2(\omega_{\pi}). $$  
Dividing both sides by $6\chi$, the left hand side equals $L^+(C)$ by the preceding 
proposition and the right hand side equals 
$\frac{\delta}{6\chi} + \kappa$. By \eqref{eq:EKZ}, we reads off 
$c = \frac{\delta}{6\chi}$. Hence the last two formulas follow immediately. 
\end{proof}

\begin{remark}
We can also deduce the above formulas by passing to the canonical double cover. Note that 
$\QQ(d_1, \ldots, d_n) \to \Omega\MM(\cdots, d_i/2, d_i/2, \cdots, d_j + 1, \cdots)$ 
for $d_i$ even and for $d_j$ odd, since the double cover 
is branched at the singularities of odd order. Restrict this to a Teichm\"uller curve 
$C$ in $\QQ(d_1, \ldots, d_n)$. Then it gives rise to a Teichm\"uller curve 
isomorphic to $C$ in the corresponding stratum of abelian differentials. 
We have the following commutative diagram
$$\xymatrix{
\XX' \ar[rr]^f \ar[dr] & & \XX \ar[dl] \\
& C  &  }$$
and let $S'_j$ be the section of $\XX' \to C$ over $S_j$ in case $d_j$ is odd and $S_{j,1}, S_{j,2}$ be the sections over $S_j$ in case $d_j$ is even. 
Then we have 
\bas
f_{*}S'_j = S_j, \quad\quad & f^{*}S_j = 2S'_j,  \\
f_{*}(S_{j,1} + S_{j,2}) = 2S_j, \quad\quad & f^{*}S_j = S_{j,1} + S_{j,2}. 
\eas
In the case when $d_j$ is odd, we have 
$$ S^2_j = (f_{*}S'_j)\cdot S_j = 2(S'_j)^2 = - \frac{\chi}{d_j+2}. $$
In the case when $d_j$ is even, we have 
$$ S^2_j = \frac{1}{2} (f_{*}(S_{j,1} + S_{j,2}))\cdot S_j = \frac{1}{2} (S^2_{j,1}+S^2_{j,2}) = - \frac{\chi}{d_j+2}. $$
Hence we recover the self-intersection formula. 
\end{remark}

\subsection{Boundary behavior} \label{sec:Teichbd}

The following results are needed later for the proofs of non-varying strata. 
Roughly speaking, they imply that degenerate half-translation surfaces parameterized 
in a Teichm\"uller curve behave similarly to the smooth ones and, 
as in the case of abelian differentials, the corresponding stable
curves are obtained by squeezing core curves of cylinders (see
\cite[Proposition~5.9]{moelPCMI}).
\par 
\begin{proposition} \label{prop:Teichbdconstr}
Suppose $C$ is a Teichm\"uller curve generated by a quadratic 
differential in $\QQ(d_1, \ldots, d_n)$.
The pointed stable curves in $\barmoduli[g,k]$ corresponding to the boundary 
points $\Delta$ of $C$ are obtained by choosing
a parabolic direction of a generating half-translation surface $(X,q)$ and
replacing each cylinder by a pair of half-infinite cylinders whose
points at $i\infty$ resp.\ at $-i\infty$ are identified.
\end{proposition}
\par
\begin{proof}
The cusps of Teichm\"uller curves are
obtained by applying the Teichm\"uller geodesic flow $(e^{t/2}, e^{-t/2})$ 
to a direction in which $(X, q)$ decomposes completely into cylinders.
Once we have shown that the object resulting from the above cylinder 
replacement construction is stable (including the punctures), 
the rest of the proof is the same as in \cite[Propositions~5.9 and~5.10]{moelPCMI}.
\par
We need to show that each rational tail (i.e.\ a genus zero
component of a stable curve joined to the rest of the curve
at a separating node and without nodes joining the tail to itself)
has at least two punctures. If we cut along the core curve $\gamma$
that produces the separating node and glue the two halves of
$\gamma$ together, we obtain a closed half-translation surface $(\PP^1,q_\PP)$
of genus zero with two simple poles on the glued $\gamma$. Since
$\deg(\divisor(q_\PP)) = -4$ and since poles are simple, 
there exist two more poles somewhere on this $\PP^1$, proving
our claim.
\end{proof}
\par
\begin{corollary}
\label{signature}
The section $q$ of $\omega_X^{\otimes 2}$ of each smooth fiber $X$ over
a Teichm\"uller curve extends to a section $q_{\infty}$ for each degenerate fiber $X_{\infty}$ over the
closure of a Teichm\"uller curve. The signature of zeros and poles of $q_{\infty}$ is the same as $q$.
\end{corollary}

\begin{corollary}
\label{component}
Let $(X_{\infty},q_\infty)$ be a degenerate fiber of a Teichm\"uller curve generated by a half-translation 
surface $(X,q)$. Then $(X_{\infty},q_\infty)$ is of polar type.
\par
Moreover, every irreducible component of $X_{\infty}$ contains at least one singularities
of $q_{\infty}$. In particular, the number of 
irreducible components of $X_{\infty}$ is bounded from above by the number of singularities of $q$.  
\end{corollary}
\par
\begin{remark}
If $q$ is a global square of an abelian differential, then $X_{\infty}$ does not have 
separating nodes, as a consequence of
the topological fact that the core curve of a cylinder does not disconnect a  
flat surface. In other words, Teichm\"uller curves generated 
by abelian differentials do not intersect $\delta_i$ for $i > 0$ in $\BMg$. 
On the other hand, if $q$ is not a global square, then the Teichm\"uller curve generated 
by $q$ may intersect $\delta_i$ for $i > 0$. 
\end{remark}

\begin{proposition} \label{prop:disjointbydegree}
Let $C$ be a Teichm\"uller curve generated by a half-translation
surface in $\QQ(d_1,\ldots,d_n)$. Let $\ol{C}$ be the closure
of the lift of $C$ to $\barmoduli[g,m]$ using the first $m \leq n$ singularities.
Then $\ol{C}$ is disjoint with the boundary divisors that have non-zero coefficients in the divisor classes of
the Brill-Noether divisors given in Section~\ref{sec:divclass}, if 
the tuple $(g,m,\QQ(d_1,\ldots,d_n))$ and the divisor are listed in the following table. 
\end{proposition}
\begin{figure}[h]
$$\begin{array}{|c|c|l|l|}
\hline
g & m & \text{\rm Stratum} \,& \text{\rm Divisor} \\
\hline
2 & 1 & \QQ(3,2,-1), \QQ(6,-1,-1), \QQ(5,1,-1,-1), \QQ(7,-1,-1,-1) & W \\
\hline
2 & 2 & \QQ(3,1,1,-1), \QQ(2,2,1,-1), \QQ(4,2,-1,-1),&  \\ 
& & \QQ(3,3,-1,-1), \QQ(3,2,1,-1,-1), \QQ(4,3,-1,-1,-1),  & BN^1_{2,(1,1)} \\
\hline
3 & 1 & \QQ(8), \QQ(7,1), \QQ(9, -1),\QQ(8,1,-1), & \\ 
&&\QQ(10, -1, -1), \QQ(9,1,-1,-1) & W \\
\hline
3 & 2 & \QQ(6,2), \QQ(5,3), \QQ(4,4), \QQ(6,1,1),\QQ(7,2,-1), \QQ(5,4,-1), & \\
&& \QQ(4,3,1), \QQ(5,2,1), \QQ(6,3,-1), \QQ(5,3,1,-1), \QQ(7,3,-1^2) \!\!\!& BN^1_{3,(2,1)} \\
\hline
3 & 3 & \QQ(4,2,2), \QQ(3,3,2), \QQ(3,3,3,-1), \QQ(4,3,2,-1) & \\
& & \QQ(3,2,2,1), \QQ(3,3,1,1) &BN^1_{3,(1,1,1)} \\
\hline
4 & 1 & \QQ(13,-1), \QQ(12), \QQ(11,1) & W\\
\hline
4 & 2 & \QQ(10,2), \QQ(9,3), \QQ(8,4), \QQ(8,3,1) & BN^1_{4,(3,1)} \\
\hline
4 & 2 & \QQ(6,6) & BN^1_{4,(2,2)}\\
\hline
4 & 3 & \QQ(7,3,2)$, $\QQ(6,3,3)$, $\QQ(5,4,3) & BN^1_{4,(2,1,1)} \\
\hline
4 & 4 & \QQ(3,3,3,3) & BN^1_{4,(1,1,1,1)} \\
\hline
\end{array}
$$
\end{figure}
\par
\begin{proof}
An irreducible component $Z$ of a degenerate half-translation surface $X_{\infty}$ 
over a cusp of $C$ contains at least one zero or pole of the degenerate 
quadratic differential $q_{\infty}$. 
Moreover, $\omega_{X_{\infty}}^{\otimes 2}$ restricted to $Z$ has degree 
equal to $4g(Z) - 4 + 2m$, where $m = \# (Z\cap \overline{X_{\infty}\backslash Z})$. 
Using these facts, the claim follows easily by a case-by-case study. For instance, 
let us show that a Teichm\"uller curve $C$ generated by a
half-translation surface in $\QQ(7,1)$ does not intersect 
$\delta_1$ in $\BM_{3,1}$. Otherwise, there exists a 
degenerate half-translation surface $X_{\infty}$ consisting of two components $Z_1$ and 
$Z_2$ of genus $1$ and $2$, respectively, joined at a node such that 
$\divisor(q_{\infty}) = 7p_1 + p_2$ for two distinct points $p_1, p_2\in X_{\infty}$. 
But the degree of $\omega^{\otimes 2}_{X_{\infty}}$ restricted to $Z_1$ is $2$, 
in particular, not equal to $7$ or $1$. Hence $Z_1$ does not contain any zero of $q_{\infty}$, contradicting Corollary~\ref{component}.   
\end{proof}
\par
\begin{proposition}
\label{global}
Let $C$ be a Teichm\"uller curve parameterizing half-translation
surfaces $(X_t, q_t)$ such that for generic $t$ the quadratic
differential $q_t$ is not a global square of an 
abelian differential. Then $q_0$ is not a global square of a stable one-form on the special fiber $X_0$.
\end{proposition}
\par
\begin{proof}
If $q_t$ has a singularity of odd order, the claim is obvious. Assume that all singularities 
are of even order. 
Let $\XX\to C$ be the universal curve and $\Gamma\subset \XX$ the divisor parameterizing 
the singularities of  $q_t$. Define $\LL = \OO_\XX(\Gamma/2)$, which is a well-defined line bundle
by the assumption. Denote by $s\in H^0(\XX, \LL^{\otimes 2})$ the section whose vanishing 
locus is $\Gamma$. Then there exists a canonical double covering $\pi: \XX' \to \XX$ such that 
$\pi^{*}\LL$ possesses a section $s'$ satisfying $(s')^2 = \pi^{*} s$. In other words, along each fiber $X'_t\to X_t$ the pullback of $q_t$ is a square of an abelian differential. Moreover, 
$q_t$ is a global square if and only if $X'_t$ is disconnected. Now the result follows from the fact that a family of connected curves cannot specialize to a disconnected one. 
\end{proof}
\par
The following proposition says that some hyperelliptic and non-hyperelliptic
strata stay disjoint even along the boundary of $\barmoduli$. For a
motivation the reader may compare with the corresponding statement
for abelian differentials in \cite[Proposition~4.4]{chenmoeller}.
\par
\begin{proposition}
\label{prop:hypopen}
Let $C$ be a Teich\"uller curve generated by a half-translation 
surface $(X,q)$.
Suppose that $(X_\infty,q_\infty)$ is a point corresponding to one 
of the cusps $\ol{C} \setminus C$.
If $(X_\infty,q_\infty)$ is a hyperelliptic half-translation surface
and $X_\infty$ is irreducible, then  $(X,q)$ is a hyperelliptic 
half-translation surface, too.
\par
If $(X,q)$ is a half-translation surface in one of the strata 
$$\QQ(6,2), \QQ(6,1,1), \QQ(3,3,2), 
\QQ(10,2), \QQ(6,-1,-1), \QQ(3,3,-1^2) \,\, \text{or}\,\, \QQ(10,-1^2),$$ 
then the conclusion
holds without the irreducibility assumption.
\par
For the stratum $\QQ(3,3,1,1)$ the same conclusion holds
for the cusps parameterizing stable curves with a separating node.
\end{proposition}
\par
\begin{proof} 
Let $(X_0,q_0)$ denote the 'core' of the half-translation surface $(X_\infty,q_\infty)$, 
that is obtained by removing from $(X_\infty,q_\infty)$ the maximal half-infinite cylinders, 
i.e.\ the cylinders corresponding to the nodes of $X_\infty$ that are swept out by
closed geodesics in the direction of the residue of $q_\infty$ at a node.
\par
If $X_\infty$ is hyperelliptic, there is an admissible double cover of a semistable
model of $X_\infty$ to $\PP^1$. Suppose first that $X_\infty$ is irreducible with $h$
nodes. Then the double cover induces an involution $\rho$ of the semistable model of $X_\infty$
that, by our hypothesis, acts as $(-1)$ on $q_\infty$. 
This involution $\rho$ preserves the central component $X_\infty$ as well as all the
non-stable projective lines. We deduce that $\rho$ preserves $X_0$ and interchanges
each pair of half-infinite cylinders that is glued together to form a node. Moreover,
$\rho$ has $2(g-h)+2$ fixed points on $X_0$.
\par
Nearby surfaces $(X,q)$ in $C$ are obtained by replacing the pairs of infinite cylinders
with cylinders of finite height. Since all pairs of infinite half-cylinders
are preserved by $\rho$, we can extend $\rho$ to an involution on $(X,q)$ still
acting as $(-1)$ on $q$ with $2$ fixed points in each cylinder. This gives
$2g+2$ fixed points in total and $(X,q)$ is a hyperelliptic half-translation surface, 
as we claimed.
\par
Boundary points of the \Teichmuller curve generated by a half-translation surface in one
of the above strata are either irreducible or consist of two components $X_1$ of genus one
and $X_2$ of genus $g-1 >1$ in the first four cases and $g(X_1)=0$, $g(X_2)=g$
in the next three cases. The involution $\rho$ induced by the admissible double
covering cannot exchange the two components and it has to fix the unique node
joining the two components. Reproducing the preceding argument for nodes
joining the irreducible components to itself, if there are, we conclude that
$\rho$ fixes all pairs of half-infinite cylinders. We can now complete the 
proof as in the preceding case.
\par
Thanks to the restrictions on the cusps in question we are just in position
to apply the argument again to the last case.
\end{proof}
\par

\section{Limit canonical curves in genus three and four} \label{sec:limits}

For a smooth, non-hyperelliptic curve $X$ of genus three, its canonical embedding is a 
smooth plane quartic. In that case, the zeros of a holomorphic quadratic differential correspond to 
the intersection of a unique plane conic with $X$. This picture holds more generally if $X$ is nodal, 
non-hyperelliptic and $\omega_X$ is very ample, i.e. the dual graph of $X$ is $3$-connected (see e.g.~\cite[Proposition 2.3]{Hassett}). 
If $X$ is  $2$-connected or less, or if $X$ is hyperelliptic, its canonical map may 
no longer be an embedding. We need a replacement for this picture, when the canonical map of 
a curve of genus three (and also of genus four) fails to be an embedding. 
\par
Since quadratic differentials are sections of $\omega_{X}^{\otimes 2}$, naturally we should consider the bicanonical map of $X$. The bicanonical linear system on a stable genus three curve $X$ provides an embedding to $\PP^5$, unless $X$ possesses an elliptic tail. Now conics in $\PP^2$ corresponding to quadratic differentials become hyperplane sections in $\PP^5$. Then the ambient $\PP^2$ containing $X$ turns out to be a surface in $\PP^5$ as the image of the Veronese embedding induced by $|\OO_{\PP^2}(2)|$. As $X$ degenerates, say to a smooth hyperelliptic curve $X_0$, we do not have a canonical embedding of $X_0$ in $\PP^2$. Nevertheless, in $\PP^5$ we have a bicanonical embedding of $X_0$ contained in a singular surface which is a degeneration of the Veronese $\PP^2$.  In summary, the idea is to treat the pair $(\PP^2, X)$ as a \emph{log surface}, i.e. a surface plus a divisor with mild singularities, and degenerate $\PP^2$ as well. This procedure was carried out by Hassett completely for stable genus three curves \cite{Hassett}. 
We summarize Hassett's results as well as the degenerations we need in genus four. 
This will be useful when we analyze the boundary behavior of  \Teichmuller curves in the exceptional strata. 
\par
Recall that a {\em Hirzebruch surface $F_d$} is a ruled surface over $\PP^1$ such
that the section $e$ with minimal self-intersection number has $e^2 = -d$. Let $f$
be a ruling of $F_d$. The rank of the Picard group of $F_d$ is two, hence any curve 
class of $F_d$ can be written as a linear combination $ae + bf$ with integer coefficients 
$a, b$ and $f^2 = 0$, $f\cdot e = 1$. Moreover, the canonical line bundle of $F_d$ has class 
$$ K_{F_d} = -2e - (2+d)f. $$
\par

\subsection{Genus three}\label{logsurface}

We start with a description of the generic situation. Consider the Veronese 
embedding $\PP^2 \hookrightarrow \PP^5$ by the complete linear system $|\OO(2)|$.
Denote by $S$ the image surface of degree four. A plane curve of degree $d$ maps to a 
curve of degree $2d$ in $S\subset \PP^5$. Holomorphic quadratic differentials on $X$ modulo scalar 
are in bijection to conics $Q$ in $\PP^2$. They become hyperplane sections of $S$ 
in $\PP^5$.
\par
\begin{proposition}
In the above setting, suppose a family of log surfaces $(S_t, X_t)$ degenerates
such that $X_0$ is hyperelliptic but still $3$-connected. Then $S_t$ degenerates
to a surface with one singularity of type $A_3$, whose resolution
is a Hirzebruch surface $F_4$. On $F_4$, the curve $X_0$ has class $2e+8f$.
\par
Two points $p_1, p_2$ are conjugate in $X_0$ if and only if they are cut out by a ruling.
A point $p$ is a Weierstrass point of $X_0$ if and only if there is a ruling 
tangent to $X_0$ at $p$. 
\par
The degeneration of divisors associated to holomorphic quadratic differentials on $X_t$ consists of four rulings or
has class $e + 4f$, depending on whether the hyperplane section $Q$ 
passes through the singularity.
\end{proposition}
\par
In fact, $S_t$ degenerates to a cone $S_{4}$ over a rational normal quartic curve 
$R$ in $\PP^5$. More precisely, take a smooth, degree four, rational curve that 
spans a $\PP^4$, and a point $v$ not contained in that $\PP^4$. Then $S_{4}$ consists 
of the union of lines connecting $v$ with each point of the rational normal quartic. 
Those lines are called rulings of $S_{4}$.  Blowing up the $A_3$-singularity at $v$, 
we obtain the resolution $F_4$. More details on this and the following proposition can be found in \cite{Hassett}.
\par
The second case we need is precisely $2$-connected.
\par
\begin{proposition}
In the above setting, suppose a family of log surfaces $(S_t, X_t)$ degenerates
such that $X_0$ is hyperelliptic and consists of two components $X_1, X_2$, both of genus one, meeting at two nodes $t_1, t_2$. Then $S_t$ degenerates to a cone $S_{2,2}$ over a one-nodal union of two conics in $\PP^5$
Let $S_1, S_2$ be the two components of $S_{2,2}$, containing $X_1, X_2$, respectively. Then their common ruling contains $t_1, t_2$.
\par
The zeros of a degenerate quadratic differential on $X_0$ are cut out by two conics in each of the $S_i$
or it consists of four rulings, two in each of the $S_i$, depending on whether the hyperplane 
section $Q$ passes through the singularity of $S_{2,2}$.
\end{proposition}
\par

\subsection{Genus four}
\label{g=4: logsurface}

We start again with a description of the generic situation. For a general genus four curve $X$, its canonical embedding 
lies in a smooth quadric surface $Q$ in $\PP^3$. Let $f_1, f_2$ be the two ruling classes of $Q$.
Then $X$ has class $3f_1 + 3f_2$. For any holomorphic quadratic differential $q$ on $X$, there exists a unique 
elliptic curve $E$ (possibly singular) as an element of the linear system $|2f_1 + 2f_2|$ such that
 $E\cdot X = \divisor(q)$. 
\par
With the hyperelliptic situation in mind, we need to 
consider the bicanonical embedding. Then $Q$ is embedded into $\PP^8$ by the linear system 
$|2f_1 + 2f_2|$, i.e.\ by its anti-canonical system $|-K_Q|$. The image of $Q$ is a surface $S$ of degree eight. 
In particular, $S$ is a conic bundle correspondence between two conics, i.e.\ the 
rulings of $Q$ map to conics in $\PP^8$.
\par
The {\em Gieseker-Petri divisor} is the (closure of) the locus of genus four curves such
that the quadric surface $Q$ is singular. We call $X$ contained in the Gieseker-Petri divisor
as \emph{Gieseker-Petri special}. 
\par
\begin{proposition}
In the above setting, suppose $X$ is Gieseker-Petri special and $3$-connected, but not hyperelliptic. Then its canonical
image lies on a quadric cone $Q_0$ in $\PP^3$. Blowing up the vertex gives a Hirzebruch surface $F_2$. 
On $F_2$ the class of $X$ is $6f + 3e$ and the zeros of a holomorphic quadratic differential
correspond to a divisor of class $4f + 2e$. 
\end{proposition}
\par
\begin{proposition}
In the above setting, suppose $X$ is hyperelliptic and $3$-connected. Then its bicanonical
image lies on a surface scroll isomorphic to a
Hirzebruch surface $F_5$. On $F_5$ the class of $X$ is  
$10f + 2e$ and the zeros of a holomorphic quadratic differential
correspond to the union of a line and a divisor of class $6f + e$. 
\end{proposition}
\par
\begin{proof} Let $X$ be a hyperelliptic curve of genus four, embedded in $\PP^8$ 
by its bi-canonical system. A pair of conjugate points under the hyperelliptic involution of $X$ span a line in $\PP^8$. 
Take the union of these lines and we obtain a surface scroll of degree seven containing $X$. 
It is easy to check that the resulting surface is $S_{1,6}$, i.e.\ the union of line correspondences between 
a line $L$ and a rational normal sextic. Namely, $S_{1,6}$ is a component of a degree eight surface $S$, which is 
a degeneration of the degree eight anti-canonical embedding of $Q$, hence the other component of $S$ must have degree one, i.e. a plane $\PP^2$. The plane is spanned by $L$ and a ruling $F$, since $S$ is embedded in $\PP^8$ by 
its anti-canonical system.  
\par
Note that $S_{1,6}$ is the image of 
the Hirzebruch surface $F_5$ by $|e + 6f|$. Under this embedding, the distinguished
section $e$ maps to $L$ and $f$ maps to a ruling. Consequently,  the class of $X$ is 
$2e + 10f$. Moreover, the elliptic curve $E$ cutting out the zeros of a quadratic differential on $X$ has the same class as the anti-canonical class 
of the log surface $S = \PP^2\cup S_{1,6}$, i.e.\ $E$ is a hyperplane section of $S$. 
Its restriction to $S_{1,6}$ has class $e+6f$, whose sections are 
rational normal curves in the hyperplane $\PP^7$. Since $(e+6f)\cdot e = 
(e+6f)\cdot f = 1$, such a rational normal curve $R$ in $\PP^7$ intersects $L$ at $t_1$ 
and intersects $F$ at $t_2$. So $E$ is the union of the line $\overline{t_1t_2}$ and $R$, i.e.\ two 
rational curves jointed at two nodes. 
\end{proof}
\par

\section{Genus three: exceptional strata} \label{sec:g3exceptional}

By \cite{lanneauENS} in genus three the strata with an exceptional number of connected components are
$$\cE_3 = \{(9,-1),(6,3,-1),(3,3,3-1)\}.$$ For each of the strata $\QQ(k_1,k_2,k_3,-1)$
in this list, the following properties are known to hold. 
\begin{itemize}
\item[i)] The stratum $\QQ(k_1,k_2,k_3,-1)$ has exactly two components.
\item[ii)] Only one of the two components $\QQ(k_1,k_2,k_3,-1)^\reg$ is adjacent to $\QQ(8)$.
\item[iii)] The stratum $\QQ(9,-1)^{\reg}$ is obtained from $\QQ(5,-1)$ by gluing
in a cylinder with angle in $\{\pi,2\pi,4\pi\}$ and $\QQ(9,-1)^\irr$ is obtained
by gluing in a cylinder with angle $3\pi$.
\end{itemize}
\par
We stick to Zorich's notation (\cite{zorichexplicit}) on the labeling '$\reg$' and '$\irr$', 
corresponding to \emph{regular} and \emph{irregular}, respectively, for a reason that
will soon become clear. Originally \cite{lanneauENS} used the labels {\em reduced} and 
{\em irreducible}, to indicate which of the strata by property ii) could be reduced 
to the stratum $\QQ(8)$. But this mnemonic works in $g=3$ only.
\par
The gluing construction used in iii) gives a topological distinction of
the two components. In \cite{lanneauENS} Lanneau asked for an algebraic distinction. We provide a solution 
to this question as one of the main results of this section. In fact, we prove the following result
by purely algebraic construction. 
\par
For a quadratic differential $q$, let ${\rm div}(q)_0$
(resp.\ ${\rm div}(q)_\infty$) be the divisor of zeros (resp.\ of poles) of $q$. Define a divisor
$$\LL(X,q) = {\rm div}(q)_0/3$$ 
and also use the same notation to denote its associated line bundle on $X$. 
\par
\begin{theorem} \label{thm:mainexcg3}
The above properties i) and ii) hold for each stratum in $\cE_3$. Moreover, we have the following parity distinction: 
\begin{itemize}
\item[iii')] The surface $(X,q)$ belongs to 
$$\QQ(k_1,k_2,k_3,-1)^\irr \quad \text{if} \quad \dim H^0(X,\LL(X,q)) = 2$$
and it belongs to  
$$\QQ(k_1,k_2,k_3,-1)^\reg \quad \text{if} \quad \dim H^0(X,\LL(X,q)) = 1.$$
\end{itemize}
\end{theorem}
\par
Consequently, we may say that in the irregular components the linear system $|\LL(X,q)|$
has an irregularly high dimension.
\par
In spirit, this behavior is very similar to abelian differentials. There, 
in a stratum $\omoduli[g](k_1,\ldots,k_r)$
where all the $k_i$ are even, the two components are distinguished by the parity of the 
spin structure, i.e.\ by the parity of $H^0(X,{\rm div}(\omega)/2)$. Whereas the parity of
the spin structure is well-known to be deformation invariant, the invariance of the number
of sections of a third root of the zero divisor of a quadratic differential appears rather strange.
\par
To prove the above theorem, the first step is an equivalent interpretation of the parity condition in
terms of a torsion order on a secretly underlying elliptic curve. Next, we construct
the components of the strata in $\cE_3$ from moduli spaces and projective bundles that
are known to be irreducible.
\par
As a by-product, we found that almost all the exceptional strata are non-varying. 
\par
\begin{theorem} \label{thm:execg3NV}
All the components of the exceptional strata in genus $g=3$ with the
exception of $\QQ(3,3,3,-1)^\irr$ are non-varying. The values are collected in the following table: 
$$\begin{array}{|c|c|c|}
\hline
& \text{\rm component} \, \QQ^\irr & \text{\rm component} \, \QQ^\reg \\
\hline
\QQ(9,-1) & L^+ = 14/11 & L^+ = 12/11 \\
\hline
\QQ(6,3,-1) & L^+ = 7/5 \, & L^+ = 23/20 \\
\hline
\QQ(3,3,3,-1) & \,\text{\rm varying} \, & L^+ = 6/5 \\
\hline
\end{array}
$$
\end{theorem}
\par
\subsection{Parity given by torsion conditions}  \label{sec:g3parity}
Let $\QQ(k_1,k_2,k_3,-1)$ be a stratum in $\cE_3$.
We consider first a half-translation surface $(X,q)$ in $\QQ(k_1, k_2, k_3, -1)$ that is not hyperelliptic.
Later, in Lemma~\ref{le:nohypcomponent} we will see that in all strata this is 
generically the case. We
write ${\rm div}(q)_0 = k_1 z_1 + k_2 z_2 + k_3 z_3$ and let $p$ be the pole of $q$.
\par
Consider the canonical embedding $X \hookrightarrow \PP^2$. Take a general line $L$ 
passing through the pole of $q$ but not through its zeros. Then $L$ intersects $X$ 
at three points $r_1$, $r_2$ and $r_3$. By the long exact sequence of cohomology 
associated to the short 
exact sequence $$ 0\to \OO_{\PP^2}(-1)\to \OO_{\PP^2}(3) \to \OO_{X}(3) \to 0 $$
we obtain an isomorphism $H^0(\PP^2, \OO_{\PP^2}(3)) \cong H^0(X, \OO_{X}(3))$.
Note that $\OO_X(3) \cong \omega_X^{\otimes 3}$. Thus
there exists a unique plane cubic $E$ such that 
\be \label{eq:EcdotX}
 E \cdot X  = {\rm div}(q)_0  + r_1 + r_2 + r_3.
\ee
\par
\begin{proposition} \label{prop:Eirred}
Fix a half-translation surface $(X,q)$ in $\QQ(k_1,k_2,k_3,-1)$ and suppose it is not hyperelliptic. For a generic choice of $L$ the plane 
cubic $E$ is irreducible and $\divisor(q)_0$ is contained
in the smooth locus of $E$.
\end{proposition}
\par
To help the reader quickly grasp our idea, we postpone the proof of the above technical statements to Appendix~\ref{sec:g3local} and continue with the parity construction. From $r_1 + r_2 + r_3 \sim \OO_{E}(1)$ and 
\eqref{eq:EcdotX} we deduce that
\be \label{eq:g3torsion}
  3\LL(X,q) \sim \OO_{E}(3).
\ee
The key observation is that linear equivalence may or may not hold when dividing both sides by three, thus providing a parity to distinguish the two components. 
\par
\begin{proposition} \label{prop:g3parity}
In the above setting, the parity $\dim H^0(X,\LL(X,q)) = 2$ if 
and only if $\LL(X,q) \sim \OO_{E}(1)$ and the parity $\dim H^0(X,\LL(X,q)) = 1$ if 
and only if $\LL(X,q) \not\sim \OO_{E}(1)$ but $ 3\LL(X,q) \sim \OO_{E}(3).$
\par
Moreover, in a family of quadratic differentials $(X_t, q_t)$ 
with $X_t$ a smooth non-hyperelliptic curve for all $t \in \Delta$, the special 
member $(X_0,q_0)$ for $t=0$ has the same parity as the generic member. 
\end{proposition}
\par
\begin{proof} 
Since $E$ is smooth at the $z_i$ by Proposition~\ref{prop:Eirred}, 
the condition $\LL(X,q) \sim \OO_{E}(1)$ says for $(X,q) \in \QQ(9,-1)$ that the 
tangent line to $E$ at $z_1$ is a flex line, i.e. a tangent line with contact of order three at $z_1$. 
Since $E$ and $X$ have contact of order $9$ at $z_1$, either they possess the same flex line or they both 
have simple tangent lines at $z_1$. Rephrasing in the language of linear systems,  
$z_1$ being a flex of $X$ is equivalent to $h^0(X, 3z_1) = 2$. 
For $(X,q) \in \QQ(6,3,-1)$ the same argument works for the secant line $\overline{z_1z_2}$ 
being tangent to $X$ at $z_1$ and for $(X,q) \in \QQ(3,3,3,-1)$ the same argument works 
for the collinearity of $z_1$, $z_2$ and $z_3$.
\par
If $\Delta$ parameterizes a family of half-translation surfaces with generic member 
satisfying $h^0(X_t,\LL(X_t,q_t)) = 2$, then any special member $(X_0, q_0)$ 
also satisfies $h^0(X_0,\LL(X_0,q_0)) = 2$, because the dimension of the linear system is 
upper semicontinuous. Suppose for a family of half-translation surfaces in this 
stratum we have $h^0(X,\LL(X_t,q_t)) = 1$ for $t \neq 0$. Then we need to prove that 
they cannot specialize to $(X_0, q_0)$ with $h^0(X, \LL(X_0,q_0)) = 2$. Since the support 
of $\LL(X_t,q_t)$ is in the smooth
locus of $E_t$ by Proposition~\ref{prop:Eirred}, $\LL(X_t,q_t) \otimes \cO_E(-1)$
is a well-defined family of Cartier divisors for all $t \in \Delta$. It is well-known that  
two distinct torsion sections are disjoint in a family of elliptic fiberations (see e.g.\ \cite{Miranda}). In our situation, it implies that
if the torsion order in the Picard group of $E_t$ is three at a generic point $t$, it cannot drop
to one at a special point.  
\end{proof}
\par
Note that at this stage we have not yet shown that the dimension of the
locus with $\dim H^0(X,\LL(X,q)) = 2$ is the same as the corresponding stratum.  
The possibility of this locus being of smaller dimension (and thus necessarily contained in the hyperelliptic
locus) will be ruled out in the next section.
\par

\subsection{Construction of components} 
\label{sec:g3construction}

Suppose we deal with an exceptional stratum $\QQ(k_1, \ldots, k_n, -1)$
with $n$ zeros and one simple pole. Consider the moduli spaces $\barmoduli[{1,n+1}]^\reg$  and $\barmoduli[{1,n+1}]^\irr$
of stable pointed elliptic curves $(E,z_0,z_1,\ldots,z_n)$ with the additional property
that 
$$\sum_{i=1}^n \frac{k_i}{3} z_i \sim 3z_0 $$
in the case '$\irr$', respectively
$$\sum_{i=1}^n \frac{k_i}{3} z_i  \not\sim 3z_0 
\quad \text{but} \quad \sum_{i=1}^n k_i z_i \sim 9z_0$$
in the case '$\reg$'.
Mapping this tuple to $(E,z_1,\ldots,z_n)$ exhibits these moduli spaces
as finite connected unramified coverings of $\barmoduli[1,n]$. We embed
the elliptic curve as a plane cubic in $\PP^2$ using the linear system $|3z_0|$, i.e.\ such
that $3z_0 \sim \cO_E(1)$. Now choose moreover
a line $L$ in $\PP^2$, i.e.\ a section of $|3z_0|$ and let $L\cdot E= r_1+r_2+r_3$. 
Define two parameter spaces $B_{(k_1,\ldots,k_n,-1)}^\reg$ and $B_{(k_1,\ldots,k_n,-1)}^\irr$
parameterizing tuples $(E,z_0,z_1,\ldots,z_n,L)$. They are fiber bundles 
over $\barmoduli[{1,n+1}]^\reg$  and $\barmoduli[{1,n+1}]^\irr$ respectively. 
\par
We  let $$f: S^\reg_{(k_1,\ldots,k_n,-1)} \to  
B_{(k_1,\ldots,k_n,-1)}^\reg \quad \text{resp.} \quad f: S^\irr_{(k_1,\ldots,k_n,-1)} \to  
B_{(k_1,\ldots,k_n,-1)}^\irr$$
be the subspace of plane quartics $X$ whose fiber over $(E,z_0,z_1,\ldots,z_n,L)$
parameterizes those $X$ such that $X \cdot E = \sum k_i z_i + r_1+r_2+r_3$.  
\par
\begin{proposition} \label{prop:g3Sred}
For both indices $\irr$ and $\reg$ the parameter spaces 
$B_{(k_1,\ldots,k_n,-1)}$ are irreducible of dimension $n+2$ and 
the parameter spaces $S_{(k_1,\ldots,k_n,-1)}$ are irreducible of dimension 
$n+5$, which is the dimension of the incidence correspondence consisting of a point in  
$\PP\QQ(k_1,\ldots,k_n,-1)$ together
with a line in $\PP^2$ passing through the unique pole.
\par
Moreover, the generic quartic $X$ parameterized by $S_{(k_1,\ldots,k_n,-1)}$
is a smooth curve of genus $3$ in all the cases.
\end{proposition}
\par
\begin{proof}
The irreducibility of $\barmoduli[{1,n+1}]^\reg$ (resp. $\barmoduli[{1,n+1}]^\irr$)
for $n=1$ is a consequence of the irreducibility of the space of elliptic
curves with marked points together with the choice of a primitive $9$-torsion
point (resp. a primitive $3$-torsion point). The case $n>1$ is reduced
to the previous case by using the irreducibility of the fiber under the addition map
$(r_1,\ldots,r_n) \mapsto \sum_{i=1}^n k_i r_i \in E$. Irreducibility of
 $B_{(k_1,\ldots,k_n,-1)}^\reg$ and $B_{(k_1,\ldots,k_n,-1)}^\irr$
follows because they are fiber bundles over the previous parameter spaces. 
Tensoring the defining sequence of the ideal sheaf of $E$ with $\OO_{\PP^2}(4)$, 
we obtain an exact sequence
$$ 0\to \OO_{\PP^2}(1)\to \OO_{\PP^2}(4) \to \OO_{E}(4) \to 0. $$
The associated long exact sequence of cohomology shows that any two quadrics in a fiber of $f$ differ
by a section of $\OO_{\PP^2}(1)$. This shows first the irreducibility.
Moreover, both for $\bullet=\reg$ and $\bullet=\irr$ we have $\dim\moduli[{1,n+1}]^\bullet = n$, 
hence $\dim B_{(k_1,\ldots,k_n,-1)}^\bullet = n+2$ and since  
$h^0(\PP^2, \OO_{\PP^2}(1)) = 3$ we
finally obtain $\dim S_{(k_1,\ldots,k_n,-1)}^\bullet = n+5$.
\par
To prove generic smoothness we may fix a point in $B_{(k_1,\ldots,k_n,-1)}$
corresponding to a smooth elliptic curve $E$. The base points
of the linear system cutting out the $X$ with 
$X \cdot E = \sum_{i=1}^n k_i z_i + r_1+r_2+r_3$ are precisely the points $z_i$
and $r_i$. Moving $L$ we may move the $r_i$ and by Bertini's theorem
for a Zariski open set in each fiber over $\moduli[{1,n+1}]$ 
the curves $X$ are smooth except possibly at the $z_i$.
\par
We now argue that at the $z_i$ the generic $X$ is also smooth.
In fact, let $D = \sum_{i=1}^n k_i z_i$ and $F = r_1+r_2+r_3$, considered
as divisors on $E$. Since $D+F$ is a section of $\OO_E(4)$, the line bundle  $\OO_E(4)\otimes \OO_E(-D-F)$ 
is trivial. It implies that for all quartics in a fiber of $f$ 
their restrictions to $E$ are unique up to scalars. Lifting to $\PP^2$, we conclude that 
the defining homogeneous polynomial of $X$ can be expressed as 
$f_0 + l e$, where $f_0$ is the defining equation for a fixed $X_0$ in $S$, $l$ is an 
arbitrary linear form and $e$ is the equation of $E$. Since $E$ is
non-singular at $z_i$, for a generic choice
of $l$, the vanishing locus of $f_0 + l e$ is non-singular at $z_i$ as well.
\end{proof}
\par
\begin{proof}[Proof of Theorem~\ref{thm:mainexcg3}] The existence of
period coordinates shows that strata of quadratic differentials are smooth. Consequently, disregarding
subsets of complex codimension at least one does not change connectivity.
By Lemmas~\ref{le:noEsingcomp} and~\ref{le:nohypcomponent} below we may thus 
restrict the question on the number of components to the complement of the 
hyperelliptic locus and to the locus where $E$ is smooth.
\par
On this restricted locus by Proposition~\ref{prop:g3parity} the parity is deformation
invariant and hence there are at least two components. To each point $(X,q)$
in a stratum $\QQ(k_1,\ldots,k_n,-1)$ listed in $\cE_3$ and the additional choice 
of $L$ we associated at the beginning
of Section~\ref{sec:g3parity} an elliptic curve $E$ in $\PP^2$. 
By the homogeneity of an elliptic curve we may take $z_0 \in E$ so that the 
embedding is given by $|3z_0|$. Let $z_1,\ldots,z_n$ be the
points in $X \cdot E \setminus L \cdot E$. Finally, 
\eqref{eq:g3torsion} implies that 
$(E,z_0,z_1,\ldots,z_n,L,X)$ defines a point in $S_{k_1,\ldots,k_n}$ with
upper index either '$\irr$' or '$\reg$'. This space has the expected
dimension $2+\dim(\PP\QQ(k_1,\ldots,k_n,-1))$ by Proposition~\ref{prop:g3Sred}
and the irreducibility statement in that proposition completes the proof.
\end{proof}
\par
\begin{lemma} \label{le:noEsingcomp}
None of the strata in $\cE_3$ has a component such that for a generic half-translation
surface in the component the plane cubic $E$ defined
by \eqref{eq:EcdotX} is singular.
\end{lemma}
\par
\begin{proof} If this was the case, we could reconstruct this component by the
argument leading to Proposition~\ref{prop:g3Sred}, but replacing $\moduli[{1,n+1}]^\bullet$
by a configuration space for $n+1$ points on an irreducible rational nodal or cuspidal curve by Proposition~\ref{prop:Eirred}. If the points
$z_1,\ldots,z_n$ lie in the smooth locus of the rational curve, there is
still the torsion constraint and this parameter space is of dimension one
smaller than $\moduli[{1,n+1}]^\bullet$. If one of the $z_i$ lies at a node, 
the torsion constraint may no longer be well-defined, but $z_i$ being restricted
to a node imposes one more condition so that this  parameter space is still
of dimension at least one smaller than $\moduli[{1,n+1}]^\bullet$.
\par
The remaining dimension argument for the fibers of $B_{(k_1,\ldots,k_n,-1)}^\bullet 
\to \moduli[{1,n+1}]^\bullet$
and $f:S^\bullet_{(k_1,\ldots,k_n,-1)} \to  B_{(k_1,\ldots,k_n,-1)}^\bullet$ 
in Proposition~\ref{prop:g3Sred} still holds with $E$ singular but irreducible. In total, the locus with $E$ singular 
is thus too small to form a component of a stratum in $\cE_3$.
\end{proof}
\par
The same type of argument using the description of hyperelliptic curves
in Section~\ref{g=4: logsurface} shows the following lemma, 
whose proof will be given in the appendix. 
\par
\begin{lemma} \label{le:nohypcomponent}
None of the strata in $\cE_3$ has a component entirely contained in 
the hyperelliptic locus.
\end{lemma}
\par
\begin{lemma} \label{le:redvsirred}
The labeling of the components in Theorem~\ref{thm:mainexcg3} is consistent with ii), i.e.\ only the
component with $\dim H^0(X,\LL(X,q)) = 2$ is adjacent to $\QQ(8)$.
\end{lemma}
\par
\begin{proof}
Suppose there is a family of half-translation surfaces in $\QQ(9,-1)^{\reg}$ that 
degenerates to $(X,q)$ in $\QQ(8)$, where $\divisor(q) = 8z_1$. Since $h^0(X, 3z_1)$ is 
upper semicontinuous, $z_1$ is a Weierstrass point on $X$ and $\omega_X\sim 3z_1+ r$ 
for some $r \in X$. Then from  $8 z_1\sim \omega^{\otimes 2}_X$ we obtain 
that $2z_1 \sim 2r$. If $z_1=r$, then $\omega_X\sim 4z_1$ and $q$ is a global square, 
leading to a contradiction. If $z_1\neq r$, then $X$ is hyperelliptic. Again 
$\omega_X\sim 4z_1$ and we obtain the same contradiction. 
\par
From the above construction it is obvious that $\QQ(k_1,k_2,k_3,-1)^\irr$ is adjacent to
$\QQ(k_1 + k_2,k_3,-1)^\irr$ and that $\QQ(k_1,k_2,k_3,-1)^\reg$ is adjacent to
$\QQ(k_1 + k_2,k_3,-1)^\reg$, while the strata with different upper indices are not
adjacent by Proposition~\ref{prop:g3parity}. 
\end{proof}
\par
\begin{remark} {\rm For the irregular components there is another construction and irreducibility
proof. It relies on the following observation. We give details for $\QQ(9,-1)^\irr$ and the other cases can be dealt with similarly. 
\par
Let $z_1$ be the $9$-fold zero of $q$.
Let $B^\irr_{(9,-1)}$ parameterize tuples $(z_1, L_1, L_2)$ with $z_1$ a
point and the $L_i$ lines in $\PP^2$ such that $z_1\in L_1$ but 
$z_1\notin L_2$. Define $S^\irr_{(9,-1)}$ to be the parameter space of plane
quartics $X$ such that
i) $L_1$ is a flex line to $X$ at $z_1$ and ii) $L_2$ is a flex line to $X$ 
at the intersection point $r$ of $L_1$ and $L_2$. 
As above one checks that the parameter space $S^\irr_{(9,-1)}$ is irreducible of dimension $13$
and hence its quotient by the action of $\PGL(3)$ is of dimension $5$. 
One easily checks that an open subset of this quotient is indeed $\PP\QQ(9,-1)^{\irr}$.
}\end{remark}
\par

\subsection{The non-varying property} \label{sec:g3NV}
Throughout this section let $C$ be a \Teichmuller curve generated by $(X,q)$ in one 
of the strata in $\cE_3$ and let $\ol{C}$ be its closure over the compactified moduli space of curves. Note that using the 
language of linear series the components $\QQ(9,-1)^\irr$, $\QQ(6,3,-1)^\irr$ and 
$\QQ(9,-1)^\irr$ lie in the divisors $BN^1_{3,(3)}$, $BN^1_{3,(1,2)}$ and $BN^1_{3,(1,1,1)}$, respectively,   
after a suitable lift to the moduli space of curves with marked points. 
Consequently these divisors are the natural candidates 
to prove non-varying for the regular components of the strata in $\cE_3$. But we are still 
left with ruling out intersections at the boundary. This is the content of
Proposition~\ref{prop:periodandplumb} together with Corollary~\ref{component}. 
\par
\begin{lemma} \label{le:disjg3irred}
For each stratum $\QQ(k_1,\ldots,k_n,-1)$ in $\cE_3$, the closures of its two components 
do not intersect at a boundary point of a \Teichmuller curve in that stratum.
\end{lemma}
\par
\begin{proof}
Since the two components are of the same dimension, if they intersect, $\widetilde{\QQ}(k_1,\ldots,k_n,-1)$ would be singular at the common locus, contradicting Proposition~\ref{prop:periodandplumb}.
\end{proof}
\par
\begin{proposition} \label{prop:91disj}
If $C$ is generated by $(X,q)$ in $\QQ(9,-1)^\irr$, then $\ol{C}$ is disjoint with  
the hyperelliptic locus. 
\end{proposition}
\par
\par
\begin{proof}
Suppose $\divisor(q) = 9z_1 - p_1$ with $z_1$ a Weierstrass point. If $X$ is hyperelliptic, we 
have $\omega_X \sim 4z_1$. By $\omega_X^{\otimes 2}\sim 9z_1 - p_1$, we conclude that $z_1\sim p_1$, contradicting that $z_1\neq p_1$. 
\end{proof}
\par
\begin{proof}[Proof of Theorem~\ref{thm:execg3NV}]
For the components of  $\QQ(9,-1)$ this is straightforward from the divisor classes 
given in Section~\ref{sec:divclass}, from Lemma~\ref{le:disjg3irred}
and Proposition~\ref{prop:91disj}, respectively, since the 
stable curves parameterized by the boundary of $\ol{C}$ are irreducible.
\par
In the case of $\QQ(6,3,-1)^\reg$ we know from 
Lemma~\ref{le:disjg3irred} that $\ol{C} \cdot BN^1_{3,(2,1)} = 0$. 
By degree considerations, the only marked 
$\delta_1$-boundary divisor such a curve could intersect is the divisor 
$\delta_{1;\{2\}}$. It does not appear in \eqref{eq:PicBN3;21} and hence
$$ \ol{C}\cdot (-\lambda+3\omega_1 +\omega_2) = 0.$$
From this we can readily calculate the sum of Lyapunov exponents by Proposition~\ref{intersection}. 
\par
For $\QQ(6,3,-1)^\irr$, the lift of $C$ is contained in $BN^1_{3,(2,1)}$, so this 
divisor does not work for the disjointness property. 
Instead, one can find a divisor inside $BN^1_{3,(2,1)}$ disjoint with 
$\ol{C}$ and perform the intersection calculation. The trade-off is a 
tedious study of the Picard group of $BN^1_{3,(2,1)}$ as well as its 
singularities. Here we take an alternative approach, adapting the 
idea of \cite{yuzuo} and using a filtration of 
the Hodge bundle on $\ol{C}$. Since it has a different flavor and 
requires a technical setup, we will treat this approach independently 
in Appendix~\ref{sec:filtration}. 
\par
Finally, we consider $\QQ(3,3,3,-1)$. The boundary terms appearing in the divisor class 
of $BN^1_{3,(1,1,1)}$ in \eqref{eq:PicBNonly1s} are either irreducible
or, if over $\delta_1$ have any number but one of the three marked points
on the elliptic tail $X_1$. Since $\omega_{X}^{\otimes 2}$ has degree two restricted to the subcurve $X_1$, a boundary point of $\ol{C}$ in $\delta_1$ has precisely one of the marked zeros
in $X_1$ along with the pole. Consequently, all the boundary terms are irrelevant
for the intersection number and from this we can calculate the sum of Lyapunov exponents.
\end{proof}
\par

\section{Genus four: exceptional strata} 
\label{sec:g4exceptional}

In genus four only one of the stratum $\QQ(12)$ was claimed in \cite{lanneauENS} 
to have an exceptional number of components. Using the parity approach below, 
we discovered that a few more strata have indeed an exceptional number
of components. Let 
$$ \cE_4 = \{(12), (9,3), (6,6), (6,3,3), (3,3,3,3)\}.$$
Some of these strata obviously possess a hyperelliptic component according 
to Section~\ref{sec:backmoduli}. In addition, even for $\QQ(12)$ in \cite{lanneauENS} there is no
terminology distinguishing the two components. We thus keep the
labels '$\irr$' and '$\reg$' as to make the following result look consistent
with Theorem~\ref{thm:mainexcg3}.
\par
\begin{theorem} \label{thm:mainexcg4}
Each of the strata $\QQ(6,6)$, $\QQ(6,3,3)$ and $\QQ(3,3,3,3)$ has a hyperelliptic component, 
denoted by $\QQ(6,6)^{\hyp}$ etc. Besides the hyperelliptic components each of the strata in $\cE_4$ has exactly two additional
components distinguished as follows. Let $\LL(X,q) = {\rm div}(q)/3$, then the surface $(X,q)$ belongs to 
$$\QQ^\irr \quad \text{if} \quad \dim H^0(X,\LL(X,q)) = 2$$
and it belongs to  
$$\QQ^\reg \quad \text{if} \quad \dim H^0(X,\LL(X,q)) = 1.$$
\end{theorem}
\par
Many of the exceptional strata are non-varying. 
\par
\begin{theorem} \label{thm:execg4NV}
The exceptional strata in genus four with the 
exception of $\QQ(6,6)^\irr$, $\QQ(6,3,3)^\irr$ and $Q(3,3,3,3)^\irr$ are non-varying. 
The values are collected as follows:
$$\begin{array}{|c|l|l|l|}
\hline
& \text{\rm hyperell.\ comp.} \,& \text{\rm component} \, 
\QQ^\irr & \text{\rm component} \, \QQ^\reg \\
\hline
\QQ(12) & --- & L^+ = 11/7 \, & L^+ = 10/7  \\
\hline
\QQ(9,3) & --- & L^+ = 92/55 \, & L^+ = 82/55 \\
\hline
\QQ(6,6) &L^+ = 2 & \,\text{\rm varying}\, & L^+ = 3/2 \\
\hline
\QQ(6,3,3) &L^+ = 11/5 &\,\text{\rm varying}\, & L^+ = 31/20 \\
\hline
\QQ(3,3,3,3) &L^+ = 12/5 &  \,\text{\rm varying}\, & L^+ = 8/5 \\
\hline
\end{array}
$$
\end{theorem}
\par
\smallskip

\subsection{Parity given by torsion conditions}  
\label{sec:g4parity}

Let $\QQ(k_1,\ldots,k_n)$ be a stratum in $\cE_4$ and consider a half-translation
surface $(X,q)$ in this stratum that is not hyperelliptic and not in the 
Gieseker-Petri locus.
Later, in Lemma~\ref{le:noGKhypcomp} we will see that this holds for a generic 
surface in all components
but the hyperelliptic ones. Let  $X\hookrightarrow \PP^3$ be the canonical embedding.
Since $X$ is not in the Gieseker-Petri locus, the image of the canonical embedding
lies in a smooth quadric surface $Q \cong \PP^1 \times \PP^1$.
\par
Let $\cO_Q(1)$ denote the hyperplane class in $\PP^3$ restricted to $Q$. 
It has divisor class $(1,1)$ in the Picard group of $Q$. A general section $E$ of $|\cO_Q(2)|$ is an elliptic 
curve of degree four with divisor class $(2,2)$.
We have the following exact sequence 
$$ 0\to \OO_{Q}(-1) \to \OO_{Q}(2) \to \OO_{X}(2) \to 0, $$
hence $H^0(Q,\OO_{Q}(2))\to H^0(X,\OO_{X}(2))$ is an isomorphism. 
Note that $\OO_X(2)\cong \omega^{\otimes 2}_{X}$. Therefore, 
a bicanonical divisor of $X$ is uniquely cut out by a section $E$ of $|\OO_Q(2)|$. 
Therefore, we can choose $E$ uniquely corresponding to $q$ modulo scalars, that is
\be \label{eq:g4defE}
E \cdot X = \divisor(q).
\ee
\par
We say that $(E,\divisor(q))$ is {\em sufficiently smooth} if $E$ is reduced (possibly
reducible) and $\divisor(q)$ is supported on the smooth locus of $E$. In that case, we also say
that $X$ has a {\em sufficiently smooth parity curve}. This condition allows us to study torsion in the Picard group of divisors supported away from 
the singular locus of $E$. Since $X$ is a section of $|\OO_Q(3)|$, we have
\be \label{eq:g4torsion}
3\LL(X,q) \sim \OO_E(3).
\ee
in the Picard group of $E$. This linear equivalence may or may not hold when dividing both 
sides by three. Contrary to the case of genus three we cannot just focus on
irreducible $E$. Our substitute is the following lemma. As before, we postpone the proof of all
the technical results to the appendix. 
\par
\begin{lemma} \label{le:g4Esuffsmooth}
In the above setting, suppose $X$ is neither hyperelliptic nor Gieseker-Petri special. 
Then 
either $E$ is sufficiently
smooth or $\dim H^0(X,\LL(X,q)) = 1$.
\end{lemma}
\par
\begin{proposition}\label{prop:g4parity}
For $X$ neither hyperelliptic nor Gieseker-Petri special and
with sufficiently smooth parity curve, the parity
$$\dim H^0(X,\LL(X,q)) = 2 \quad \text{if and only if} \quad \LL(X,q) \sim \OO_{E}(1)$$ 
and 
$$\dim H^0(X,\LL(X,q)) = 1 \quad \!\!\text{if and only if} \quad\!\! \LL(X,q) \not\sim \OO_{E}(1) 
\quad\! \text{but} \quad \!  3 \LL(X,q) \sim \OO_{E}(3).$$
\par
Moreover, in a family of quadratic differentials $(X_t, q_t)$ 
with $X_t$ a smooth non-hyperelliptic curve, not Gieseker-Petri special and
with sufficiently smooth parity curve
for all $t \in \Delta$, the special member $(X_0,q_0)$ has the same parity as the generic member. 
\end{proposition}
\par
\begin{proof} By Riemann-Roch, $\dim H^0(X,\LL(X,q)) = 2$ is equivalent
to the existence of a unique section $C$ of $|\cO_Q(1)|$ in $Q$ and two points $r,s \in X$
such that $C \cdot X = \LL(X,q) +r +s$. Since $E \cdot X = 3\LL(X,q)$, 
this is equivalent to $C \cdot E = \LL(X,q)$. Since by definition
$C \cdot E \sim \cO_E(1)$, this equivalence proves the first statement.
\par
For a sufficiently smooth parity curves $E$, the line bundle $\LL(X,q) \otimes \OO_E(-1)$
is well-defined and is torsion of order either one or three that cannot jump
in a family. As in genus three, using upper semicontinuity of cohomology, we only 
need to show that in a family of half-translation surfaces  whose special fiber is no
longer sufficiently smooth, the dimension of $H^0(X,\LL(X,q))$ does 
not increase. This is precisely the content of Lemma~\ref{le:g4Esuffsmooth}.
\end{proof}
\par
Note that at this stage we have not yet shown that the dimension of the
locus with $\dim H^0(X,\LL(X,q)) = 2$ is the same as the corresponding 
stratum in $\cE_4.$ The case of smaller dimension (thus necessarily contained 
in the hyperelliptic
locus or in the Gieseker-Petri divisor) will be ruled out in the next section.
\par

\subsection{Construction of components} \label{sec:g4construction}

Let $B_{(k_1,\ldots,k_n)}^\irr$ resp.\ $B_{(k_1,\ldots,k_n)}^\reg$ be subset of the moduli 
space of stable elliptic curves  $(E,q_1,r_1,\ldots,r_n, D)$ with
$n+1$ marked points and additionally with the linear equivalence class
$D$ of a divisor of degree two satisfying the condition
\be \label{eq:g4:parcond}
\sum_{i=1}^n \frac{k_i}3 r_i \sim 2q_1+D
\ee  resp.\ satisfying 
$$\sum_{i=1}^n \frac{k_i}3 r_i \not \sim 2q_1+D \quad \text{but} \quad 
\sum_{i=1}^n k_i r_i \sim 6q_1+3D$$
and $2q_1 \not \sim D$. The linear systems $|2q_1|$ and $|D|$ define maps
$E \to \PP^1$ and by definition of the parameter spaces the product $E \to Q:= \PP^1 \times
\PP^1$ is an embedding. Instead of $D$ we could choose a point $q_2 \in E$
up to a two-torsion translation and let $D$ be the divisor class associated to $\cO_E(2q_2)$.
We will consider $Q \hookrightarrow \PP^3$ using the Veronese
embedding $|\cO(1,1)|$ and, as above, let $\cO_Q(k)$ be the restriction 
of $\cO_{\PP^3}(k)$ to $Q$. Thus, $E$ can be viewed as a section of $|\cO_Q(2)|$. 
\par
Let $f:S_{(k_1,\ldots,k_n)}^\irr \to B_{(k_1,\ldots,k_n)}^\irr$ 
and $f:S_{(k_1,\ldots,k_n)}^\reg \to B_{(k_1,\ldots,k_n)}^\reg$ be the fibrations
whose fiber over $(E,q_1,r_1,\ldots,r_n,D)$ consists of all sections 
$X$ of $|\cO_Q(3)|$ together with the quadratic differential $q$ 
obtained from restricting  $E$ to $X$.
\par
\begin{proposition} \label{prop:g4parametersp}
The parameter spaces $B^\bullet_{(k_1,\ldots,k_n)}$ are irreducible of dimension $n+1$ 
for both upper indices $\bullet = \irr$ and $\bullet = \reg$
and the parameter spaces $S^\bullet_{(k_1,\ldots,k_n)}$ are irreducible of dimension 
$n+5$, which equals $\dim \PP\QQ(k_1,\ldots,k_n)$.
\par
Moreover, a generic section $X$ parameterized by $S_{(k_1,\ldots,k_n)}$
is a smooth curve of genus four in all the cases.
\end{proposition}
\par
\begin{proof}
Forgetting the last marked point and using the interpretation of $D$ exhibit a finite, dominant map to a quotient of the moduli
space of elliptic curves with $n+1$ marked points by a finite group action. 
This proves the first dimension count. Tensoring the defining sequence for $E$ with $\OO_{Q}(3)$, 
we obtain an exact sequence
$$ 0\to \OO_{Q}(1) \to \OO_{Q}(3) \to \OO_{E}(3) \to 0, $$ 
and read off $h^0(Q,\cO_Q(1))=4$, which implies that the dimension of $S_{(k_1,\ldots,k_n)}$ is four 
larger than $B_{(k_1,\ldots,k_n)}$. The irreducibility of $B_{(k_1,\ldots,k_n)}$
for $n=1$ is a consequence of the irreducibility of the space of elliptic
curves with one marked point together with a choice of primitive $4$-torsion
point (respectively a primitive $12$-torsion point). The case $n>1$ is reduced
to the previous case by using the irreducibility of the fiber under the addition map
$(r_1,\ldots,r_n) \mapsto \sum_{i=1}^n k_i r_i \in E$. 
\par
The proof that a generic curve $X$ is smooth is completely parallel
to the case of genus three.
\end{proof}
\par
\begin{proof}[Proof of Theorem~\ref{thm:mainexcg4}] The existence of
period coordinates shows that strata are smooth. Consequently, disregarding
subsets of complex codimension at least one does not change connectivity.
By Lemmas~\ref{le:nog4Esingcomp} and~\ref{le:noGKhypcomp} below we may thus 
restrict the question on
the number of components to the complement of the hyperelliptic locus 
and the Gieseker-Petri locus and to half-translation surfaces with sufficiently
smooth parity curves. 
\par
On this complement by Proposition~\ref{prop:g4parity} the parity is deformation
invariant and hence there are at least two components for each stratum listed in $\cE_4$. 
Recall that to a point $(X,q)$ in a stratum $\QQ(k_1, \ldots, k_n)$ in $\cE_4$ we associated at the beginning
of Section~\ref{sec:g4parity} an elliptic curve $E$ and a map 
$E \to Q \cong \PP^1 \times \PP^1$. By the homogeneity of an elliptic curve
we may define $q_1 \in E$ so that the first projection is given by $|2q_1|$
and let $D$ be the pullback of $\cO_{\PP^1}(1)$ via the second projection.
Since $X \cdot E = \divisor(q)$, we can associate $X$ canonically
$n$ points on $Q$. Finally, equation~\eqref{eq:g4torsion} implies that 
$(X,E,q_1,D,\divisor(q))$ defines a point in $S_{k_1,\ldots,k_n}$ with
upper index either $\irr$ or $\reg$. The irreducibility statement in 
Proposition~\ref{prop:g4parametersp} completes the proof.
\end{proof}
\par 
\begin{lemma}  \label{le:nog4Esingcomp}
None of the strata in $\cE_4$ has a component such that for a generic half-translation
surface in that component the parity curve defined by \eqref{eq:g4defE} is singular. 
\end{lemma}
\par
\begin{proof}
The proof is identical to Lemma~\ref{le:noEsingcomp}. If $E$ is singular but sufficiently
smooth, the torsion constraint is still valid, so dimension of the base
space drops by one, since the parameter for the $j$-invariant of $E$ has been lost.
If $E$ is singular and at least one of the $z_i$ is at a node, the parity condition
\eqref{eq:g4:parcond} no longer makes sense, but this cannot compensate the
loss of at least two parameters for the $j$-invariant of $E$ and for the location
of the $z_i$ at a node.
\end{proof}
\par
\begin{lemma} \label{le:noGKhypcomp} 
Except the hyperelliptic components defined in Section~\ref{sec:hyploci} no component of a stratum in $\cE_4$
is contained in the hyperelliptic locus or contained in the Gieseker-Petri locus.
\end{lemma}
\par
The proof of Lemma~\ref{le:noGKhypcomp} will be given in the appendix. 
\par
\begin{remark} {\rm For the irregular components there is another construction and irreducibility
proof. It relies on the following observation. We give details for $\QQ(12)^\irr$ and
the other cases can be dealt with similarly. 
\par
Let $z_1$ be the $12$-fold zero of $q$.
Since $z_1$ is a Weierstrass point, $\omega_X\sim 4z_1+ r+s$ for some points
$r$ and $s$. Since $\omega_X^{\otimes 2} \sim 12 z_1$, we conclude that 
$4 z_1\sim 2r+2s$, hence $\omega_X\sim 3r+3s$. Conversely, if we can find 
$r,s\in X$ such that $\omega_X\sim 3r+3s$ and $\omega_X \sim 4p+r+s$, 
then one easily checks that $\omega^{\otimes 2}_X\sim 12z_1$. Now, 
consider $(X,\omega)$ in $\Omega\MM_4(3,3)^{\nonhyp}$ such that 
$\divisor(\omega) = 3r+3s$. The dimension of this locus modulo scalars is $8$. 
Let $z_1$ be a Weierstrass point of $X$ and choose a plane section through $z_1$ 
whose intersection with the canonical embedding of $X$ is $4z_1+x+y$. If we 
require $(x,y) = (r,s)$, this imposes two conditions, hence the dimension 
of the locus where $\omega_X\sim 3r+3s\sim 4z_1+r+s$ is equal to $8-2 = 6$, 
which equals $\dim \PP\QQ(12)$. 
}\end{remark}

\subsection{The non-varying properties} \label{sec:g4NV}

Throughout this section let $C$ be a \Teichmuller curve generated by $(X,q)$ in one 
of the strata in $\cE_3$ and let $\ol{C}$ be its closure in the compactified moduli space. The non-varying property of the
hyperelliptic components is an immediate consequence of Corollary~\ref{cor:hypNonvar}.
Note that using the 
language of linear series the components $\QQ(12)^\irr$, $\QQ(9,3)^\irr$, 
$\QQ(6,6)^\irr$ , $\QQ(6,3,3)^\irr$ and $\QQ(3,3,3,3)^\irr$ lie in the divisors 
$BN^1_{4,(4)}$, $BN^1_{4,(3,1)}$, $BN^1_{4,(2,2)}$, $BN^1_{4,(2,1,1)}$ and $BN^1_{4,(1,1,1,1)}$ 
respectively, after a suitable lift the moduli space of curves with marked points. 
Consequently these divisors are the natural candidates 
to prove non-varying for the regular components. But we are still 
left with ruling out intersections at the boundary. As for $g=3$, this is the content of
Proposition~\ref{prop:periodandplumb} together with Corollary~\ref{component}. 
\par
\begin{lemma} \label{le:disjg4irred}
For any stratum $\QQ(k_1,\ldots,k_n)$ in $\cE_4$, the closures of its two components 
$\widetilde{\QQ}(k_1,\ldots,k_n)^\irr$ and $\widetilde{\QQ}(k_1,\ldots,k_n)^\reg$
do not intersect at a boundary point of a \Teichmuller curve in the stratum lifted to $\barmoduli[4,n]$ by marking the zeros of $q$.
\end{lemma}
\par
\begin{proof}[Proof of Theorem~\ref{thm:execg4NV}]
By Lemma~\ref{le:disjg4irred} and the disjointness of the boundary
of the \Teichmuller curve with the boundary terms appearing in Lemma~\ref{le:PicBN4}
as established in Proposition~\ref{prop:disjointbydegree}, the non-varying property of the regular components in $\cE_4$ is a direct 
consequence of Proposition~\ref{intersection}. 
\par
For $\QQ(12)^{\irr}$ and $\QQ(9,3)^{\irr}$, due to the same issue as $\QQ(6,3,-1)^{\irr}$ in genus three, we will take an alternative approach to prove their non-varying property. The details will be given in Appendix~\ref{sec:filtration}. 
\end{proof}

\section{Genus one} \label{sec:g1}

\begin{theorem}  \label{thm:nvg1}
In genus one, the following strata are non-varying:
$$\begin{array}{|l|l||l|l|}
\hline
\text{\rm Stratum} \,& L^+ \,&\text{\rm Stratum} \,& L^+ \, \\
\hline
\QQ(n,-1^n) & L^+ = 2/(n+2) &
\QQ(n,1,-1^{n+1}) & L^+ = 2/(n+2) \\
\hline
\end{array}
$$
\par
Moreover, the stratum $\QQ(4,2,-1^6)$ is varying.
\end{theorem}
\par
An example justifying that  $\QQ(4,2,-1^6)$ is varying will be given in 
the appendix.
\par
\begin{proof}
Let $C$ be a Teichm\"uller curve generated by a half-translation surface in  $\QQ(n, -1^n)$
and lift $C$ to $\barmoduli[{1,n+1}]$ by marking all the zeros and poles of $q$. For any degenerate fiber over $\ol{C}$ 
the zero $z_1$ does not lie in a component of genus zero, because $\omega_{X}^{\otimes 2}$ restricted to this rational tail has degree $-2$. 
Hence we conclude that 
$$\ol{C} \cdot \sum_{1\in S} \delta_{0;S} = 0,$$ 
i.e.\ $C\cdot (\psi_1 - \lambda) = 0$. 
Using Proposition~\ref{intersection} we read off $L^{+}(C) = \frac{2}{n+2}. $
\par
For the stratum  $\QQ(n, 1, -1^{n+1})$ the same argument works without any change.
\end{proof}
\par
We suspect that all the strata  $\QQ(a, b, -1^{a+b})$ for $a \geq b>1$ 
are varying. Genus one strata are also the first testing ground for
finer asymptotic questions on the behavior of Lyapunov exponents as
the number of poles grows. This is beyond the scope of this paper. Nevertheless, the above method already provides an upper bound of $L^{+}$ for all 
strata in genus one. 
\par
\begin{theorem}
Let $C$ be a Teichm\"uller curve in the stratum $\QQ(d_1, \ldots, d_r)$ in genus one. Suppose $d_1$ is the largest order of zeros of $q$. Then 
$L^{+}(C)\leq \frac{2}{d_1 + 2}$. 
\end{theorem}

\begin{proof}
As in the preceding proof,  we now have an inequality 
$$\ol{C} \cdot \sum_{1\in S} \delta_{0;S} \geq 0,$$
since $C$ is not entirely contained in the boundary of $\barmoduli[{1,r}]$. 
Using the relation of divisor classes on $\barmoduli[{1,r}]$, we conclude that $C\cdot (\psi_1 - \lambda) \geq 0$. Then by Proposition~\ref{intersection} the desired upper bound follows right away. 
\end{proof}

\section{Genus two}

Note that any Riemann surface of genus two is hyperelliptic, but being a hyperelliptic 
half-translation surface requires more, i.e. the quadratic differential is invariant under 
the hyperelliptic involution. Throughout we denote by $a'$ the conjugate of $a$
under the hyperelliptic involution. 
\par
Several strata of (non-hyperelliptic) quadratic differentials with a small number
of zeros are indeed empty, e.g.\ $\QQ(4)$, $\QQ(3,1)$, $\QQ(2,2)^{\nh}$, $\QQ(2,1,1)^{\nh}$ 
and $\QQ(1,1,1,1)^{\nh}$. This was shown by \cite{lahypcomp} and can also be quickly retrieved
in our language. For example, for $(X,q) \in \QQ(1,1,1,1)$, let 
$\divisor(q) = z_1 + z_2 + z_3 + z_4\sim z_1 + z_2 + z'_1 + z'_2$. Hence 
$z_3 + z_4\sim z'_1 + z'_2$. Up to relabeling, we conclude that $z_1$ and $z_2$ are conjugate and
so are $p_3$ and $p_4$. Consequently $(X,q)$ is a hyperelliptic half-translation surface, hence 
$\QQ(1,1,1,1)^{\nh}$ is empty.
\par
\begin{theorem}  \label{thm:nvg2}
In genus two, besides the components of hyperelliptic flat surfaces
and the empty strata mentioned above, the following strata are non-varying: 
$$\begin{array}{|l|l||l|l|}
\hline
\text{\rm Stratum} \,& L^+ \,&\text{\rm Stratum} \,& L^+ \, \\
\hline
\QQ(5,-1) & L^+ = 6/7 &
\QQ(4,2,-1,-1) & L^+ = 5/6 \\
\hline
\QQ(6,-1,-1)^\nh &L^+ = 3/4  &
\QQ(3,3,-1,-1)^\nh & L^+ = 4/5 \\
\hline
\QQ(4,1,-1) &L^+ = 1  &
\QQ(3,1,1,-1) & L^+ = 16/15 \\
\hline
\QQ(3,2,-1) &L^+ = 9/10 & 
\QQ(2,2,1,-1) & L^+ = 1 \\
\hline
\QQ(7,-1,-1,-1) & L^+ = 2/3&
\QQ(4,3,-1,-1,-1) & L^+ = 11/15 \\
\hline
\QQ(5,1,-1,-1) & L^+ = 6/7 &
\QQ(3,2,1,-1,-1) & L^+ = 9/10 \\
\hline
\end{array}
$$
All the other strata in genus two of dimension less than or equal to seven
are varying.
\end{theorem}
\par
Examples of square-tiled surfaces certifying the varying strata are listed in the appendix. 
\par
The proof of Theorem~\ref{thm:nvg2} uses three types of divisors, grouped in the following sections.
\subsection{Irreducible degenerations}
\begin{lemma} For a Teichm\"uller curve $C$ generated by $(X,q)$ 
in one of the strata $\QQ(5,-1)$ or $\QQ(4,1,-1)$ 
all the stable curves parameterized
by points in the boundary of $C$ are irreducible. Consequently,
Theorem~\ref{thm:nvg2} follows in these cases from Lemmas~\ref{slope10} and \eqref{eq:Lcs}.
\end{lemma}
\par
\begin{proof} This follows immediately from Corollary~\ref{component}.
\end{proof}

\subsection{The Weierstrass divisor}
\begin{lemma} For a Teichm\"uller curve $C$ generated by $(X,q)$ 
in one of the strata $\QQ(3,2,-1)$, $\QQ(7,-1,-1,-1)$, $\QQ(5,1,-1,-1)$ or $\QQ(6,-1,-1)^{\nh}$, lift $C$ to $\barmoduli[2,1]$ using the
first zero. Then $\ol{C}$ does not intersect the divisor of Weierstrass points, i.e.\ $W \cdot \ol{C} = 0$.
\end{lemma}
\par
Together with \eqref{eq:relg2} and \eqref{eq:Lcs} we can solve  
$W \cdot \ol{C} = 0$ for the sum of Lyapunov exponents and
this completes the theorem in these cases.
\par
\begin{proof} Suppose to the contrary $X$ is in the intersection.
In the first case $\divisor(q) = 3z_1 + 2z_2 - p_1$ together with
$2z_1\sim \omega_X$ this implies  $z_1 + p_1\sim 2z_2$ for $X$ irreducible,
impossible for $p_1\neq p_3$.
This reasoning is also valid if $X$ is reducible, consisting of 
two elliptic components $X_1, X_2$, where $X_1$ contains $z_1, p_1$, $X_2$ 
contains $p_2$ and they intersect at a node $t$. Analyzing the admissible 
double cover, we see that Figure~\ref{cap:admdouble.eps} is
the only possibility and we obtain $2z_1\sim 2t$ in $X_1$. 
\begin{figure}[hb]
    \centering      
      \psfrag{t}{$t$}
      \psfrag{z1}{$z_1$} 
      \psfrag{z2}{$z_2$}
      \psfrag{g=1}{$g=1$}
      \psfrag{P1}{$\mathbb{P}^1$}
      \includegraphics[width=9cm]{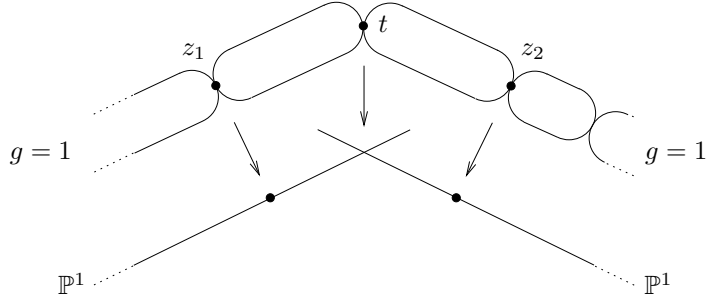}
\caption{
Admissible double cover for a reducible degeneration with $z_i$
a Weierstrass point
} \label{cap:admdouble.eps}
\end{figure}
Restricting $q$ to 
$X_1$, we also have $3z_1 - p_1\sim 2t$. Hence we conclude that 
$z_1\sim z_3$, leading to a contradiction. 
\par
In the second case, $\divisor(q) = 7z_1 - p_1 - p_2 - p_3$
and  $2z_1\sim \omega_X$, hence $4z_1\sim \divisor(q)$. 
Consequently $3z_1 \sim p_1+ p_2 + p_3$. 
By Riemann-Roch, we know $h^0(X, 3z_1) = h^0(X, 2z_1) = 2$.
This shows that $z_1$ is a base point of the linear system $|3z_1|$. 
But $p_1+p_2 + p_3 $ is a section of this linear system, 
contradicting that $z_1\neq p_1, p_2, p_3$. 
\par
In the third case, $\divisor(q) = 5z_1 + z_2 - p_1 - p_2$
and $4z_1\sim \omega^{\otimes 2}_{X}$ 
imply $p_1 + p_2\sim z_1 + z_2$, hence $z_2 \sim z'_1 = z_1$, impossible 
for $z_1\neq z_2$.
\par
In the last case, $\divisor(q) = 6z_1 - p_1 - p_2$ and $4z_1\sim \omega^{\otimes 2}_{X}$ 
imply $p_1 + p_2\sim 2z_1$, hence $(X,q)$ is a hyperelliptic half-translation surface, contradicting 
Proposition~\ref{prop:hypopen}. 
\end{proof}
\par

\subsection{The Brill-Noether divisor $BN_{2,(1,1)}$}
\begin{lemma} For a Teichm\"uller curve $C$ generated by $(X,q)$ 
in one of the strata  $\QQ(3,1,1,-1)$,
$\QQ(2,2,1,-1)$, $\QQ(4,3,-1,-1,-1)$, $\QQ(3,2,1,-1,-1)$, 
$\QQ(4,2,-1,-1)$ or $\QQ(3,3,-1,-1)^\nh$, lift $C$ to $\barmoduli[2,2]$ using the first 
two zeros. Then $\ol{C}$ does not intersect the Brill-Noether   divisor, 
i.e.\ $BN_{2,(1,1)} \cdot \ol{C} = 0$.
\end{lemma}
\par
Next, one checks that in each of the cases, $\ol{C}$ does not
intersect the boundary terms appearing in \eqref{eq:BNg2}. Thus
one can solve for the sum of Lyapunov exponents.
\par
\begin{proof}
In the first case $\divisor(q) = 3z_1 + z_2 + z_3 - p_1$. We claim 
that $z_1$ and $z_2$ are not conjugate. Otherwise $z_1+z_2\sim \omega_X$ 
and consequently $z_2+p_1\sim z_1+z_3$. Hence $p_1$ and $z_1$ are both conjugate to $z_2$, 
impossible for $p_1\neq z_1$. This reasoning is also valid if $X$ is 
reducible, consisting of two elliptic components $X_1, X_2$, where $X_1$ 
contains $z_1, p_1$, $X_2$ contains $z_2, z_3$ and they intersect at a 
node. Analyzing the double cover, we see that it is impossible for 
$z_1$ and $z_2$ to have the same image. 
\par
The strata $\QQ(2,2,1,-1)$, $\QQ(4,2,-1,-1)$ and $\QQ(3,2,1,-1,-1)$ 
can be solved by the same argument.
\par
In the case $\QQ(4,3,-1,-1,-1)$ let 
$\divisor(q) = 4z_1 + 3z_2 - p_1 - p_2 - p_3$. 
If $z_1, z_2$ are conjugate, then $z_1 + z_2\sim \omega_X$ and 
$2z_1 + 2z_2\sim \divisor(q)$. Consequently we have 
$2z_1 + z_2 \sim p_1 + p_2 + p_3$. 
By Riemann-Roch, we know $h^0(X, 2z_1+z_2) = h^0(X, z_1+z_2) = 2$, 
hence $z_1$ is a base point of the linear system $|2z_1+z_2|$. 
But $ p_1 + p_2 + p_3$ is also a section of this linear system, 
contradicting that $z_1\neq p_1, p_2, p_3$. 
\par
Finally in the case $\QQ(3,3,-1,-1)^\nh$, 
let $\divisor(q) = 3z_1 + 3z_2 - p_1 - p_2$. 
If $z_1$ and $z_2$ are conjugate, then $2z_1 + 2z_2\sim \divisor(q)$ and 
$p_1 + p_2 \sim z_1 + z_2$, contradicting that $(X,q)$ is a non-hyperelliptic 
half-translation surface by Proposition~\ref{prop:hypopen}.
\end{proof}
\par


\section{Genus three: non-exceptional strata}

\begin{theorem} \label{thm:nvg3}
In genus three, besides the components of hyperelliptic flat surfaces
and those non-varying exceptional components, the following strata are non-varying: 
$$\begin{array}{|l|l||l|l|}
\hline
\text{\rm Stratum} \,& L^+ \,&\text{\rm Stratum} \,& L^+ \, \\
\hline
\QQ(8) & L^+ = 6/5  &
\QQ(8,1,-1) & L^+ = 6/5 \\
\hline
\QQ(7,1) & L^+ = 4/3  &
\QQ(7,2,-1) & L^+ = 7/6 \\
\hline
\QQ(6,2)^\nh &L^+ = 5/4  &
\QQ(5,4,-1) & L^+ = 25/21 \\
\hline
\QQ(5,3) &L^+ = 44/35  &
\QQ(10,-1,-1)^\nh & L^+ = 1 \\
\hline
\QQ(4,4) &L^+ = 4/3 & 
\QQ(5,3,1,-1) & L^+ = 44/35 \\
\hline
\QQ(6,1,1)^\nh & L^+ = 17/12  &
\QQ(4,3,2,-1) & L^+ = 37/30 \\
\hline
\QQ(5,2,1)  &L^+ = 19/14  &
\QQ(3,3,1,1) & L^+ = 22/15 \\
\hline
\QQ(4,3,1) &L^+ = 7/5  &
\QQ(3,2,2,1) & L^+ = 7/5 \\
\hline
\QQ(4,2,2) & L^+ = 4/3 &
\QQ(7,3,-1,-1) & L^+ = 16/15 \\
\hline
\QQ(3,3,2)^\nh & L^+ = 13/10 &
& \\\hline
\end{array}
$$
All the other strata in genus three of dimension less than or equal to eight
are varying.
\end{theorem}
\par
Examples of square-tiled surfaces certifying the varying strata are listed in the appendix. 
\par
Again, we prove Theorem~\ref{thm:nvg3} by the disjointness of Teichm\"uller curves with various divisors on moduli spaces of genus three curves with marked points. 
\par
\subsection{The Weierstrass divisor} 
\begin{lemma} For a Teichm\"uller curve $C$ generated by $(X,q)$ 
in $\QQ(8)$, $\QQ(7,1)$, $\QQ(8,1,-1)$ or $\QQ(10, -1, -1)^{\nh}$, lift $C$ by the first zero to $\barmoduli[3,1]$. Then the intersection
with the Weierstrass divisor $ W\cdot \ol{C} = 0.$
\end{lemma}
\par
\begin{proof} Suppose $(X,q)$ lies in the intersection of $\ol{C}$ with $W$. 
For the first stratum, $\divisor(q) = 8z_1$ but by definition $q$ is not a global 
square of an abelian differential, namely $4z_1 \not\sim \omega_X$. 
This implies that $h^0(X, 4z_1) = 2$. If $h^0(X, 3z_1) = 2$, 
then $h^0(X, \omega_X(-3z_1)) = 1$, hence 
$\omega_X \sim 3z_1 + r$ for some $r \neq z_1$. This holds even 
if $X$ is stable, since by Corollary~\ref{component} it is irreducible and consequently 
has a canonical morphism to $\PP^2$ with $z_1$ as a flex point. 
Then $6z_1 + 2r \sim 8z_1$, $2r \sim 2z_1$, hence $X$ is hyperelliptic. Since $z_1$ is a Weierstrass point, 
we conclude that $4z_1\sim \omega_X$, contradicting that $q$ is not a global square. 
\par
For the second stratum, we have $\divisor(q) = 7z_1 + z_2\sim \omega^{\otimes 2}_X$ for $z_1\neq z_2$. 
Note that $3z_1 + z_2 \not\sim \omega_{X}$, since otherwise $z_1\sim z_2$, impossible. 
The hypothesis implies $h^0(X, 4z_1) = 2$. 
If $h^0(X, 3z_1) = 2$, then $\omega_X\sim 3z_1 + r$ for some $r\neq z_1, z_2$. 
This holds even if $X$ is degenerate, since by Corollaries~\ref{signature} and \ref{component} 
$X$ is irreducible and 
consequently has a canonical embedding in $\PP^2$ with $z_1$ as a flex point. 
Consequently, $6z_1 + 2r\sim 7z_1 + z_2$ implies that $X$ is hyperelliptic and 
$z_1$ and $z_2$ are conjugate. But then 
$2z_1 + 2z_2\sim \omega_X\sim 3z_1 + r$ and $z_2$ has to be a Weierstrass point, 
contradicting that it is conjugate to $z_1$. 
\par
For the third stratum, we have $\divisor(q) = 8z_1 + z_2 - p_1\sim \omega^{\otimes 2}_X$. 
First consider the case when $X$ is irreducible. Note that $z_1$ cannot be a Weierstrass point. Otherwise we would have 
$\omega_X\sim 3z_1+r$, hence $2r + p_1\sim 2z_1 + z_2$. 
Then $\omega_X\sim 2z_1 + z_2 + s$ and $z_1 + r\sim z_2 + s$. If $X$ is 
hyperelliptic, then $\omega_X\sim 4z_1$ and $z_2\sim p_1$, impossible. 
Otherwise we have $r = z_2$ and $s = z_1$. Then $z_2 + p_1\sim 2z_1$, which 
still implies that $X$ is hyperelliptic, leading to the same contradiction. 
\par
If $X$ is reducible, by Corollaries~\ref{signature} and \ref{component}, 
it consists of two irreducible components $X_0$ and $X_2$ of genus $0$ and $2$, respectively, 
meeting at two nodes $t_1$ and $t_2$, where $X_0$ contains $z_2, p_1$ and $X_2$ 
contains $z_1$. If $h^0(X, 3z_1) = 2$, by analyzing the associated triple admissible cover, we 
obtain that $3z_1\sim 2t_1 + t_2$ in $X_2$ (up to relabeling $t_1, t_2$). 
Since $8z_1 + z_2 - p_1\sim \omega^{\otimes 2}_{X}$, we have 
$8z_1\sim \omega^{\otimes 2}_{X_2} ( 2t_1 + 2t_2)\sim 2z_1+2z_1' + 3z_1 + t_2$, 
hence $2t_1 + t_2\sim 3z_1\sim t_2 + 2z'_1$ and $t_1, z'_1$ are both Weierstrass points
of $X_2$. Then $z_1$ is also a Weierstrass point. Since $2z_1\sim 2t_1$, then we get 
$z_1\sim t_2$, impossible for $z_1$ being contained in the smooth locus of $X$. 
\par
For the last stratum, we have $\divisor(q) = 10z_1 - p_1 - p_2 \sim \omega^{\otimes 2}_X$. 
First consider the case when $X$ is irreducible. We can write 
$\omega_X\sim p_1 + p_2 + r + s$
for some $r,s \in X$. If $z_1$ is a Weierstrass point, then $\omega_X\sim 3z_1 + t$. Hence 
$3\omega_X \sim 9z_1 +  3t \sim 10z_1 + r + s$ and consequently $3t\sim z_1 + r + s$. 
Then $\omega_X\sim z_1 + r + s + u \sim p_1 + p_2 + r + s$ 
and $z_1 + u\sim p_1 + p_2$. We conclude that $X$ is hyperelliptic, contradicting that 
this is a non-hyperelliptic stratum and Proposition~\ref{prop:hypopen}.
\par
If $X$ is reducible, by Corollaries~\ref{signature} and \ref{component}, it consists 
of two irreducible components $X_0$ and $X_3$ of genus $0$ and $3$, respectively, 
meeting at a node $t$ such that $X_0$ contains $p_1, p_2$ and $X_3$ contains $z_1$. 
If $h^0(X, 3z_1) = 2$, by analyzing the associated triple admissible cover, we 
see that $h^0(X_3, 3z_1) = 2$, hence $\omega_{X_3}\sim 3z_1 + s$ 
for some $s\in X_3$.  
Since $10z_1\sim \omega^{\otimes 2}_{X_3}(2t) \sim 6z_1 + 2s + 2t$, we conclude that $4z_1\sim 2s + 2t$. 
We may write $\omega_{X_3}\sim s + t + u + v$. Then $t + u + v\sim 3z_1$ and $u + v+ z_1\sim 2s + t$. 
Note that $z_1\neq t$. If $z_1 = s$, then 
$2z_1\sim 2t$ and $X_3$ is hyperelliptic, contradicting the non-hyperelliptic assumption. 
If $z_1\neq s$, we have 
$\omega_{X_3}\sim u + v + z_1 + w$ and $s + t\sim z_1 + w$. Since $z_1\neq t, s$, 
the curve $X_3$ is hyperelliptic and we deduce the same contradiction. 
\end{proof}
\par
For a \Teichmuller curve $C$ in the first two strata, since the stable curves parameterized by $\ol{C}$ 
are irreducible, the theorem follows in these cases from the above lemma together with \eqref{eq:PicWP} and
Proposition~\ref{intersection}. For the last two strata $\ol{C}$ 
hits none of the boundary terms in the presentation \eqref{eq:PicBNonly1s} of $W$ by 
Proposition~\ref{prop:disjointbydegree}. Using the trivial 
intersection with $W$ we can thus calculate the values of $L^+$. 

\subsection{The Brill-Noether divisor $BN^1_{3,(2,1)}$} 
\begin{lemma} For a Teichm\"uller curve $C$ generated by $(X,q)$ 
in one of the strata $\QQ(6,2)^\nh$, $\QQ(6,1,1)^{\nh}$, $\QQ(5,3)$, $\QQ(5,2,1)$, 
$\QQ(4,4)$, $\QQ(4,3,1)$, $\QQ(5,4,-1)$, $\QQ(5,3,1,-1)$, $\QQ(7,2,-1)$ or $\QQ(7,3,-1,-1)$, 
lift $C$ to $\barmoduli[3,2]$ using the first two zeros of $q$. Then the intersection with the Brill-Noether divisor 
$BN^1_{3,(2,1)}\cdot \ol{C} = 0.$
\end{lemma}
\par
\begin{proof} The proofs for $\QQ(6,2)^{\nh}$ and $\QQ(6,1,1)^{\nh}$ are completely analogous, 
so we deal with $\QQ(6,2)^{\nh}$ only. 
Suppose that $(X,q)$ is in the intersection with the Brill-Noether divisor and suppose moreover that $X$ is
irreducible first. We have $\divisor(q) = 6z_1 + 2z_2\sim \omega^{\otimes 2}_X$ and 
$3z_1 + z_2 \not\sim \omega_X$ for $z_1\neq z_2$. If $h^0(X, 2z_1 + z_2) = 2$, 
then $\omega_X\sim 2z_1 + z_2 + r$ and $ 2z_1\sim 2r$ for some $r\neq z_1$, 
hence $X$ is hyperelliptic and $z_1, z_2$ are Weierstrass points. This contradicts 
the assumption that $C$ lies in the non-hyperelliptic component and Proposition~\ref{prop:hypopen}.
\par
When $X$ is reducible, by Corollaries~\ref{signature} and \ref{component}, it consists of 
a one-nodal union of two components $X_1$ and $X_2$ with genus $1$ and $2$, respectively, where 
$z_1\in X_2$ and $z_2\in X_1$. Denote the node by $t$. If $h^0(X, 2z_1 + z_2) = 2$, 
by analyzing the associated triple admissible cover, we conclude that $z_1$ and $t$ are 
both Weierstrass points in $X_2$. Since the restrictions of $\omega_X$ to $X_1$ and $X_2$ are 
$\OO_{X_1}(t)$ and $\OO_{X_2}(2z_1+t)$ respectively, squaring it and comparing to 
$6p_1 + 2p_2$, we conclude that $2t\sim 2z_2$ on $X_1$ and $2t\sim 2z_1$ on $X_2$. 
We thus obtain that $X$ is contained in the closure of hyperelliptic curves and $z_1, z_2$ 
are both ramification points of the double admissible cover. 
It contradicts again the assumption that $C$ lies in the non-hyperelliptic component. 
\par 
The proofs for $\QQ(5,3)$, $\QQ(5,2,1)$ and $\QQ(7,2,-1)$ are completely analogous, so we prove for $\QQ(5,3)$ only. Suppose that $X$ is in the intersection with the Brill-Noether divisor. We have $\divisor(q) = 5z_1 + 3z_2 \sim \omega^{\otimes 2}_X$. By Corollaries~\ref{signature} and \ref{component}, $X$ is irreducible. If $h^0(X, 2z_1 + z_2) = 2$, we would have $\omega_X\sim 2z_1 + z_2 + r$ for some $r\in X$. It follows that $2r\sim z_1 + z_2$, hence 
$X$ is hyperelliptic and $z_1$ and $z_2$ are conjugate. This implies that 
$\omega^{\otimes 2}_X\sim 4z_1 + 4z_2\sim 5z_1 + 3z_2$, hence $z_1\sim z_2$, impossible for $z_1\neq z_2$. 
\par
The proofs for $\QQ(4,4)$, $\QQ(4,3,1)$ and $\QQ(5,4,-1)$ are similar, so we prove for $\QQ(4,4)$ only. 
Suppose a half-translation surface $(X, q)$ parameterized by $\ol{C}$ also lies in $BN^1_{3,(2,1)}$. 
We have $\divisor(q) = 4z_1 + 4z_2\sim \omega^{\otimes 2}_X$. Moreover, $h^0(X, 2z_1 + z_2) = 2$ 
and $\omega_X\not\sim 2z_1 + 2z_2$. First consider the case $X$ is irreducible. We have $\omega_X\sim 2z_1 + z_2 + r$ 
for some $r\neq z_2$ and $2z_2\sim 2r$. Hence $X$ is hyperelliptic and $z_2, r$ are 
both Weierstrass points. But this is impossible.
\par 
If $X$ is reducible, by Corollaries~\ref{signature} and \ref{component}, 
it consists of two irreducible components $X_1$ and $X_2$ of genus $g_1$ and $g_2$, 
containing $z_1$ and $z_2$, respectively, joined at $n$ nodes $t_1,\ldots, t_n$ such 
that $2g_i - 2 + n  =  2$ for $i = 1, 2$. Moreover, $(g_1, g_2, n)$ is either 
$(1,1,2)$ or $(0,0,4)$. Below we show that both cases are impossible. 
\par
For the former case, analyzing the triple cover given by $|2z_1 + z_2|$ as shown in 
Figure~\ref{cap:admtriple.eps}, we 
conclude that $t_1 + t_2\sim 2z_1$ in $X_1$.
\begin{figure}[hb]
    \centering
      \psfrag{z1}{$z_1$} 
      \psfrag{z2}{$z_2$}
      \psfrag{x1}{$X_1$} 
      \psfrag{x2}{$X_2$}
      \psfrag{g=1}{$g=1$}
      \psfrag{P1}{$\mathbb{P}^1$}
      \includegraphics[width=9cm]{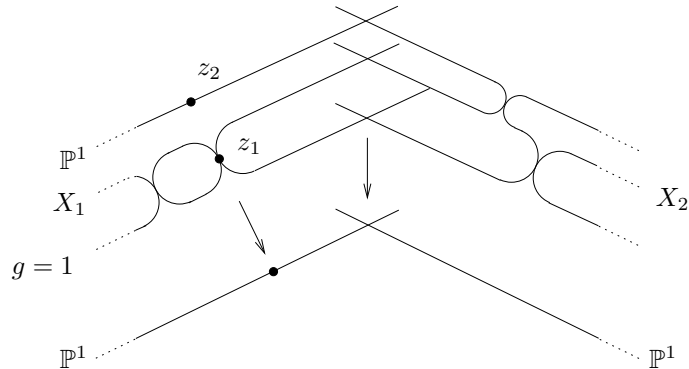}
\caption{
Admissible double cover for a reducible degeneration of type $(g_1, g_2, n) = (1,1,2)$
} \label{cap:admtriple.eps}
\end{figure}
Consider the canonical model of a generic curve $X'$ in $C$ degenerating to $X$. 
By the description in Section~\ref{logsurface}, if $X'$ is non-hyperelliptic, the log surface 
$(S, X')$ degenerates to $(S_{2,2}, X)$, where $S$ is the Veronese $\PP^2$ and $S_{2,2}$ 
is a cone over the one-nodal union of two conics in $\PP^5$. If $X'$ is hyperelliptic, 
the log surface $(S_{4}, X')$ degenerates to $(S_{2,2}, X)$, where $S_{4}$ is a cone 
over a rational normal quartic in $\PP^5$. The two components $S_1, S_2$ of $S_{2,2}$ 
contain $X_1, X_2$, respectively. The common ruling of $S_1, S_2$ contains the 
two nodes $t_1, t_2$. There exists a hyperplane section $Q'$ of $S$ or $S_4$ such 
that $Q\cdot X' = 4z'_1 + 4z'_2$. The limit of $Q'$ is a hyperplane section 
$Q$ of $S$ satisfying $Q\cdot X = 4z_1 + 4z_2$. Suppose that $Q_1$ and $Q_2$ 
are the components of $Q$ lying in $S_1$ and $S_2$, respectively. 
Then $Q_i\cdot X_i = 4z_i$ for $i = 1, 2$. Since $2z_1\sim t_1 + t_2$, the ruling 
$L_1 = \overline{vp_1}$ is tangent to $X_1$ at $z_1$. If $Q_1$ is smooth, i.e.\ if
the ruling does not pass through $v$, then $Q_1$ has $L_1$ as its tangent line at $z_1$. 
But $Q_1\cdot L_1 = 1$, leading to a contradiction. 
Hence we conclude that $Q_1 = 2L_1$ is a double ruling. The other conic $Q_2$ now also passes 
through the vertex $v$, because the hyperplane cutting out $Q$ does. Hence $Q_2$ 
consists of two rulings in $S_2$. Since $Q_2\cdot X_2 = 4z_2$, the only possibility 
is that $Q_2$ is a double ruling $2L_2$, where $L_2 = \overline{vp_2}$ is tangent to $X_2$ at $z_2$. 
We thus conclude that $t_1 + t_2 \sim 2z_2$ in $X_2$. Therefore, $2z_1 + 2z_2$ is a section of $\omega_X$. Consequently the quadratic 
differential $q$ is a global square, contradicting Proposition~\ref{global}.
\par 
For the latter case, both components of $X$ are $\PP^1$. First assume that $X$ 
is not hyperelliptic. Since $X$ is $4$-connected, its canonical embedding consists of 
two conics $X_1, X_2$ in $\PP^2$, intersecting at four nodes. The condition 
$h^0(X, 2z_1 + z_2) = 2$ implies that the line $L = \overline{z_1z_2}$ 
is tangent to $X_1$ at $z_1$. Consider the conic $Q$ satisfying 
$Q\cdot X = 4z_1 + 4z_2$. If $Q$ is smooth, then the intersection 
$Q\cdot L$ contains $2z_1 + z_2$, contradicting B\'{e}zout's theorem. 
Then $Q$ must be a double line $2L_1$ with $L_1\cdot X = 2z_1 + 2z_2$, contradicting that  
$q$ is not a global square based on Proposition~\ref{global}. If $X$ is hyperelliptic, consider the 
log surface $(S_{4}, X)$. Then both $X_1$ and $X_2$ have class equal to 
$4f$, where $f$ is the ruling class. The condition $h^0(X, 2z_1 + z_2) = 2$ 
implies that there exists a ruling $L$ such that $L\cdot X = 2z_1$ or $z_1 + z_2$. 
Consider the limiting hyperplane section $Q$ in $S_4$ such that 
$Q\cdot X_1 = 4z_1$ and $Q \cdot X_2 = 4z_2$. Since $Q\cdot L = 1$, if $Q$ is 
smooth, the only possibility is that $L\cdot X = z_1 + z_2$. Since $\OO_X(2L)\sim \omega_X$, 
this implies that $q$ is a global square, contradicting Proposition~\ref{global}. 
If $Q$ contains $L$ as a component, then it passes through $v$, hence as a consequence  
$Q = 4L$, which again contradicts that $q$ is not a global square. 
\par
For $\QQ(5,3,1,-1)$, the proof simply combines that of $\QQ(6,2)$ and of $\QQ(4,4)$, since the degeneration types of $X$ are covered in those two strata. 
Finally for $\QQ(7,2,-1)$ and $\QQ(7,3,-1,-1)$, an argument similar to that of $\QQ(6,2)$ works without any further complication. 
\end{proof}
\par
Based on the argument above, for $\QQ(6,2)^{\nh}$ the only possible intersection
of the lift of $C$ to $\barmoduli[3,2]$ with a component of the boundary
over $\delta_1$ is $\delta_{1,\{2\}}$. This boundary term does not appear
in the divisor class \eqref{eq:PicBN3;21} of $BN^1_{3,(2,1)}$. More generally,
by Proposition~\ref{prop:disjointbydegree} the \Teichmuller curve
hits none of the boundary terms in the presentation \eqref{eq:PicBN3;21} 
of $BN^1_{3,(2,1)}$. From the lemma we can thus calculate $L^+(C) = 5/4$. The same reasoning applies to 
the other strata listed in the lemma and one can quickly deduce the corresponding values of $L^+$.  

\subsection{The Brill-Noether divisor $BN^1_{3,(1,1,1)}$}
\par
\begin{lemma} For a Teichm\"uller curve $C$ generated by a half-translation surface $(X,q)$ in 
one of the strata $\QQ(4,2,2)$, $\QQ(3,3,2)^{\nh}$, $\QQ(4,3,2,-1)$, $\QQ(3,2,2,1)$
and $\QQ(3,3,1,1)^{\nh }$, lift $C$ to $\barmoduli[3,3]$ using the first three zeros of $q$.
Then $\ol{C}$ does not intersect the Brill-Noether divisor 
$BN^1_{3,(1,1,1)}$ in $\BM_{3,3}$. 
\end{lemma}
\par
\begin{proof}
For $\QQ(4,2,2)$, suppose that a half-translation surface $X$ parameterized by $\ol{C}$ 
also lies in $BN^1_{3,(1,1,1)}$. We 
have $\divisor(q) = 4z_1 + 2z_2 + 2z_3 \sim \omega^{\otimes 2}_X$.  
The linear system $|z_1 + z_2 + z_3|$ yields a $g^1_3$ for $X$. First 
consider the case when $X$ is irreducible. 
The associated triple cover implies that $\omega_X\sim z_1 + z_2 + z_3 + r$ for some 
$r\in X$ and we conclude that $2z_1 \sim 2r$. If $r = z_1$, then 
$\omega_X\sim 2z_1 + z_2 + z_3$. Consequently $q$ is a global square, 
contradicting Proposition~\ref{global}. If $r\neq z_1$, then $X$ is 
hyperelliptic and $z_1, r$ are both Weierstrass points. But 
$\omega_X\sim z_1 + z_2 + z_3 + r$ is still impossible
for $z_2, z_3\neq z_1$.  
\par
If $X$ is reducible, by Corollaries~\ref{signature} and \ref{component}, $X$ consists of 
two irreducible components $X_1$ and $X_2$ of genus $g_1$ and $g_2$, containing $z_1$ 
and $z_2, z_3$ respectively, meeting at $n$ nodes $t_1,\ldots, t_n$ such that $2g_i - 2 + n  
=  2$ for $i = 1, 2$. Moreover, $(g_1, g_2, n)$ is either $(1,1,2)$ or  $(0,0,4)$. Below we 
show that both cases are impossible. 
\par
For the former case, analyzing the triple cover, we have $z_2 + z_3\sim t_1 + t_2$ 
in $X_2$. Restricting $q$ to $X_1$, we also have $2t_1 + 2t_2\sim 4z_1$ in $X_1$. 
Consider the canonical model of a generic curve $X'$ in $C$ degenerating to $X$. By
the description in Section~\ref{logsurface}, the log surface $(S, X')$ degenerates to $(S_{2,2}, X)$. The image of
$S$ in $\PP^5$ is the Veronese $\PP^2$ if $X'$ is non-hyperelliptic, 
or it is a cone $S_4$ over a rational normal quartic if $X'$ is hyperelliptic, and 
$S_{2,2}$ is a cone over the one-nodal union of two conics. 
Denote by $S_1$ and $S_2$ the two components of $S$ and let $v$ be the vertex. Note 
that $X_i$ embeds into $S_i$ for $i=1,2$. The common ruling of $S_1, S_2$ contains the 
two nodes $t_1, t_2$. For $X'$, there exists a hyperplane section $Q'$ of $S$ such 
that $Q'\cdot X' = 4z'_1 + 2z'_2 + 2z'_3$. The limiting hyperplane section $Q$ of 
$S_{2,2}$ satisfies $Q\cdot S_{2,2} = 4z_1 + 2z_2 + 2z_3$. Let $Q_1$ and $Q_2$ be the 
components of $Q$ lying in $S_1$ and $S_2$, respectively. 
Then $Q_1\cdot X_1 = 4z_1$ and $Q_2\cdot X_2 = 2z_2 + 2z_3$. 
Since $z_2 + z_3\sim t_1 + t_2$, the ruling $L_{23} = \overline{z_2z_3}$ passes through 
$v$ and consequently $Q_2$ is the double ruling $2L_{23}$. Then $Q_1$ also passes 
through $v$, hence it consists of two rulings. Since $Q_1\cdot X_1 = 4z_1$, it 
must be the double ruling $2L_1$, where $L_1$ is tangent to $X_1$ at $z_1$. Then 
$L_1 + L_{23}$ cuts out $2z_1 + z_2 + z_3$ in $X$, hence it is a section of $\omega_X$. 
Then we conclude that $q$ is a global square, contradicting Proposition~\ref{global}. 
\par
If both components of $X$ are $\PP^1$, first assume that $X$ is not hyperelliptic. 
Since $X$ is $4$-connected, its canonical embedding consists of two conics in $\PP^2$ 
intersecting 
at four nodes. The assumption $h^0(X, z_1 + z_2+ z_3) = 2$ 
implies that $z_1, z_2, z_3$ are contained in a line $L$. Consider the conic $Q$ 
that cuts out $4z_1 + 2z_2 + 2z_3$ in $X$. 
If $Q$ is smooth, then the intersection $Q\cdot L$ contains $z_1 + z_2 + z_3$, 
contradicting B\'{e}zout's theorem. Hence $Q$ must be a union of two lines $L_1$ 
and $L_2$ such that 
$L_1\cdot X = 2z_1 + 2z_2$, $L_2\cdot X = 2z_1 + 2z_3$ or 
$L_1\cdot X = L_2 \cdot X = 2z_1 + z_2 + z_3$. 
The former case is impossible, because $L_1, L_2$ would be the same line tangent to 
$X$ at $z_1$. The latter is also impossible, because it would imply that $q$ is 
a global square, contradicting Proposition~\ref{global}. Now consider the case 
when $X$ is hyperelliptic. We use the log surface $(S_{0,4}, X)$. Both $X_1$ 
and $X_2$ have class equal to $4f$, where $f$ is the class of a ruling. There exists a 
ruling $L$ such that $L\cdot X = z_1 + z_2$ (up to relabeling 
$z_2, z_3$). Since $\omega_X\sim \OO_X(2L)$, we conclude that 
$4z_1 + 4z_2\sim 4z_1 + 2z_2 + 2z_3$, hence $2z_2\sim 2z_3$. Then the 
ruling $L$ through $z_2$ is tangent to $X$ at $z_2$, contradicting that $L\cdot X = z_1 + z_2$. 
\par
For $\QQ(3,3,2)^{\nh}$, suppose that $X$ is in the intersection with the Brill-Noether divisor. 
We have $\divisor(q) = 3z_1 + 3z_2 + 2z_3 \sim \omega^{\otimes 2}_X$. Let us deal with irreducible $X$ first. If $h^0(\OO_X(z_1 + z_2 + z_3)) = 2$, 
then $\omega_X\sim p_1 + p_2 + p_3 + r$ and $ p_1 + p_2 \sim 2r$, which contradicts that 
it is non-hyperelliptic together with Proposition~\ref{prop:hypopen}.
\par
If $X$ is reducible, by Corollaries~\ref{signature} and \ref{component}, it consists 
of two irreducible components $X_1$ and $X_2$ of genus $1$ and $2$, containing 
$z_3$ and $z_1, z_2$ respectively, meeting at a node $t$. If $h^0(X, z_1 + z_2 + z_3) = 2$, 
by analyzing the triple admissible cover, we have 
$z_1 + z_2\sim 2t$ in $X_2$, hence $t$ is a Weierstrass point and $z_1, z_2$ are 
conjugate in $X_2$. Since $3z_1 + 3z_2 + 2z_3\sim \omega^{\otimes 2}_X$, we conclude that
$3z_1 + 3z_2 \sim 6t$ in $X_2$ and $2z_3\sim 2t$. Therefore, $X$ admits a double 
admissible cover with $z_3$ as a ramification point and $z_1, z_2$ being conjugate. It contradicts
the non-hyperelliptic assumption. 
\par
The proofs for the other strata listed in the lemma are completely analogous to either one of the above two. 
\end{proof}
\par
By Proposition~\ref{prop:disjointbydegree}  a \Teichmuller curve in one of the above strata 
hits none of the boundary terms in the presentation \eqref{eq:PicBNonly1s} of 
$BN^1_{3,(1,1,1)}$. From the above lemma we can thus calculate the values of $L^{+}$ accordingly.

\section{Genus four: non-exceptional strata} \label{sec:g4nonexc}

\begin{theorem} \label{thm:nvg4}
In genus four, besides the components of hyperelliptic flat surfaces
and those non-varying exceptional components, the following strata are non-varying:
$$\begin{array}{|l|l||l|l|}
\hline
\text{\rm Stratum} \,& L^+ \,&\text{\rm Stratum} \,& L^+ \, \\
\hline
\QQ(13,-1) & 4/3 & \QQ(7,5) & 32/21  \\
\hline
\QQ(11,1) & 20/13 & \QQ(8,3,1) & 8/5 \\
\hline
\QQ(10,2)^\nh & 3/2 & \QQ(7,3,2) & 47/30 \\
\hline
\QQ(8,4)& 23/15 & \QQ(5,4,3) & 167/105 \\
\hline
\end{array}
$$
All the other strata in genus four of dimension less than or equal to nine
are varying.
\end{theorem}
\par
Examples of square-tiled surfaces certifying the varying strata are listed in the appendix. 

\subsection{The Weierstrass divisor}

\begin{lemma} For a Teichm\"uller curve $C$ generated by $(X,q)$ 
in one of the strata $\QQ(13,-1)$ or $\QQ(11,1)$, lift $C$ to $\barmoduli[4,1]$ using the first zero of $q$.
Then the intersection with the Weierstrass divisor $W \cdot \ol{C} = 0.$
\end{lemma}
\par
\begin{proof} Suppose that $(X,q)$ is in the intersection. 
In the case of $\QQ(13,-1)$ we have $\divisor(q) = 13z_1 - p_1 \sim \omega^{\otimes 2}_X$. 
If $z_1$ is a Weierstrass point, then $\omega_X\sim 4z_1 + r +s$ for some $r$ and $s$,
hence $5z_1\sim p_1 + 2r + 2s$. If $z_1$ is a base point of $|5z_1|$, then 
$r = z_1$ (up to relabeling $r, s$) and 
$3z_1\sim p_1 + 2s$. By Riemann-Roch, we get $h^0(X, 2z_1 + s) = 2$,  
hence $s = z_1$ and $z_1\sim p_1$, impossible. If $z_1$ is not a base point of $|5z_1|$, we have $h^0(X, 5z_1) = h^0(X, 4z_1) + 1 = 3$. 
By Riemann-Roch, we conclude that $h^0(X, r + s - z_1) = 1$, hence $r = z_1$. 
Then $\omega_X\sim 5z_1 + s$
and $3z_1\sim p_1 + 2s$. Arguing as above, we deduce the same contradiction. 
\par
For $\QQ(11,1)$, we have $\divisor(q) = 11z_1 + z_2 \sim \omega^{\otimes 2}_X$. If $z_1$ is a
Weierstrass point, then $\omega_X\sim 4z_1 + r +s$ for some $r$ and $s$,
hence $3z_1\sim z_2 + r + s$. Then $z_1 + r + s$ gives rise to a $g^1_3$ which is 
different from the one given by $z_2 + r + s$. If $X$ is not hyperelliptic, 
using its canonical embedding in a quadric surface in $\PP^3$, there do not 
exist two different rulings passing through $r$ and $s$ or both tangent to $X$
at $r$ (in the case $r=s$), leading to a contradiction. If $X$ is hyperelliptic, since $z_1$ 
is a Weierstrass point, we have $12z_1\sim \omega^{\otimes 2}_X$, hence
$z_1\sim z_2$, impossible. 
\end{proof}
\par
Since the stable curves parameterized by $\ol{C}$ in the above strata are irreducible, 
Theorem~\ref{thm:nvg4} follows in this case from this lemma together with \eqref{eq:PicWP} and
Proposition~\ref{intersection}.

\subsection{The Brill-Noether divisors $BN^1_{4, (3,1)}$, $BN^1_{4,(2,2)}$ and $BN^1_{4,(2,1,1)}$}

\begin{lemma}\label{le:g4case102} 
For a Teichm\"uller curve $C$ generated by $(X,q)$ 
in one of the strata $\QQ(10,2)^\nh$, $\QQ(8,4)$ or $\QQ(8,3,1)$, 
lift $C$ to $\barmoduli[4,2]$ using the first two zeros of $q$.
Then the intersection with the Brill-Noether divisor $BN^1_{4, (3,1)} \cdot \ol{C} = 0$. 
\end{lemma}
\par
\begin{proof} 
Suppose that $(X,q)$ is in the intersection of $\ol{C}$ with $BN^1_{4, (3,1)}$. 
In the case of $\QQ(10,2)^\nh$ we know that $\divisor(q) = 10z_1 + 2z_2 \sim \omega^{\otimes 2}_X$. 
From $h^0(X, 3z_1 + z_2) = 2$ we deduce that $\omega_X\sim 3z_1 + z_2 + r + s$, hence 
$4z_1\sim 2r + 2s$ for some $r$ and $s$. Then $\omega_X\sim 4z_1 + u + v \sim 3z_1 + z_2 + r + s$. 
If these are the same divisor, we have $u = z_2$ and $r = z_1$, hence $2z_1 \sim 2s$. Note that $s\neq p_1$, 
for otherwise $q$ would be a global square. Then we conclude 
that $X$ is hyperelliptic, $z_1$ is a Weierstrass point and $s = p_2$ is also a 
Weierstrass point, contradicting that $(X,q)$ is non-hyperelliptic
and Proposition~\ref{prop:hypopen}.  
\par
If $4z_1 + u + v \sim 3z_1 + z_2 + r + s$ are two different divisors, 
then $z_2 + r + s$ yields a $g^1_3$. By Riemann-Roch, $3z_1$ also yields a 
$g^1_3$. By $h^0(X, 4z_1) = h^0(X, 3z_1) = 2$, $z_1$ is a base 
point of $|4z_1|$. But $2r + 2s$ is a section, hence $r = z_1$. Then we still conclude 
that $X$ is hyperelliptic and $z_1, z_2$ are Weierstrass points, leading to the same contradiction. 
\par
The proofs for $\QQ(8,4)$ and $\QQ(8,3,1)$ are completely analogous. 
\end{proof}
\par
\begin{lemma} For a Teichm\"uller curve $C$ generated by $(X,q)$ 
in $\QQ(7,5)$, lift $C$ to $\barmoduli[4,2]$ using the two zeros of $q$.
Then the intersection with the Brill-Noether divisor $BN^1_{4, (2,2)} \cdot \ol{C} = 0.$
\end{lemma}
\par
\begin{proof}The proof of Lemma~\ref{le:g4case102} can be 
copied almost verbatim, replacing
coefficients $(3,1)$ by $(2,2)$ at the appropriate places in the above.
\end{proof}
\par
\begin{lemma} \label{le:g4case732}
For a Teichm\"uller curve $C$ generated by $(X,q)$ 
in $\QQ(7,3,2)$ or $\QQ(5,4,3)$, lift $C$ to $\barmoduli[4,3]$ using the three zeros of $q$.
Then the intersection with the Brill-Noether divisor $BN^1_{4, (2,1,1)} \cdot \ol{C} = 0.$
\end{lemma}
\par
\begin{proof}
The proof of this lemma is completely parallel to Lemma~\ref{le:g4case102}.
\end{proof}
\par
By Proposition~\ref{prop:disjointbydegree} the \Teichmuller curves in these strata
hit none of the boundary terms in the presentation \eqref{eq:PicBN4b} 
of $BN^1_{4,(3,1)}$, $BN^1_{4,(2,2)}$ or $BN^1_{4,(2,1,1)}$, respectively. From
the lemma the claim on the sum of Lyapunov exponents follows.
\par

\begin{appendix} 

\section{Filtration of the Hodge bundle} \label{sec:filtration}
There are three exceptional components $\QQ(6,3,-1)^{\irr}$, $\QQ(12)^{\irr}$ and $\QQ(9,3)^{\irr}$, 
non-varying as claimed in Theorems~\ref{thm:execg3NV}
and~\ref{thm:execg4NV}, whose proofs are not given yet. Since their lifts are contained in 
the Brill-Noether divisors that were used to show non-varying 
for the regular components, respectively, it seems difficult to come up with some other divisors 
disjoint with the \Teichmuller curves in these strata. Here we take an alternative approach, 
adapting the idea of \cite{yuzuo}. Let us first do some preparatory setup. Some of the lemmas below are also stated in \cite{yuzuo}, but we include their proofs for completeness. 
\par 
Let $f : \XX \to C$ be the universal curve over a \Teichmuller curve $C$. Let $S_1, \ldots, S_k$ be 
the disjoint sections specializing to the zeros $z_1,\ldots,z_n$ in each fiber. 
If $h^0(X, \sum_{i=1}^k a_iz_i) = n$ for every fiber $X$ in $\XX$, then $f_{*}\OO_{\XX}(\sum_{i=1}^k a_i S_i)$ 
is a vector bundle of rank $n$ on $C$. Moreover, any subsheaf of a vector bundle on $C$ is locally free. 
\par
\begin{lemma}
\label{le:filtration}
If $h^0(X, \sum_{i=1}^k d_i z_i)$ is the same for every fiber $X$ and if  
$$h^0\Big(X, \sum_{i=1}^k d_i z_i\Big) = h^0\Big(X, \sum_{i=1}^k (d_i - a_i) z_i\Big) +  \sum_{i=1}^k a_i, $$
then the first chern classes of the corresponding direct images are related by
$$ c_1 \Big(f_{*}\OO_{\XX}\Big(\sum_{i=1}^k d_i S_i\Big)\Big) = c_1 \Big(f_{*}\OO_{\XX}\Big(\sum_{i=1}^k (d_i-a_i) S_i\Big)\Big) + \sum_{i=1}^k c_1 (f_{*}\OO_{a_i S_i}(d_i S_i)). $$
\end{lemma}
\par
\begin{proof}
We have the exact sequence 
$$ 0 \to f_{*}\OO_{\XX}\Big(\sum_{i=1}^k (d_i-a_i) S_i\Big) \to f_{*}\OO_{\XX}\Big(\sum_{i=1}^k d_i S_i\Big) \to \sum_{i=1}^k f_{*}\OO_{a_i S_i}(d_i S_i) $$
$$ \to R^1 f_{*}\OO_{\XX}\Big(\sum_{i=1}^k (d_i-a_i) S_i\Big) \to R^1f_{*}\OO_{\XX}\Big(\sum_{i=1}^k d_i S_i\Big) \to 0. $$
All terms are locally free by assumption. The two $R^1f_{*}$ terms have the same rank 
by Riemann-Roch, hence they are isomorphic and then the claim follows.  
\end{proof}
\par
\begin{lemma}
\label{le:trivial}
Suppose the $a_i$ are positive integers. If $h^0(X, \sum_{i=1}^k a_i z_i) = 1$ for every 
fiber $X$, then $f_{*}\OO_{\XX}(\sum_{i=1}^k a_i S_i) = \OO_C$. 
\end{lemma}

\begin{proof}
Consider the exact sequence 
$$ 0 \to f_{*}\OO_{\XX}((a_1-1)S_1+\sum_{i = 2}^k a_i S_i) \to f_{*}\OO_{\XX}\Big(\sum_{i=1}^k a_i S_i\Big) \to f_{*}\OO_{S_1}\Big(\sum_{i=1}^k a_i S_i\Big) $$
$$ \to R^1f_{*}\OO_{\XX}\Big((a_1-1)S_1+\sum_{i = 2}^k a_i S_i\Big) \to R^1f_{*}\OO_{\XX}\Big(\sum_{i=1}^k a_i S_i\Big) \to 0. $$
By assumption we have 
$$h^0\Big(X, \sum_{i=1}^k a_i S_i\Big) = h^0\Big(X, (a_1-1)S_1+\sum_{i = 2}^k a_i S_i\Big) = 1,$$ 
hence the first two $f_{*}$ terms are line bundles and the last two $R^1f_{*}$ terms are vector bundles whose ranks differ by one. The middle term is also a line bundle isomorphic to $\OO_{S_1}(a_1S_1)$, which cannot have a torsion subsheaf. Hence the first two line bundles are isomorphic. Now the claim follows by induction on $a$. 
\end{proof}
\par
\begin{lemma}
\label{le:non-reduced}
Let $S$ be a section in the above setting and $a \geq 1$. Define $aS$ to be the subscheme of $\XX$ whose ideal sheaf is $\OO(-aS)$.  
Then for any integer $b$ we have 
$$ c_1(f_{*}\OO_{aS}(bS)) = \sum_{i=0}^{a-1} c_1(f_{*}\OO_{S}((b-i)S)). $$ 
\end{lemma}
\par
\begin{proof}
The claim holds trivially for $a=1$. Suppose it is true for all positive integers smaller than or equal to $a$. Treating $aS$ as a subscheme of $(a+1)S$, we have 
$$ 0 \to \II_{aS/(a+1)S}   \to \OO_{(a+1)S} \to \OO_{aS} \to 0. $$
The ideal sheaf $\II_{aS/(a+1)S}$ is isomorphic to $(N^{*}_{S/\XX})^{\otimes a}$, where $N^{*}_{S/\XX}$
is the conormal bundle isomorphic to $\OO_S(-aS)$.
Tensor the sequence with $\OO_{\XX}(bS)$ and apply $f_{*}$. We obtain that 
$$ 0 \to f_{*}\OO_S((b-a)S) \to f_{*}\OO_{(a+1)S}(bS) \to f_{*}\OO_{aS}(bS) \to 0, $$
since $R^1f_{*}$ is zero acting on any line bundle over $S$. Then the claim follows by induction. 
\end{proof}
\par
Let $\omega$ be the relative dualizing sheaf, $\gamma = c_1(\omega)$ and $\eta$ the nodal locus in $\XX$. In order to deal with quadratic differentials, 
we need to express $c_1(f_{*}(\omega^{\otimes 2}))$ by tautological classes on the moduli space of curves. An introduction to the calculations of the following type can be found in \cite[Chapter 3.E]{harrismorrison}.
\par
Since all the higher direct images of $\omega^{\otimes 2}$ are zero, by Grothendieck-Riemann-Roch we have 
\bas
\chern(f_{*}(\omega^{\otimes 2}))  & =  f_{*}\Big(\chern(\omega)^2\cdot\Big(1 - \frac{\gamma}{2} + \frac{\gamma^2 + \eta}{12}  \Big)\Big) \\
                                                             & =  f_{*}\Big( \Big(1 + \gamma + \frac{\gamma^2}{2}  \Big)^2\cdot  \Big(1 - \frac{\gamma}{2} + \frac{\gamma^2 + \eta}{12}  \Big)\Big) \\
                                                              & =  f_{*}\Big(1  + \frac{3}{2}\gamma + \frac{13\gamma^2 + \eta}{12}   \Big) \\
                                                               & =  (3g-3) +  \Big(\frac{13}{12}\kappa + \frac{1}{12}\delta \Big) \\
                                                                 & =  (3g-3) + (\lambda + \kappa), 
\eas
where $\kappa = f_{*}(\gamma^2)$. 
Hence we conclude that 
$$ c_1 (f_{*}(\omega^{\otimes 2})) = \lambda + \kappa. $$
Recall the notation and the intersection calculation in Proposition~\ref{intersection}. We have 
$$ \omega^{\otimes 2} = f^{*}\cF \otimes \OO_{\XX}\Big(\sum_{j=1}^n d_j S_j\Big). $$
By the projection formula, we get 
$$ f_{*}(\omega^{\otimes 2}) = \cF \otimes f_{*}\OO_{\XX}\Big(\sum_{j=1}^n d_j S_j\Big). $$
Note that $f_{*}\OO_{\XX}(\sum_{j=1}^n d_j S_j)$ is a vector bundle of rank $3g-3$, whose fibers are $H^0(X, \omega_X^{\otimes 2})$ for $X$ parameterized in $C$. Then we conclude that 
\bas
 \frac{c_1(f_{*}\OO_{\XX}(\sum_{j=1}^n d_j S_j))}{\chi} & =  \frac{C\cdot \lambda}{\chi} + \frac{C\cdot\kappa}{\chi} - (3g-3)    \\
                                                      & =  \frac{1}{2}L^{+}(C) + 6\kappa_{\mu} - (3g-3), 
 \eas
 where $\chi = \deg(\cF)$ and 
 $\kappa_{\mu} = \frac{1}{24} \big(\sum_{j=1}^{n}\frac{d_j(d_j+4)}{d_j + 2}\big)$
 for the signature $\mu = (d_1, \ldots, d_n)$. 
Therefore, we obtain the following expression: 
\ba \label{eq:GRR}
 L^{+}(C) & = (6g-6) + 2 \cdot \frac{c_1(f_{*}\OO_{\XX}(\sum_{j=1}^n d_j S_j))}{\chi} - 12 \kappa_{\mu}. 
\ea
\par

\subsection{The component $\QQ(6,3,-1)^{\irr}$} \label{sec:63-1irr}
Let $(X,q)$ be a half-translation surface parameterized by a \Teichmuller curve $C$ in $\QQ(6,3,-1)^{\irr}$. 
We have $6z_1 + 3z_2 - p \sim \omega_{X}^{\otimes 2}$. Moreover, by the parity we know $ h^0(X, 2z_1 + z_2) = 2$. 
\par
\begin{lemma}
We have $h^0(X, 3z_1 + z_2) = 2$ and $h^0(X, z_1 + z_2) = 1$ for all $X$ in $C$. 
\end{lemma}
\par
\begin{proof}
Consider first the case when $X$ is irreducible. If $h^0(X, 3z_1 + z_2) \geq 3$, then $\omega_X\sim 3z_1 + z_2$ and 
$6z_1 + 2z_2 \sim 6z_1 + 3z_2 - p$. It implies that $z_2\sim p$, impossible. If $h^0(X, z_1 + z_2) \geq 2$, then $X$ 
is hyperelliptic and $z_1, z_2$ are conjugate. Then we have $\omega_X\sim 2z_1 + 2z_2$ and $2z_1 \sim z_2 + p$, contradicting 
that $z_1, z_2$ are conjugate. 
\par
Next, suppose $X$ is reducible. By Corollaries~\ref{signature} and~\ref{component}, the only degenerate type is that $X$ consists of $X_1$ and $X_2$ of genus 
$1$ and $2$, respectively, joined at a node $t$ such that $X_1$ contains $z_2, p$ and $X_2$ contains $z_1$. 
By $h^0(X_2, 3z_1) = 2$ and $h^0(X_1, z_2) = 1$, it implies that $h^0(X, 3z_1 + z_2) = 2$, since a section on $X_1$ and a section on $X_2$ need to have the same value if they can be glued to form a global section on $X$. Similarly by $h^0(X_2, z_1) = 1$ and $h^0(X_1, z_2) = 1$, we conclude that 
$h^0(X, z_1 + z_2) = 1$. 
\end{proof}
\par
Now we are ready to calculate $L^+$. Let $S_1$, $S_2$ and $S_3$ be the sections in $\XX$ corresponding to the loci of $z_1$, $z_2$ and $p$, respectively. 
By the exact sequence 
$$ 0 \to f_{*}\OO_{\XX}(6S_1+3S_2 - S_3) \to f_{*}\OO_{\XX}(6S_1+3S_2) \to f_{*}\OO_{S_3} \to 0, $$
we have 
$$c_1(f_{*}\OO_{\XX}(6S_1+3S_2 - S_3)) = c_1(f_{*}\OO_{\XX}(6S_1+3S_2)). $$
Since $R^1f_{*}\OO_{X}(3S_1 + S_2) = 0$ by the above lemma, we obtain that  
\bas
c_1 (f_{*}\OO_{\XX}(6S_1 + 3S_2 )) & =  c_1(f_{*}\OO_{\XX}(3S_1+S_2)) + c_1 (f_{*}\OO_{3S_1}(6S_1)) + c_1(f_{*}\OO_{2S_2}(3S_2)) \\
                                                & =  c_1(f_{*}\OO_{\XX}(3S_1 + S_2)) + 15 S_1^2 + 5 S_2^2, 
\eas
where Lemma~\ref{le:non-reduced} is used in the last equality. 
Similarly using Lemmas~\ref{le:filtration} and~\ref{le:trivial} one can show that 
\bas
 c_1(f_{*}\OO_{\XX}(3S_1 + S_2)) & =  c_1(f_{*}\OO_{\XX}(2S_1 + S_2)) \\
 & =  c_1(f_{*}\OO_{\XX}(S_1 + S_2)) + 2S_1^2  \\
 &=  2S_1^2. 
\eas
Then we obtain that 
$$  c_1 (f_{*}\OO_{\XX}(6S_1+3S_2 - S_3)) = 17 S_1^2 + 5S_2^2 = - \frac{25}{8} \chi, $$
where we use the self-intersection formula of the $S_j$ in Proposition~\ref{intersection}. 
Finally by the equality~\eqref{eq:GRR}, we conclude that $L^{+}(C) = 7/5$.  
\par

\subsection{The component $\QQ(12)^{\irr}$} \label{sec:12irr}
Let $(X,q)$ be a half-translation surface parameterized by a \Teichmuller curve $C$ in $\QQ(12)^{\irr}$. 
We have $12z \sim \omega_X^{\otimes 2}$. By the parity we know $h^0(X, 4z) = 2$. 

\begin{lemma}
We have $h^0(X, 6z) = 3$, $h^0(X, 5z) = 2$ and $h^0(X, 3z) = 1$ for all $X$ in $C$. 
\end{lemma}
 
\begin{proof}
Since $q$ has a unique zero, $X$ is irreducible. If $h^0(X, 6z) \geq 4$, we have 
$\omega_X \sim 6z$ and $q$ is a global square, impossible. If $h^0(X, 5z)\geq 3$, then 
$\omega_X\sim 5z+ r$ for some $r\neq z$. We have $10z + 2r\sim 12z$, hence $2r \sim 2z$,  
$X$ is hyperelliptic and $z$ is a Weierstrass point, still contradicting that $q$ is not a global square. 
 If $h^0(X, 3z) \geq 2$, $X$ cannot be Gieseker-Petri special, for otherwise $q$ would be a global square. Then $|3z|$ gives rise to a 
 $g^1_3$ and suppose $|z + r + s|$ is the other $g^1_3$. We have $\omega_X\sim 4z + r + s$ and $2r + 2s \sim 4z$. But $z$ is a base point 
 of $|4z|$, hence $r = z$ and $\omega_X\sim 5z+ s$, contradicting that $h^0(X, 5z) = 2$ as shown before. 
\end{proof}

Let $S$ be the section of zeros of $q$ in $\XX$. By the above lemma, $R^1f_{*}\OO_{X}(5S) = 0$, hence we conclude that  
\bas
c_1 (f_{*}\OO_{\XX}(12S)) & =  c_1(f_{*}\OO_{\XX}(5S)) + c_1 (f_{*}\OO_{7S}(12S)) \\
                                                & =  c_1(f_{*}\OO_{\XX}(5S)) + 63 S^2. 
\eas
Similarly we can show that
\bas
 c_1(f_{*}\OO_{\XX}(5S)) & =  c_1(f_{*}\OO_{\XX}(4S)) \\
 & =  c_1(f_{*}\OO_{\XX}(3S)) + 4S^2 \\
 &=  4S^2. 
\eas
Then we obtain that 
$$  c_1 (f_{*}\OO_{\XX}(12S)) = 67 S^2 = - \frac{67}{14} \chi. $$
By the identity~\eqref{eq:GRR}, we conclude that $ L^{+}(C) = 11/7$. 
\par
\subsection{The component $\QQ(9,3)^{\irr}$} \label{sec:93irr}
Let $(X,q)$ be a half-translation surface parameterized by a \Teichmuller curve $C$ in $\QQ(9,3)^{\irr}$. 
We have $9z_1 + 3z_2 \sim \omega_X^{\otimes 2}$. By the parity we know $h^0(X, 3z_1 + z_2) = 2$. 
\par
\begin{lemma}
We have $h^0(X, 4z_1 + z_2) = 2$ and $h^0(X, 2z_1 + z_2) = 1$ for all $X$ parameterized in $C$. 
\end{lemma}

\begin{proof}
By Corollaries~\ref{signature} and ~\ref{component}, $X$ is irreducible. If $h^0(X, 4z_1 + z_2) \geq 3$, then
$\omega_X\sim 4z_1 + z_2 + r$, hence $2r \sim z_1 + z_2$, $X$ is hyperelliptic and $z_1, z_2$ are conjugate. Then
$\omega_X\sim 3z_1 + 3z_2$, hence $3z_1\sim 3z_2$, which implies that $z_1, z_2$ are both Weierstrass points, contradicting that they are conjugate. 
If $h^0(X, 2z_1 + z_2) \geq 2$, $X$ cannot be Gieseker-Petri special, for otherwise $\omega_X\sim 4z_1 + 2z_2$ and $z_1\sim z_2$, impossible. 
Consequently $|2z_1 + z_2|$ provides a $g^1_3$ and suppose that $|z_1 + r + s|$ is the other $g^1_3$. We have $\omega_X\sim 3z_1 + z_2 + r + s$, hence 
$2r + 2s \sim 3z_1 + z_2$. Note that $z_1$ is a base point of $|3z_1 + z_2|$, hence $r = z_1$ and we thus obtain $\omega_X\sim 4z_1 + z_2 + s$, 
contradicting that $h^0(X, 4z_1 + z_2) = 2$ as shown before. 
\end{proof}

By the above lemma $R^1f_{*}\OO_{X}(4S_1 + S_2) = 0$ and consequently
\bas
c_1 (f_{*}\OO_{\XX}(9S_1 + 3S_2)) & =  c_1(f_{*}\OO_{\XX}(4S_1+S_2)) + c_1 (f_{*}\OO_{5S_1}(9S_1)) + c_1(f_{*}\OO_{2S_2}(3S_2)) \\
                                                & =  c_1(f_{*}\OO_{\XX}(4S_1 + S_2)) + 35 S_1^2 + 5 S_2^2.  
\eas
Similarly one can show that
\bas
 c_1(f_{*}\OO_{\XX}(4S_1 + S_2)) & =  c_1(f_{*}\OO_{\XX}(3S_1 + S_2)) \\
 & =  c_1(f_{*}\OO_{\XX}(2S_1 + S_2)) + 3S_1^2  \\
 &=  3S_1^2. 
\eas
Then we obtain that 
$$  c_1 (f_{*}\OO_{\XX}(9S_1 + 3S_2)) = 38 S_1^2 + 5S_2^2 = - \frac{49}{11} \chi. $$
Finally by~\eqref{eq:GRR}, we conclude that $L^{+}(C) = 92/55$.   
\par
\begin{remark}
Once we understand the boundary behavior as well as $h^0(X, \sum a_i z_i)$ for all $X$ in a Teichm\"uller curve $C$, apparently the above method may provide a parallel proof for many other non-varying strata of quadratic differentials in low genus. We leave it as an exercise to the reader.   
\end{remark}

\subsection{The missing non-varying strata of abelian differentials}

In our earlier work \cite{chenmoeller} three strata of abelian differentials $\OM_{4}(4,2)^{\odd}$, 
$\OM_{4}(4,2)^{\even}$ and $\OM_{5}(6,2)^{\odd}$ were 
predicted to be non-varying based on the computer data of Zorich and Delecroix,  but no proof was
given there. In \cite{yuzuo} an argument for their non-varying property was found using the 
filtration of the Hodge bundle. For completeness we include a detailed proof in this section. 
\par
Let $(X,\omega)$ be a flat surface generating a \Teichmuller curve $C$ in a 
stratum of abelian differentials. Use $z_i$ to denote the zeros of $\omega$. 
\par
\begin{lemma} For all $X$ in $C$, we have the following results. 
\par 
If $C$ is in $\OM_{4}(4,2)^{\odd}$, then $h^0(X, 2z_1 + z_2) = 1$. 
\par
If $C$ is in $\OM_{4}(4,2)^{\even}$, then $h^0(X, 2z_1 + z_2) = 2$ and $h^0(X, z_1 + z_2) = 1$. 
\par
If $C$ is in $\OM_{5}(6,2)^{\odd}$, then $h^0(X, 3z_1 + z_2) = 1$. 
\end{lemma}
\par 
\begin{proof}
The reader may refer to \cite[Section 4]{chenmoeller} for properties of \Teichmuller curves generated by abelian differentials. All the claims in the lemma follow directly 
from the spin parity except that $h^0(X, z_1 + z_2) = 1$ for the stratum $\OM_{4}(4,2)^{\even}$. Suppose on the contrary $h^0(X, z_1 + z_2)\geq 2$. Then $X$ is in the hyperelliptic locus and $z_1, z_2$ are conjugate. It implies that $\omega_X\sim 3z_1 + 3z_2 \sim 4z_1 + 2z_2$, hence $z_1\sim z_2$, contradicting that $z_1\neq z_2$. 
It holds even if $X$ is reducible. In that case each component $X_i$ of $X$ has to contain a zero $z_i$. By $h^0(X_i, z_i) = 1$ for $i = 1,2$, we conclude that 
$h^0(X, z_1 + z_2) = 1$, because gluing sections on the $X_i$ to form a global section on $X$ imposes an additional condition.   
\end{proof}
\par
Let $S_i$ be the section in $\XX$ corresponding to the zero $z_i$ of order $m_i$. Let $\omega$ be the relative dualizing sheaf of $f: \XX\to C$. We have   
$$ \omega = f^{*}\LL \otimes \OO_{\XX}\Big(\sum_{i=1}^k m_i S_i\Big), $$
where $\LL$ is the line bundle on $C$ corresponding to the generating abelian differential and $\deg(\LL) = \chi/2$. 
The projection formula implies that 
$$f_{*}\omega = \LL \otimes f_{*}\Big(\OO_{\XX}\Big(\sum_{i=1}^k m_i S_i\Big)\Big). $$
Since $f_{*}(\OO_{\XX}(\sum_{i=1}^k m_i S_i))$ is a vector bundle of rank $g$ whose fibers are $H^0(X, \omega_X)$, we have 
\bas
C\cdot \lambda & = c_1(f_{*}\omega) \\
      & = g \cdot \frac{\chi}{2} + c_1\Big(f_{*}\Big(\OO_{\XX}\Big(\sum_{i=1}^k m_i S_i\Big)\Big)\Big). 
\eas
By \cite[Proposition 4.5]{chenmoeller} we conclude that
\ba
\label{eq:abelian}
 L(C) & = 2\cdot \frac{C\cdot \lambda}{\chi} \\
          & = g + 2\cdot \frac{c_1(f_{*}(\OO_{\XX}(\sum_{i=1}^k m_i S_i)))}{\chi}.
\ea

\begin{theorem}
Let $C$ be a \Teichmuller curve generated by a flat surface in one of the strata
$\OM_{4}(4,2)^{\odd}$, $\OM_{4}(4,2)^{\even}$ or $\OM_{5}(6,2)^{\odd}$. 
Then the sum of Lyapunov exponents $L(C)$ equals
$29/15$, $32/15$ or $46/21$, respectively. 
\end{theorem}
\par
\begin{proof}
For the stratum $\OM_{4}(4,2)^{\odd}$, using the preceding Lemma to verify the hypothesis
of the Lemmas~\ref{le:filtration}, ~\ref{le:trivial} and~\ref{le:non-reduced}, we conclude
\bas
c_1(f_{*}\OO_{\XX}(4S_1 + 2S_2)) & = c_1(f_{*}\OO_{\XX}(2S_1 + S_2)) + c_1(f_{*}\OO_{2S_1}(4S_1)) + c_1(f_{*}\OO_{S_2}(2S_2)) \\
& = 7S_1^2 + 2 S_2^2 \\
& = - \frac{31}{30}\cdot \chi, 
\eas
where the self-intersection formula of the $S_i$ in \cite[Proposition 4.5]{chenmoeller} is used in the last step. 
By the relation~\eqref{eq:abelian}, we finally obtain that 
$$L(C) = 4 - 2\cdot \frac{31}{30} = \frac{29}{15}.$$
\par
For the stratum $\OM_{4}(4,2)^{\even}$, by the above lemmas we have 
\bas
c_1(f_{*}\OO_{\XX}(4S_1 + 2S_2)) & =  c_1(f_{*}\OO_{\XX}(3S_1 + S_2)) + c_1(f_{*}\OO_{S_1}(4S_1)) + c_1(f_{*}\OO_{S_2}(2S_2)) \\
                                                              & =  c_1(f_{*}\OO_{\XX}(3S_1 + S_2)) + 4 S_1^2 + 2S_2^2 \\
                                                              & = c_1(f_{*}\OO_{\XX}(2S_1 + S_2)) + 4 S_1^2 + 2S_2^2 \\
                                                              & = c_1(f_{*}\OO_{S_1}(2S_1)) + 4 S_1^2 + 2S_2^2 \\
                                                              & =  6S_1^2 + 2S_2^2 \\
                                                              & = - \frac{14}{15}\cdot \chi. 
\eas
Then by the equality~\eqref{eq:abelian}, we conclude that 
$$ L(C) = 4 - 2\cdot \frac{14}{15} = \frac{32}{15}. $$
\par
For the stratum $\OM_{6}(6,2)^{\odd}$, we have 
\bas
c_1(f_{*}\OO_{\XX}(6S_1 + 2S_2)) & = c_1(f_{*}\OO_{\XX}(3S_1 + S_2)) + c_1(f_{*}\OO_{3S_1}(6S_1)) + c_1(f_{*}\OO_{S_2}(2S_2)) \\
                                                               & = c_1(f_{*}\OO_{\XX}(3S_1 + S_2)) + 15 S_1^2 + 2S_2^2 \\
                                                               & =  15 S_1^2 + 2S_2^2 \\
                                                               & = - \frac{59}{42}\cdot\chi.                                                          
\eas
\par
Finally by~\eqref{eq:abelian}, we conclude that 
$$ L(C) = 5 - 2\cdot \frac{59}{42} = \frac{46}{21}. $$
\end{proof}
\par 

\section{Local calculations for the exceptional strata} \label{sec:exclocal}

\subsection{Genus three} \label{sec:g3local}

We begin with the proof of Proposition~\ref{prop:Eirred}, which is 
decomposed into the following lemmas. We follow the notation of
Section~\ref{sec:g3exceptional}. 
\par
\begin{lemma} \label{le:Eirrednew}
If $X$ is a non-hyperelliptic, smooth genus three curve, then 
for a generic choice of a line passing through $p$ the parity plane cubic $E$ is irreducible.
\end{lemma}
\par
\begin{proof} Suppose that $E$ consists of three lines. Then one of them
passes through $r_1$,$r_2$ and $r_3$ but none of the $p_i$. Hence the
other two lines intersect $X$ at $9$ points (counting with multiplicity), leading to a contradiction. 
\par
Next, suppose that $E$ consists of a line $L$ and an irreducible conic $Q$. 
By the same argument, $L$ has to contain precisely one of the $r_i$, 
say $r_1$. We now discuss case by case.
\par 
In the case $\QQ(9,-1)$ it follows that
$L \cdot X = 3z_1 + r_1$ and $Q \cdot X = 6z_1 + r_2 + r_3$. 
Since $Q \sim 2L$ as divisor classes, we have 
$2r_1 \sim r_2 + r_3$ on $X$, which implies that $X$ is hyperelliptic,
contradicting the hypothesis.
\par 
In the case $\QQ(6,3,-1)$ we conclude that $L \cdot X$ is
one of the divisors
$$r_1 + 3z_1, \quad r_1 + 2z_1 + z_2, \quad r_1 + z_1 + 2z_2 \quad \text{or} \quad  r_1 + 3z_2.$$ 
In the first case, $Q \cdot X = r_2 + r_3 + 3z_1 + 3z_2.$ This 
implies $3z_2 + r_2 + r_3 - r_1 \sim r_1 + r_2 + r_3 + p$ as sections 
of $\cO_X(1)$. Then $3z_2 \sim 2r_1 + p$. Since $3z_2 - p$ is determined 
by $q$, there are  at most finitely many choices for such $r_1$. 
One can choose $L$ away from these $r_1$. In the 
second case, $Q \cdot X = r_2 + r_3 + 4z_1 + 2z_2$. Using $Q \sim 2L$, we 
conclude that $2r_1 \sim r_2 + r_3$, contradicting the non-hyperelliptic 
assumption. In the third case, we have $r_1 + z_1 + 2z_2 ~\sim r_1 + r_2 + r_3 + p$ as sections of $\cO_X(1)$. 
Hence $r_2 + r_3 \sim z_1 + 2z_2 - p$. Since $z_1 + 2z_2 - p$ is determined by $q$, 
for such $r_2$ and $r_3$ there are finitely many choices,
since otherwise $X$ would be hyperelliptic. 
In the last case, we have $r_1 + 3z_2 \sim r_1 + r_2 + r_3 + p$, 
hence $r_2 + r_3 \sim 3z_2 - p$ which restricts  $r_2$ and $r_3$ to 
finitely many choices. 
\par
In the case $\QQ(3,3,3,-1)$ the intersection $L \cdot X$ could be 
$$r_1 + 3z_1, \quad r_1 + 2z_1 + z_2 \quad \text{or}\quad  r_1 + z_1 + z_2 + z_3.$$ 
In the first case, $3z_1 \sim r_2 + r_3 + p$, hence the choices of 
$r_2$ and $r_3$ are limited to a finite number, since $X$ is not 
hyperelliptic. In the second case, $2z_1 + z_2 \sim r_2 + r_3 + p$
and the same argument applies. Finally, in the last case 
since $Q \cdot X = r_2 + r_3 + 2z_1 + 2z_2 + 2z_3$, we have 
$2r_1 \sim r_2 + r_3$ on $X$, impossible for $X$ being non-hyperelliptic.
\end{proof}
\par
Next two preparatory lemmas imply that the $z_i$ are contained in the smooth locus of $E$.
\par
\begin{lemma}
\label{le:g3cusp}
Let $E$ be an irreducible plane cuspidal cubic with $z_1$ as its cusp. 
If a plane quartic $X$ has intersection multiplicity 
$(X\cdot E)_{z_1} \geq 4$, then $X$ is singular at $z_1$. 
\end{lemma}
\par
\begin{proof}
Without loss of generality, let $y^2 - x^3 = 0$ be the defining equation of $E$ (in affine coordinates) and 
$z_1 = (0,0)$. Suppose that the quartic $X$ is defined by 
$$f(x,y) = \sum_{i+j\leq 4 } a_{ij}x^i y^j.$$ 
Then we have 
$$\dim_{\mathbb C} \mathbb C[[x,y]]/(y^2 - x^3, f(x,y))\geq 4. $$
We now use the rational parameterization of $E$ by setting $y = t^3$ and $x = t^2$. 
Then 
$$f (t) =  \sum_{i+j\leq 4} a_{ij}t^{2i+3j} $$
and
$$ \dim_{\mathbb C} \mathbb C[[t^2, t^3]]/ (f(t))\geq 4. $$
Since $X$ contains $z_1$, we know that $a_{00} = 0$. If $a_{10} \neq 0$, then 
$f(t)$ is proportional to $t^2(1+b_1t+b_2t^2 + \cdots)$. The vector space $\mathbb C[[t^2, t^3]]/ (f(t))$ can be 
generated by $1, t^2, t^3$, since $1+b_2t^2 + \cdots$ is 
invertible in $\mathbb C[[t^2, t^3]]$. Consequently, 
$\dim_{\mathbb C} \mathbb C[[t^2, t^3]]/ (f(t))\leq 3$, 
contradicting the assumption. Next, if $a_{01}\neq 0$, then $f(t)$ is proportional to $t^3(1+c_1t+c_2t^2+\cdots)$. By the same token, 
$\mathbb C[[t^2, t^3]]/ (f(t))$ can be generated by $1, t^2, t^4$, hence $\dim_{\mathbb C} \mathbb C[[t^2, t^3]]/ (f(t))\leq 3$, leading 
to a contradiction. We thus conclude that $f\in (x,y)^2$, hence $X$ is singular at $z_1 = (0,0)$. 
\end{proof}
\par
\begin{lemma}
\label{le:g3node}
Let $E$ be an irreducible plane rational nodal cubic with $z_1$ as its node. Suppose that a 
plane quartic $X$ has intersection multiplicity 
$(X\cdot E)_{z_1} \geq 4$ and that $X$ is smooth at $z_1$. Then $z_1$ is not a flex of $X$.  
\end{lemma}
\par
\begin{proof}
Without loss of generality, let $y^2 - x^2 - x^3 = 0$ be the equation of 
$E$ and $z_1 = (0,0)$. Then the two branches at the node have tangent lines 
$L^-: x - y = 0$ and $L^{+}: x+y=0$, respectively. Since $X$ is smooth at $z_1$, 
it intersects one branch, say the one tangent to $L^+$, transversality, and 
intersects the other with multiplicity $\geq 3$. We use a local rational parameterization 
of $E$ by setting $x = s(s+2)$ and $y = s(s+1)(s+2)$. Suppose 
$f(x,y) = \sum_{i+j\leq 4}a_{ij}x^iy^j$ is the defining equation of $X$. 
Then we have  
$$f(s) = \sum_{i+j\leq4}a_{ij}s^{i+j}(s+2)^{i+j}(s+1)^j, $$ 
and $\dim_{\mathbb C}\mathbb C[[s]] / (f(s)) \geq 3$. 
Writing out the coefficients we see that 
$$ a_{00} = 0, \quad  a_{10} + a_{01} = 0, \quad  a_{01} + 2(a_{20} + a_{11} + a_{02}) = 0. $$
Now suppose that $z_1$ is a flex of $X$. Then $L^{-}$ is the corresponding flex line. 
The condition $\dim_{\mathbb C}\mathbb C[[x,y]] / (f(x,y), x-y) \geq 3$ implies that 
$$a_{00} = 0,  \quad  a_{10} + a_{01} = 0, \quad a_{20} + a_{11} + a_{02} = 0. $$
Combining the above equations, we conclude that 
$$ a_{00} = a_{10} = a_{01} = 0, $$
hence $f \in (x,y)^2$ and $X$ is singular at $z_1$, contradicting the assumption. 
\end{proof}
\par
\begin{proof}[Proof of Proposition~\ref{prop:Eirred}]
It remains to show that the $z_i$ are located at non-singular points of $E$. 
In the case of $\QQ(9,-1)$ the possibility that $z_1$ is a singular 
point of $E_1$ can be ruled out by Lemmas~\ref{le:g3cusp}.
In the case of $\QQ(6,3,-1)$ the fact that the $6$-fold zero $z_1$ is a 
smooth point can also be verified by Lemmas~\ref{le:g3cusp} 
and \ref{le:g3node}. Smoothness at $z_2$ is clear because 
the line $\overline{z_1z_2}$ intersects $E$ at $z_2$ with multiplicity one. 
The case $\QQ(3,3,3,-1)$ is dealt with by the same argument, using
the line through $z_1$, $z_2$ and $z_3$ instead.
\end{proof}
\par
We conclude this section by showing that an exceptional component in genus three cannot be contained in the
hyperelliptic locus.
\par
\begin{proof}[Proof of Lemma~\ref{le:nohypcomponent}]
Suppose on the contrary that such a component exists. In Section~\ref{sec:g3construction} we 
added a line $L$ in $\PP^2$ to the given bicanonical divisor in order 
to work with the effective divisor $\LL(X,q)+L \cdot X$. Such a line corresponds 
to a conic in $\PP^5$ under the Veronese embedding.  Here we use the 
hyperelliptic assumption and add a line in  $\PP^5$ to get
an effective divisor. More precisely, we first start with $(X,q) \in \QQ(k_1,\ldots,k_n,-1)$ 
and $X$ hyperelliptic. As in Section~\ref{logsurface}, let $S$ be a cone over a rational normal quartic
in $\PP^5$, which is the image of the ruled surface $F_4$. On $F_4$, the class of $X$ is $2e+8f$ 
and $\omega_X\sim \OO_X(2f)$. Take a ruling  $f_p$ passing through the unique pole $p$ of $\divisor(q)$. 
Let $p' = (f_p\cdot X) - p$ be the conjugate of $p$ in $X$.
Let $D = \divisor(q) + p$ be a degree $9$ divisor. A bicanonical
divisor is a section of $\OO_{X}(e+4f)$, so $D$ is a section of $\OO_X(e+5f)$.
By the exact sequence 
$$ 0\to \OO_{F_4}(-e - 3f) \to \OO_{F_4}(e+5f) \to \OO_X(e+5f) \to 0, $$
we know that $D+p'$ is cut out by a unique rational curve $R$ of class $e + 5f$. The pair $(R,D+p')$ 
takes the role of $(E,\sum_{i=1}^n k_i z_i + r_1+r_2+r_3)$. Then we can mimic the dimension count
in Proposition~\ref{prop:g3Sred}.
\par
Suppose that we have shown that $R$ is irreducible. The dimension of the linear system containing $R$ 
equals $\dim \PP H^0(F_4, e+5f) = 7$. Hence 
the parameter space for $(R, D+p')$ has dimension $n+8$. From 
$$ 0\to \OO_{F_4}(e + 3f) \to \OO_{F_4}(2e+8f) \to \OO_R(2e+8f) \to 0 $$
we deduce that for fixed $(R,D+p')$ the space of $X$ with $X\cdot R = D + p'$  
has dimension $h^0(F_4, e + 3f) =4$. 
From the triple $(X, R, D+p' )$ we recover $\divisor(q) = X \cdot R - X \cdot f_{p'}$.
Altogether, since the automorphism group of $F_4$ has dimension $9$, the space
of the triples $(X, R, D+p')$ has dimension $n+3$, smaller than 
$\dim \PP\QQ(k_1,\ldots,k_n,-1) = n+4$.
\par
Finally, we deal with the case when $R$ is reducible. In this case $R$ consists of $a$ 
fibers along with a component of class $e + (5-a)f$. Since $e^2=-4$, we have $a = 5$ or
$a = 1$, for otherwise $R$ would not be effective. 
\par
In the first case we can disregard $e$, since $e\cdot X = 0$. Let $f$ be the
fiber containing the pole $p$. It contains also some zero $z_i$, different from $p$.
This implies that $z_i = p'$, hence $f \cdot X = 2z_i$ and $z_i$ is a Weierstrass 
point, contradicting that $p \neq z_i$.
\par
Now consider $a=1$. In the case $\QQ(9,-1)$ we have $f \cdot X = 2z_1$ by the
same argument, hence $10 z_1 \sim 9 z_1 + p'$ and $z_1 \sim p'$, leading to the same contradiction. In the
remaining cases $f \cdot X = 2z_1$ with $z_1$ a $6$-fold zero or  
$f \cdot X = z_1 + z_2$ runs into a similar contradiction. We thus may suppose
that $f \cdot X =2 z_i$ with $z_i$ a $3$-fold zero. In this case we count
the dimension as above. 
\par
The component of class $e + 4f$ moves in a linear system of dimension $5$. 
For the stratum $\QQ(6,3,-1)$ the choice of $z_1, z_2$ and $p$ gives $3$ more
parameters. Now the choice of $X$ accounts for $h^0(F_4, e + 4f) =6$
parameters, and quotienting by ${\rm Aut}(F_4)$ we obtain a parameter space
of dimension $5$, less than $\dim \PP\QQ(6,3,-1) = 6$. For the stratum $\QQ(3,3,3,-1)$
the parameter space has one more dimension stemming from the choice of $z_3$, 
but since the dimension of this stratum is one larger than $\dim \QQ(6,3,-1)$
we deduce the same contradiction.
\end{proof}

\subsection{Genus four} \label{sec:g4local}
We prove here the technical lemmas for the exceptional strata in genus four.
Although this may seem a tedious case distinction, it cannot be circumvented
in some form, since the hyperelliptic components appear from our perspective
for some type of reducible parity curves.
\par
\begin{proof}[Proof of Lemma~\ref{le:g4Esuffsmooth}]
The elliptic curve $E$ has class $(2,2)$ on $Q$.
For the stratum $\QQ(12)$ the curve $E$ is indeed irreducible and reduced. 
If it is not the case, suppose first that $E$ contains a conic $C_1$ of class $(1,1)$. The the other 
component of $E$ is also a conic $C_2$ of class $(1,1)$. Both $C_1$ and $C_2$ cut out $6z_1$ in $X$, hence 
$q$ is a global square, leading to a contradiction. The other possibility is that   
$E$ consists of a line $L$ of class $(1,0)$ union a curve $R$ of class $(1,2)$, 
we have $L\cdot X = 3z_1$ and $R\cdot X = 9z_1$. Then $h^0(X, 3z_1) = 2$, 
hence $3z_1$ and $z_1 + r + s$ both admit a $g^1_3$, where $r, s$ are points in 
$X$ not equal to $z_1$. Then $\omega_X\sim 4z_1 + r + s$ and $4z_1\sim 2r + 2s$. 
Since $h^0(X, 4z_1) = h^0(X, 3z_1) = 2$, we conclude that 
$p$ is a base point of $|4z_1|$. But $2r + 2s$ is a section of this linear
series, which contradicts that $r, s\neq z_1$. 
\par
For all the other strata we first suppose that $E$ is not
reduced. If $E = 2C$ with $C$ a conic of class $(1,1)$, 
then $q$ is a global square, contradicting the standing 
assumption. If $E = 2L + C$ with $L$ a line of class $(1,0)$ and $C$ a curve of class $(0,2)$, then 
$C$ has to be a union of two distinct lines $L_1, L_2$ of class $(0,1)$. Each $L_i$ cuts out a degree $3$ divisor in $X$. 
Note that $L_1\cdot X$ and $L_2\cdot X$ are disjoint, because they belong to two different rulings in the same $g^1_3$.
Moreover, if $L$ and $L_i$ meet $X$ at the same $z_j$, then one of them intersects $X$ transversely at $z_j$. Using these observations 
we can easily rule out $\QQ(9,3)$ and $\QQ(6,6)$. For $\QQ(6,3,3)$, the only possibility is that $L\cdot X = 3z_1$, $L_1\cdot X = 3z_2$ 
and $L_2\cdot X = 3z_3$. It implies that $\omega_X\sim 3z_1 + 3z_2$, 
hence $3z_1$ gives a $g^1_3$ and $3z_2 \sim 3z_3$ gives the other.  
This is possible, but we claim that then $h^0(X, 2z_1 + z_2 + z_3) = 1$. 
Otherwise by Riemann-Roch we have $h^0(X, z_1 + 2z_2 - z_3) = 1$, 
hence $z_1 + 2z_2$ also provides a $g^1_3$, leading to a contradiction.
For $\QQ(3,3,3,3)$, the only possibility is that $L\cdot X = z_1 + z_2 + z_3$ 
(up to relabeling the $z_i$). Then $(L_1 + L_2)\cdot X = z_1 + z_2 + z_3 + 3z_4$ 
and one of them, say $L_1$, has to contain two points among $z_1, z_2, z_3$, 
which contradicts that $L_1\neq L$.  
\par
Now assume that $E$ is reduced, but not sufficiently smooth and
moreover that $\dim H^0(X,\divisor(q)/3) =2$. Below we will seek a contradiction.  
The assumption implies the existence of a hyperplane $H$ such that 
$H \cdot E = \divisor(q)/3$. Suppose first that $E = L + R$ 
for some line $L$, say of class $(1,0)$, necessarily contained in $H$. 
Then $H \cdot Q = L + L'$ with $L'$ of class $(0,1)$. 
\par
If $R$ decomposes further, say it contains another line $L_2$ of class $(1,0)$
and a curve $C$ of class $(0,2)$. From $H \cdot E = \divisor(q)/3$ we 
deduce that $L \cdot L' = z_1$, $L_2 \cdot L' = z_2$ and $L \cdot C = z_3+z_4$
are the zeros of $q$ (for some choice of numbering). But then $X \cdot (L+L_2+C)$
has multiplicity at most one at $z_1$, contradicting that it is at least three. If $R$
contains a line $L_2$ of class $(0,1)$ the intersection point of $L$ and $L'$
provides again the contradiction.
\par
Suppose now that $R$ is irreducible. We let again $L \cdot L' = z_1$
and obtain the same contradiction as above unless $R$ passes through $z_1$.
Then, if $z_2$ denotes the second intersection point of $L$ and $R$, we
obtain $H \cdot E = (L+L') \cdot (L + R) = 3z_1 + z_2$. Consequently, we deal
with the stratum $\QQ(9,3)$. Since $L' \cdot E = 2z_1$, a section 
of one of the two $g^1_3$ corresponding to $L'$ must be $2z_1 + t$ for some $t$. 
Then $\omega_X \sim (L + L')\cdot X = (z_1 + 2z_2) + (2z_1 + t)$, since $L$ and $L'$ belong to 
different ruling classes due to $H \cdot Q$ being of class $(1,1)$. Doubling it,
we obtain $6z_1 + 4z_2 + 2t \sim 9z_1 + 3z_2$, hence $3z_1 \sim z_2 + 2t$ is also a 
$g^1_3$. But one checks that it can neither be equivalent to $z_1 + 2z_2$ nor 
to $2z_1 + t$. 
\par 
The second case is that $E$ consists of two irreducible conics $C_1$ and $C_2$. If
the supports of $X \cdot C_1$ and  $X \cdot C_2$ are disjoint, then $E$ is
sufficiently smooth. If these supports are equal, then $q$ is a global square.
These supports can consist of at most three points, since otherwise $C_1$
and $C_2$ have three points in common, contradicting the intersection degree.
The same argument applies to counting a common tangent, e.g.\ for $\QQ(6,3,3)$
and  $X \cdot C_1 = 4 z_1 + 2z_2$. If $X \cdot C_1 - X \cdot C_2$ is the difference of 
two effective divisors of degree two, then $X$ is hyperelliptic. Using these
arguments, only the following cases remain.
\par
Case $\QQ(9,3)$ with $X \cdot C_1 = 6z_1$ and  $X \cdot C_2 = 3z_1 + 3z_2$.
This implies $3z_1 \sim 3z_2$, hence the ruling $L$ tangent to $X$ at
$z_1$ is a flex line. But then $C_1 = 2L$, contradicting its irreducibility.
\par
Case $\QQ(6,6)$ with $X \cdot C_1 = 5z_1 + z_2$ and  $X \cdot C_2 = z_1 + 5z_2$.
The existence of a hyperplane $H$ with $H \cdot X = 2z_1 + 2z_2$ implies
$\omega_X \sim 5z_1 + z_2 \sim z_1 + 5z_2 \sim 2z_1 + 2z_2 + r + s$ for some $r,s$.
Hence $h^0(X,3z_1) = h^0(X, 3z_2) =2$. 
Consequently, we also have $h^0(X, 2z_1 + z_2) = h^0(X, z_1 + 2z_2) =2$. 
Since $X$ is not hyperelliptic, the
ruling through $z_1$ and $z_2$ is tangent both at $z_1$ and $z_2$, a contradiction.
\par
Case $\QQ(6,3,3)$ with $X \cdot C_1 = 5z_1 + z_2$ and  $X \cdot C_2 = z_1 + 2z_2 + 3 z_3$.
The existence of a hyperplane $H$ with $H \cdot X = 2z_1 + z_2 + z_3$ implies
$$\omega_X \sim 5z_1 + z_2 \sim z_1 + 2z_2 + 3 z_3 \sim 2z_1 + z_2 + z_3 + r + s$$ 
for some $r,s$. Since $X$ is not hyperelliptic, $h^0(\omega_X(-z_1-z_2)) \leq 2$, 
hence $z_2 + z_3 = r+s$, which leads to the contradiction $z_1 \sim z_2$.
\par
Case $\QQ(3,3,3,3)$ with $X \cdot C_1 = 3z_1 + 2z_2 + z_3$ and  $X \cdot C_2 = z_2 + 2z_3 + 3z_4$.
The existence of a hyperplane $H$ with $H \cdot X = z_1 + z_2 + z_3 + z_4$ implies
that $h^0(X, 2z_1 + z_2)  = h^0(X, z_1+z_2+z_3) = 2$.
This is a contradiction, since they both contain $z_1$ and $z_2$.
\end{proof}
\par
\begin{proof}[Proof of Lemma~\ref{le:noGKhypcomp}] This
follows from the two lemmas below, given that the strata are smooth.
\end{proof}
\par
\begin{lemma}
No component of a stratum in $\cE_4$ except the hyperelliptic components lies
entirely in the hyperelliptic locus.
\end{lemma}
\par
\begin{proof}
By Section~\ref{g=4: logsurface}, a hyperelliptic curve $X$ of genus four
lies in $S_{1,6}$, which is the image of the Hirzebruch surface 
$F_{5}$ in $\PP^8$. Let $R$ be a rational curve of class $e+6f$. Since $e\cdot X = 0$, we have $\OO_X(R)\sim \OO_X(6f)\sim 
\omega^{\otimes 2}_X$. By the exact sequence 
$$ 0\to \OO_{F_5}(-e-4f)\to \OO_{F_5}(e+6f)\to \OO_{X}(e+6f)\to 0, $$
there exists a unique section $R$ in $|e+6f|$ that cuts out $D=\divisor(q)$ in $X$. 
\par
The pair $(R, D)$ now takes the role of $(E, D)$ and we mimic 
the dimension count in Proposition~\ref{prop:g4parametersp}. Suppose that $R$ is irreducible.
The parameter space for $R$ has dimension equal to 
$\dim \PP H^0(F_5, e+6f) = 8$.
Hence the parameter space for $(R, D)$ has dimension $8+n$. From 
$$ 0\to \OO_{F_5}(e+4f)\to \OO_{F_5}(2e+10f)\to \OO_{R}(2e+10f)\to 0, $$
we deduce that for fixed $(R,D)$, the space of (necessarily hyperelliptic) genus
four curves $X$ with $X \cdot R = D$ has dimension
$h^0(F_5, e+4f) = 5$. From $(X, R, D)$ we recover 
$\divisor(q) = X \cdot R$. Since the automorphism group of $F_5$ is $10$-dimensional, 
the space of triples $(X, R, D)$ has dimension $n+3$, smaller than
$\dim \PP \QQ(k_1,\ldots,k_n) = n+5$.
\par
Finally, we have to treat the case when $R$ is reducible. Suppose $R$ consists
of $a$ fibers $f$ along with a component of class $e+(6-a)f$. Since $e^2 = -5$ we have
$a=6$ or $a=1$, since otherwise $R$ would not be effective. 
\par
We first deal with the case $a=1$ and perform a dimension count similar
to the above. The parameter space for the component of class $e+5f$ has dimension 
equal to $\dim \PP H^0(F_5,\cO_{F_5}(e+5f)) = 6$ and the parameter space
for $(R, D)$ has dimension $6+n$. The other component of $R$
of class $f$ intersects $X$ in $2z_i$ or in $z_i + z_j$. In any case, it
is determined up to finitely many choices by $R$ and $D$. Now
the space of genus four curves $X$ with $X \cdot R = D$ has dimension
$h^0(F_5, e+5f) = 7$. As above, since the automorphism group of $F_5$ is 
$10$-dimensional, the space of triples $(X, R, D)$ has 
dimension $n+3$, smaller than the dimension of the corresponding stratum. 
\par
Finally, we deal with the case $a=6$ to retrieve the
hyperelliptic components. Now $R$ is the divisor $e$ union several (possibly 
non-reduced) fibers.
\par
For $\QQ(12)$, the fiber has to have multiplicity $6$ and $z_1$ has to be a Weierstrass point.
But then $q$ would be a global square, contradicting the hypothesis.
\par
For $\QQ(9,3)$ both $k_i$ are odd, hence it is impossible that $9z_1 + 3z_3$
is cut out by fibers only. 
\par
For $\QQ(6,6)$ and $\divisor(q) = 6z_1 + 6z_2$, there are two possibilities. 
First $R$ may contain a $6$-fold fiber $f$ with $f \cdot X = z_1 + z_2$. But then 
the $z_i$ are Weierstrass points and $q$ is a global square, contradicting the hypothesis.
The other case is that $R$ contains two $3$-fold fibers $f_1$ and $f_2$ with $f_i \cdot X = 2z_i$. 
Then both $z_i$ are Weierstrass points. We recover in this way the hyperelliptic 
component of this stratum.
\par
For $\QQ(6,3,3)$, there is only one possibility, namely $R$ contains
$3f_1 + 3f_2$, where $f_1 \cdot X = 2z_1$ and $f_2 \cdot X = z_2 + z_3$.
 Hence $z_1$ is a
Weierstrass point and $z_2, z_3$ are conjugate. We thus  recover the hyperelliptic 
component. 
\par
For $\QQ(3,3,3,3)$  there is again only one possibility, namely that
$R$ contains $3f_1 + 3f_2$, with $f_1 \cdot X = z_1 + z_2$ and $f_2 \cdot X = z_3 + z_4$ 
for an appropriate ordering of the zeros. In this way we recover again
the hyperelliptic component  of this stratum.
\end{proof}
\par
\begin{lemma}
No component of a stratum in $\cE_4$ is contained entirely in 
the Gieseker-Petri locus but is not entirely contained in the hyperelliptic locus.
\end{lemma}
\par
\begin{proof}
By Section~\ref{g=4: logsurface} we work on the Hirzebruch surface $F_{2}$ 
with $R$ a rational curve of class 
$2e+4f$ and $X$ of class $3e+6f$. Since $e\cdot X = 0$, we have $\OO_X(R)\sim \OO_X(6f)\sim 
\omega^{\otimes 2}_X$. The argument is parallel to the hyperelliptic case. 
By the exact sequence 
$$ 0\to \OO_{F_2}(-e-2f)\to \OO_{F_2}(2e+4f)\to \OO_{X}(2e+4f)\to 0, $$
there exists a unique section $R$ in $|2e+4f|$ that cuts out 
$D=\divisor(q)$ in $X$. 
\par
Suppose that $R$ is irreducible. The parameter space for $R$ has dimension equal 
to $\dim \PP H^0(F_2, 2e+4f) = 8$.
Hence the parameter space for $(R, D)$ has dimension $8+n$. From 
$$ 0\to \OO_{F_2}(e+2f)\to \OO_{F_2}(3e+6f)\to \OO_{R}(3e+6f)\to 0, $$
we deduce that for fixed $(R,D)$, the space of (necessarily hyperelliptic) genus
four curves $X$ with $X \cdot R = D$ has dimension
$h^0(F_2, e+2f) = 2$. From $(X, R, D)$ we recover 
$\divisor(q) = X \cdot R$. Since the automorphism group of $F_2$ is $7$-dimensional,
the space of triples $(X, R, D)$ has dimension $n+3$, smaller than
$\dim \PP \QQ(k_1,\ldots,k_n) = n+5$.
\par
If $R$ is reducible, there are three possible decompositions: first, two components 
$R_1$ and $R_2$, both of class $e+2f$; second, two components of class $e$ together
with four fibers and third, $R_1 = e + bf$ with $b \geq 2$ together with a
component of class $e$ and $4-b$ fibers.
\par
The components $R_1$ and $R_2$ in the first case  cannot be identical, otherwise $q$ 
would be a global square. Consequently, $X \cdot R_i = 6$ and $R_1 \cdot R_2 = 2$. If the 
divisor $X \cdot R = D$ is supported away from the singular locus $R_{\rm sing}$ 
of $R$, the same dimension count as above goes through. If $D$ meets $R_{\rm sing}$ 
and moreover the components $R_1$ and $R_2$ 
intersect transversely at two nodes $p_1$ and $p_2$, at each node $p_i$ the curve $X$ 
can only be tangent to one of the two branches. This involves a choice of
two possibilities, but once we choose one of the two, i.e.\ once we know which 
branch is tangent to $X$, then the above fiber dimension count is 
still fine. More generally, if $R_1$ and $R_2$ are tangent at $p$ 
and if $D$ contains $p$ with multiplicity $n$, then suppose $(X \cdot R_i)_p = n_i$ 
with $n_1 + n_2 = n$ and $n_i >1$. Once we specify the pair $(n_1, n_2)$, which
again amounts to finitely many choices, the Cartier divisor $D$ is determined and a similar fiber dimension count goes through. 
\par
In the second case, if the four fibers are all distinct,  
$X \cdot R = D$ is located away from  $R_{\rm sing}$ and the dimension count
is the same. If the configuration is $2f_1 + f_2 + f_3$, then 
$X \cdot R = 2D_1 + D_2 + D_3$ with disjoint divisors $D_i$. The only possible 
stratum is $(6,3,3)$ with $D_i = 3z_i$. Since $2f_1$ is a double ruling,
the parameter space of $X$ for a given $(R,D)$ increases dimension by three
compared to the above count, but the base dimension drops by three, so the dimension argument is 
fine. If the configuration is $2f_1 + 2f_2$ or $4f_1$, then $q$ is a 
global square, impossible. Finally, if the configuration is 
$3f_1 + f_2$, we are in the stratum $\QQ(9,3)$. Then the parameter space of $X$ 
for a given $(R,D)$ increases dimension by six compared to the count in the 
irreducible case, but the base dimension drops by seven, so the dimension argument still goes through. 
\par 
In the third case, suppose first that $R$ consists of an irreducible component $R_1$ of 
class $e+3f$ together with a ruling $f_1$ besides $e$. Then $R_1 \cdot f_1 = 1$ and they 
intersect transversely. So we can apply the above argument by specifying $X$ tangent 
to $R_1$ or $f_1$, if $X \cdot R$ contains the node. Next, suppose that $R$ consists 
of $R_1$ of class $e+2f$ and two rulings $f_1, f_2$ besides $e$. First suppose $f_1$ and 
$f_2$ are distinct. If $X$ is disjoint from $R_{\rm sing}$, we are done. If $X$ intersects 
$R_{\rm sing}$, since $f_i \cdot R_1 = 1$, this is also fine. Now suppose 
$f_1 = f_2$ is a double ruling and $X \cdot R$ contains the node $p = f_1 \cdot R_1$. 
Write $(X \cdot R)_p = n$, $(X \cdot R_1)_p = m$ and we have $(X \cdot f_1)_p = k < 4$ with
$m+2k = n$. This amounts to a finite case distinction. Since $R_1$ and $f_1$ 
intersect transversely at $p$, we have $m = 1$ or $k = 1$. If $m =1$, 
the condition imposed to $X$ is $k = (n-1)/2$ instead of $n$ as the expected fiber codimension. But the locus of such  $(R, 2f_1)$ in the linear system 
$|\OO(2)|$ has dimension $4$, i.e.\ codimension $5$, which is higher than 
$n - k = k+1$ since $k < 4$. So the total dimension of the parameter space is not 
enough for being a component. If $k = 1$, then the number of conditions imposed 
to $X$ is $m = n - 2$. The parameter space for $X$ has dimension two larger than
in the irreducible case,  but the base dimension drops by more than that. 
\end{proof}

\section{Varying strata: examples} \label{sec:varying}

In this section we collect all data of half-translation surfaces indicating that
beyond the cases discussed in our main results, most of the other strata are varying. We
emphasize though, that we do not dispose of any proof that all strata
beyond a certain dimension are varying. Below we give lists of explicit
surfaces, discussing one stratum at a time. 
\par
Almost all the half-translation surfaces were calculated using a computer
program by Vincent Delecroix, built on its predecessor by
Anton Zorich. These programs were designed for square-tiled surfaces. We
thus list below the monodromy permutations of the canonical double cover of a 
half-translation surface we are interested in. The program gives the sum $L$ 
of Lyapunov exponents for the double cover, but using \eqref{eq:L+L-} we
can also compute the quantity $L^+$.
\par
In the table 'Index' refers to the index of the Veech group 
in ${\rm SL}_2(\ZZ)$.
\par
{\bf The exceptional strata.} We give examples completing
Theorems~\ref{thm:execg3NV} and~\ref{thm:execg4NV}, proving that certain
components of the exceptional strata are varying:
$$\begin{array}{|l|l||l|l|l|}
\hline
\text{\rm Stratum} \,& \text{\rm Monodromy} & \text{\rm Index} & L & L^+  \\
\hline
\QQ(3,3,3,-1)^\irr & r = (2,3)(4,5,6,7) &  200 & 3 & 
\frac{13}{10} \\
& \phantom{r = \ } (9,10,11,12)(14,15)&&& \\
& u =(1,2)(3,4,8,9,5,10) &&&\\
& \phantom{r = \ }  (6,11,7,13,12,14)(15,16)&&& \\
\QQ(3,3,3,-1)^\irr & r =  (3,4)(5,6,7,8)&  2350 & \frac{150}{47} & \frac{328}{235} \\
& \phantom{r = \ }(10,11,12,13)(15,16) &&& \\
& u =(1,2,3)(4,5,9,10,6,11) &&&\\
& \phantom{r = \ } (7,12,8,14,13,15)(16,17,18) &&& \\
\hline
\hline
\QQ(6,6)^\irr & r = (2,3)(4,5)(7,8) & 768 & \frac{103}{32} & 
\frac{19949}{12015} \\
& \phantom{r = \ }(10,11)(12,13)(15,16) &&& \\
& u =(1,2,4,6,5,10,12)&&&\\
& \phantom{r = \ } (3,7,9,8,11,14,15) &&& \\
\QQ(6,6)^\irr & r = (2,3,4,5)(8,9,10,11) & 48 & \frac{7}{2} & \frac74 \\
& u = (1,2,6,4)(3,7,8,12)&&&\\
& \phantom{r = \ } (5,13,10,15)(9,14,11,16) &&& \\
\hline
\QQ(6,3,3)^\irr & r = (2,3)(4,5)(6,7,8,9) & 9420 & \frac{5593}{1570} & 
\frac{1359}{785} \\
& \phantom{r = \ }(10,11)(12,13,14,15)(17,18) &&& \\
& u =(1,2,4,6,10,12,5)(7,15)&&&\\
& \phantom{r = \ } (3,14,18,11,16,17,8)(9,13)  &&& \\
\QQ(6,3,3)^\irr & r =  (2,3)(5,6,7)(8,9)(10,11) & 3480 & 
\frac{514}{145} & \frac{999}{580} \\
& \phantom{r = \ } (12,13)(15,16,17)&&& \\
& u = (1,2,4,5,8,10,12,6)&&&\\
& \phantom{r = \ }  (3,13,16,18,11,14,15,9)&&& \\
\hline
\QQ(3,3,3,3)^\irr & r = (2,3)(4,5)(6,7,8,9) & 3390 & \frac{2121}{565} & 
\frac{1004}{565} \\
& \phantom{r = \ }(11,12,13,14)(16,17)(18,19) &&& \\
& u =(1,2,4,6,10,11,3,14,7,5)&&&\\
& \phantom{r = \ } (8,15,13,18,17,20,19,12,9,16)  &&& \\
\QQ(3,3,3,3)^\irr & r = (2,3)(5,6)(7,8,9,10)  & 690 & 
\frac{441}{115} & \frac{209}{115} \\
& \phantom{r = \ }(11,12,13,14)(15,16)(17,18) &&& \\
& u = (1,2,4,5,7,3,9,6)(8,11)&&&\\
& \phantom{r = \ } (10,13)(12,15,14,17,19,16,20,18) &&& \\
\hline
\end{array}
$$
\par
\medskip
{\bf In genus one} we give a pair of examples justifying that
$\QQ(4,2,-1^6)$ is varying:
$$\begin{array}{|l|l||l|l|l|}
\hline
\text{\rm Stratum} \,& \text{\rm Monodromy} & \text{\rm Index} & L & L^+  \\
\hline
\QQ(4,2,-1^6) & r = (1,2,3,4,5,6)(7,8,9,10,11,12) & 36 & 
\frac{5}{3} & 0 \\
& u = (1,9,4,12)(2,8,5,11)(3,7,6,10) &&&\\
\QQ(4,2,-1^6) & r = (1,2,3,4,5,6,7,8)& 148  & 
\frac{13}{6} & \frac{1}{4} \\
& \phantom{r = \ }(9,10,11,12,13,14,15,16) &&& \\
& u = (1,9,6,10,4,14,8,12)&&&\\
& \phantom{r = \ } (2,16,5,13)(3,15,7,11) &&& \\
\hline
\end{array}
$$
\par
\medskip

{\bf In genus two} all the strata of dimension at most seven except
for those listed in Theorem~\ref{thm:nvg2} are varying. We give examples in
the two cases, with largest and smallest possible orders of zeros:
$$\begin{array}{|l|l||l|l|l|}
\hline
\text{\rm Stratum} \,& \text{\rm Monodromy} & \text{\rm Index} & L & L^+  \\
\hline
\QQ(8,-1^4) & r = (2,3)(4,5)(6,7) & 160 & 
\frac{11}{5} & \frac{3}{5} \\
& \phantom{r = \ }(9,10)(11,12)(13,14) &&& \\
& u = (1,2,4,6)(3,5,8,9) &&&\\
& \phantom{r = \ } (7,11,13)(10,12,14) &&& \\
\QQ(8,-1^4) & r = (2,3)(4,5)(6,7)& 105  & 
\frac{13}{5} & \frac{4}{5} \\
& \phantom{r = \ }(9,10)(11,12)(13,14) &&& \\
& u = (1,2,4,6,7,10,12)&&&\\
& \phantom{r = \ } (3,5,8,9,11,13,14) &&& \\
\hline
\QQ(2^2, 1^2, -1^2) & r =  (2,3)(6,7,8)(9,10,11)(13,14) & 192 & 3 & 
\frac{7}{6} \\
& \phantom{r = \ } &&& \\
& u =(1,2,4,3,5,6)(7,9)&&&\\
& \phantom{r = \ } (8,10)(11,12,13,15,14,16)  &&& \\
\QQ(2^2, 1^2, -1^2) & r = (2,3,4)(6,7,8)   & 522 & 
\frac{220}{87} & \frac{27}{29} \\
& \phantom{r = \ }(9,10,11)(12,13,14) &&& \\
& u = (1,2,5,6,9,12,4,7)&&&\\
& \phantom{r = \ } (3,8,11,15,13,16,10,14) &&& \\
\hline
\end{array}
$$
\par
\medskip
{\bf In genus three} all the strata of dimension at most eight except
for those listed in Theorem~\ref{thm:nvg3} are varying. We give examples in
the two cases, with largest and smallest possible orders of zeros: 
$$\begin{array}{|l|l||l|l|l|}
\hline
\text{\rm Stratum} \,& \text{\rm Monodromy} & \text{\rm Index} & L & L^+  \\
\hline
\QQ(11,-1^3) & r = (2,3)(4,5)(6,7)(9,10) & 9828 & 
\frac{6388}{2457} & \frac{173}{189} \\
& \phantom{r = \ } (11,12,13)(14,15,16)&&& \\
& u = (1,2,4,6)(3,5,8,9) &&&\\
& \phantom{r = \ } (7,11,13,15)(10,12,14,16) &&& \\
\QQ(11,-1^3) & r =(3,4)(5,6)(7,8)(11,12)& 78390  & 
\frac{101078}{39195} & \frac{2728}{3015} \\
& \phantom{r = \ }(13,14,15)(16,17,18) &&& \\
& u = (1,2,3,5,7)(4,6,9,10,11)&&&\\
& \phantom{r = \ } (8,13,15,17)(12,14,16,18) &&& \\
\hline
\QQ(2,2,2,2) & r = (2,3)(5,6)(9,10)(12,13) & 24 & \frac{5}{2} & 
\frac{5}{4} \\
& u =(1,2,4,5,7,3,8,9)&&&\\
& \phantom{r = \ } (6,11,12,14,10,15,13,16)  &&& \\
\QQ(2,2,2,2) & r =  (2,3,4)(6,7,8)  & 36 & 
\frac{8}{3} & \frac{4}{3} \\
& \phantom{r = \ }(9,10,11)(12,13,14) &&& \\
& u = (1,2,5,6,9,8,10,7)&&&\\
& \phantom{r = \ } (3,12,4,14,16,11,15,13) &&& \\
\hline
\end{array}
$$
\par
\medskip
{\bf In genus four} all the strata of dimension at most nine except
for those listed in Theorem~\ref{thm:nvg4} are varying. We give examples in
the two cases, with largest and smallest possible orders of zeros: 
$$\begin{array}{|l|l||l|l|l|}
\hline
\text{\rm Stratum} \,& \text{\rm Monodromy} & \text{\rm Index} & L & L^+  \\
\hline
\QQ(14,-1^2)^{{\rm nh}} & r = (2,3)(5,6)(7,8)(9,10,11) & 381280 & 
\frac{138559}{47660} & \frac{114729}{95320} \\
& \phantom{r = \ }(13,14)(16,17,18) &&& \\
& u = (1,2,4,5,7,9)(3,10) &&&\\
& \phantom{r = \ } (6,8,12,13,15,16)(14,18) &&& \\
\QQ(14,-1^2)^{{\rm nh}} & r = (3,4)(6,7)(8,9)(10,11,12)&  3454784 & 
\frac{ 626281}{215924} & \frac{518319}{431848} \\
& \phantom{r = \ }(14,15)(18,19,20) &&& \\
& u = (1,2,3,5,6,8,10)(4,11)&&&\\
& \phantom{r = \ } (7,9,13,14,16,17,18)(15,20) &&& \\
\hline
\QQ(4,4,4) & r = (2,3)(5,6)(8,9)& 6480 & \frac{139}{45} & 
\frac{139}{790} \\
& \phantom{r = \ }(10,11)(12,13)(16,17) &&& \\
& u =(1,2,3,4,5,7,8,10,9)&&&\\
& \phantom{r = \ } (6,12,14,11,15,16,17,18,13)  &&& \\
\QQ(4,4,4) & r = (2,3)(5,6)(8,9)  & 180 & 
\frac{44}{15} & \frac{22}{15} \\
& \phantom{r = \ } (10,11)(12,13)(16,17)&&& \\
& u = (1,2,4,5,7,8,10,3,9)&&&\\
& \phantom{r = \ } (6,12,14,11,15,16,18,13,17) &&& \\
\hline
\end{array}
$$
We remark that showing a varying stratum, say for dimension greater
nine in genus four, would heavily requires computer resources. For example, 
in dimension nine, the simplest pillow-case tiled surface
in the stratum $(12,1,-1)$ has $9$ tiles, whose Veech group
has index $292824$ in $\SL_2(\ZZ)$ and $L^+ = 793/581$. The next
example with $10$ tiles has a Veech group of index  $2635416$
in $\SL_2(\ZZ)$ and also $L^+ = 793/581 \approx 1.36488 $. However, this stratum
is varying, as an example with $11$ tiles, a 
Veech group of index $13187664$ and $L^+ = 851/623 \approx 1.3659$ shows.
\par
\medskip
{\bf In genus five} even the smallest stratum is varying: 
$$\begin{array}{|l|l||l|l|l|}
\hline
\text{\rm Stratum} \,& \text{\rm Monodromy} & \text{\rm Index} & L & L^+  \\
\hline
\QQ(16) & r = (2,3)(5,6)(7,8)(9,10,11) & 648810 & \frac{39898}{12015} & 
\frac{19949}{12015} \\
& \phantom{r = \ }(12,13)(16,17,18) &&& \\
& u = (1,2,4,5,7,9)(6,10)&&&\\
& \phantom{r = \ }(3,12,14,8,15,16)(13,18)  &&& \\
\QQ(16) & r =  2,3)(5,6)(7,8,9)(10,11)& 6480 & \frac{10}{3} & \frac53 \\
& \phantom{r = \ }(12,13,14)(16,17) &&& \\
& u = (1,2,4,5)(3,7,10,12,14)&&&\\
& \phantom{r = \ } (6,13,17,9,8)(11,15,16,18) &&& \\
\hline
\end{array}
$$
\end{appendix}
%
%
%
\bibliography{my}

\end{document}